\documentclass{amsart}

\newcommand{\Ci}{\mathscr{C}}
\newcommand{\Sone}{\mathbb{S}^1}
\newcommand{\Es}{\mathscr{S}}
\newcommand{\N}{\mathbb{N}}
\newcommand{\R}{\mathbb{R}}
\newcommand{\C}{\mathbb{C}}
\newcommand{\Z}{\mathbb{Z}}

\newcommand{\T}{\mathbb{T}}

\newcommand{\De}{\mathscr{D}}

\newcommand{\ve}{\epsilon}

\newcommand{\supp}{\operatorname{supp}}
\newcommand{\Arg}{\operatorname{Arg}}

\newcommand{\vertiii}[1]{{\left\vert\kern-0.25ex\left\vert\kern-0.25ex\left\vert #1 
    \right\vert\kern-0.25ex\right\vert\kern-0.25ex\right\vert}}

\newcommand{\esssup}{\operatorname*{ess\,sup}}
\newcommand{\essinf}{\operatorname*{ess\,inf}}

\newcommand{\Ker}{\operatorname{Ker}}

\newcommand{\Cor}{\operatorname{Cor}}

\usepackage[initials]{amsrefs}
\usepackage{amsmath,amsthm,amscd}
\usepackage{enumerate}
\usepackage[dvipdfm]{graphicx}

\usepackage{amssymb}
\usepackage{mathrsfs}
\usepackage{graphicx}

\usepackage{color}
\usepackage{mathtools}
\mathtoolsset{showonlyrefs=true}

\usepackage[T1]{fontenc}

\usepackage{amsaddr}

\theoremstyle{plain}
\newtheorem{thm}{Theorem}[section]
\newtheorem{cor}[thm]{Corollary}
\newtheorem{lem}[thm]{Lemma}
\newtheorem{prop}[thm]{Proposition}

\theoremstyle{definition}
\newtheorem{dfn}[thm]{Definition}

\theoremstyle{remark}

\newtheorem*{rmk}{Remark}
\newtheorem*{rmk1}{Remark 1}
\newtheorem*{rmk2}{Remark 2}

\numberwithin{equation}{section}

\begin{document}

\title[Random perturbations of partially expanding maps]{On the spectra of quenched random perturbations of partially expanding maps on the torus}

\date{\today}

\author{Yushi Nakano and Jens Wittsten}
\address{Graduate School of Human and Environmental Studies, 
Kyoto University,  
Yoshida Nihonmatsu-cho, Sakyo-ku,
Kyoto, 606-8501,
Japan}

\email{nakano.yushi.88m@st.kyoto-u.ac.jp, jensw@maths.lth.se}

\subjclass[2010]{37C30 (primary), 37D30, 37H99, 58J40 (secondary)}

\keywords{Random dynamical system, transfer operator,
partially expanding map, semiclassical analysis, decay rate of correlations}

\begin{abstract}
We consider quenched random perturbations of skew products of rotations on the 
unit circle over uniformly expanding maps on the unit circle. It is known
that if the skew product satisfies a certain condition (shown to be generic in the 
case of linear expanding maps), 
then the transfer operator of the skew product has a spectral gap. 
Using semiclassical analysis we show that the spectral gap is preserved under small random perturbations.
This implies exponential decay of quenched random correlation functions for 
smooth observables at small noise levels.
\end{abstract}

\maketitle

\section*{Introduction}\label{section:introduction}

Let $X$ be a compact smooth Riemannian manifold. 
Recall that a dynamical system $f:X\to X$ is said to be
mixing with respect to an invariant measure $\mu$ when $\mu(A\cap f^{-n}B)$
converges to $\mu(A) \mu(B)$ as time $n$ goes to infinity for any Borel sets $A$ and $B$.
This means that the events $A$ and $f^{-n}B$ are asymptotically independent, so mixing
indicates a certain amount of complexity of the dynamical system. 
For mixing dynamical systems, a fundamental question 
is how fast the correlation functions decay (see Section \ref{section:DR} for definitions).
In fact, if the correlation functions of a mixing system decay exponentially 
fast, then several other statistical properties of the dynamical system
also hold. For an extensive background on such matters we refer to 
Bonatti, D{\' \i}az and Viana~\cite{BDV04}*{Appendix E} and the references therein.

For a hyperbolic dynamical system $f$,
individual trajectories tend to have chaotic behavior. Statistical properties
of the system, such as exponential decay of correlations,
are therefore preferably obtained by instead studying how densities of points evolve
under a so-called transfer operator $M_f$ induced by $f$.
The typical approach to proving exponential decay of correlations
is through the construction of a functional space $\mathcal H$ adapted to the dynamics such that the transfer operator
$M_f:\mathcal H\to\mathcal H$ has a {\it spectral gap},
that is, there exists a disc of 
strictly smaller radius than the spectral radius of $M_f$ outside of which the spectrum of $M_f$ consists only of 
discrete eigenvalues of finite multiplicity. 
Essentially, exponential decay of correlations is equivalent to the existence of a gap in the spectrum between
the eigenvalue 1 and the second largest eigenvalue of $M_f$, counting multiplicities. (In the presence of noise,
spectral stability thus becomes a natural object of study.)
Early work in this area was done by 
Bowen, Ruelle and Sinai, 
see for example the newly revised edition of Bowen's book~\cite{Bowen08} for a historical account.
A celebrated construction of anisotropic Banach spaces
was later established by Blank, Keller and Liverani~\cite{BKL02}, 
whose approach has been further developed by several authors.
Recently, these techniques have been shown to also be applicable
to dynamical systems with a one dimensional
nonhyperbolic direction (such as hyperbolic flow), see
for example Liverani~\cite{Liverani04}, Butterley and Liverani~\cites{BuL07,Bul13} and Tsujii~\cites{Tsujii,TsujiiQC10};
see also Baladi and Liverani~\cite{BaL12}
for a historical account. As evidenced by the mentioned articles,
the spectral analysis becomes more delicate for systems with a
nonhyperbolic direction.

Let the two dimensional torus be denoted $\mathbb{T} ^2= \Sone \times \Sone$, where $\Sone=\mathbb{R} / \mathbb{Z}$.
Consider skew products on $\T^2$ of the form
\begin{equation}
(x,s)\mapsto
( E(x) , s+\tau(x)\text{ mod }1)
\end{equation}
where $E:\Sone\to\Sone$ is a hyperbolic system, and $\tau$ is a real valued function on $\Sone$.
Known as {\it compact group extensions}, these were studied by
Dolgopyat~\cite{Dolgopyat2002} who proved superpolynomial decay of correlations under Diophantine conditions
in the case when $E$ is an Anosov diffeomorphism.
Tsujii~\cite{Tsujii} considered the closely related model given by the semi-flow
obtained by suspending a uniformly expanding map $E$ under a ceiling function $\tau$,
and obtained a precise description of the spectra of 
the corresponding transfer operators by imposing a {\it transversality} condition on the dynamics.
Tsujii~\cite{Tsujii} showed that this condition is generic for linear maps $E$,
and that it fails precisely when the ceiling $\tau$ is cohomologous to a constant.
The corresponding smooth compact group extension was studied by Faure~\cite{Faure}
who introduced a similar condition on the dynamics, using
the terminology {\it partially captive} for such systems.
(These satisfy the transversality condition,
and the conditions are comparable when the expanding map $E$ is linear, see Section \ref{subsection:transversality} below.
In particular, if $f_0$ is partially captive then the function $\tau_0$ in \eqref{eq:unperturbedsystem}
cannot be cohomologous to a constant, see the remark following Definition \ref{def:randompartiallycaptive}.)
Here we mention the recent preprints by Butterley and Eslami~\cite{ButterleyEslami}
and Eslami~\cite{Eslami}, wherein the same dynamics is studied
under much weaker regularity assumptions, utilizing an extension of Tsujii's transversality condition.

The map studied by Faure~\cite{Faure}, which
throughout this article shall be referred to as our unperturbed dynamical system,
is the simplest model of a hyperbolic system with a central direction. It is defined as follows:
Let $g_0: \Sone \rightarrow  \Sone$ be a ${\Ci}^\infty$ diffeomorphism and let
$\tau_0: \Sone \rightarrow \mathbb{R}$ be a ${\Ci}^\infty$ function. 
Let $k\ge 2$ be a positive integer, and consider the skew product $f_0:\T^2\to\T^2$ of class $\Ci^\infty$
given by
\begin{equation}\label{eq:unperturbedsystem}
f_0:\binom{x}{s}\mapsto
\left( \! \! \begin{array}{cc} kg_0(x) & \!\! \mod 1\\ s+\frac{1}{2\pi}\tau_0(x) & \!\! \mod 1\end{array} \! \! \right)
\end{equation}
on the torus. The map
$E_0:x\mapsto kg_0(x)\!\mod 1$ is assumed to be an {\it expanding map} on $\Sone$ in the sense that
$\min_x E_0'(x)>1$, and we then say that $f_0$ is a {\it partially expanding map} on $\T^2$.
(Since the differential of $x\mapsto E_0(x)$ is linear and $T_x\Sone\simeq\R$,
it follows that $(dE_0)_x:T_x\Sone\to T_{E_0(x)}\Sone$ is a scalar,
which we denote by $E_0'(x)$.) 
Note that with this terminology, 
the derivative of an expanding map is always positive, and that $g_0$ is orientation-preserving
by assumption.

Faure~\cite{Faure}
establishes the existence of a spectral gap
through semiclassical analysis, which is an asymptotic theory in which
the Planck constant appearing in the Shr{\"o}dinger equation is regarded as a small
parameter $h>0$.
It is a fairly recent discovery that spectral properties of transfer operators of (partially) hyperbolic maps
are naturally studied within this framework, the ideas having appeared in Baladi and Tsujii~\cites{BaladiTsujii07,BaladiTsujii08} 
(see also Avila, Gou{\"e}zel and Tsujii~\cite{AvilaGouezelTsujii}), and formalized 
in a series of papers primarily by Faure, Roy and Sj{\"o}strand  
\cites{FR06,FaureRoySjo,FS11}.
This approach has been getting traction lately with contributions in this and related areas also by
Arnoldi~\cite{Arnoldi}, Arnoldi, Faure and Weich~\cite{ArnoldiFaureWeich},
Dyatlov and Zworski~\cite{DyatlovZworski}, Faure and Tsujii~\cites{FaureTsujii1,FaureTsujii2},
and Tsujii~\cite{TsujiiFBI12}, among others.
For partially expanding maps, the first two references are particularly relevant.
So far, the focus seems to have been on deterministic systems,
and one of our goals is to show that the semiclassical approach is also applicable in the
case of random perturbations.

To circumvent the lack of hyperbolicity of $f_0$ in the $s$ direction, Faure~\cite{Faure}
uses Fourier analysis in the $s$ direction to
decompose the transfer operator induced by $f_0$ 
into a collection of (weighted) transfer operators
of the expanding map $E_0:\Sone _x \to \Sone _x$, indexed by a Fourier parameter $\nu \in \Z$.
The resulting operators are examples of
Fourier integral operators, and thus naturally studied using
microlocal analysis (when $\nu\in\Z$ is fixed) and semiclassical analysis
(with a semiclassical parameter of size $h\sim 1/\lvert\nu\rvert$, tending to 0).
Roughly speaking, if $f_0$ is partially captive, then the spectral 
radius decreases in the {\it semiclassical limit} $\lvert\nu\rvert\to\infty$.
On the other hand, outside a small disc,
the spectrum of each transfer operator (for fixed $\nu\in\Z$) 
consists of discrete eigenvalues of finite multiplicity (the so-called {\it Ruelle resonances}),
resulting in a spectral gap for the collection.
This (and an additional assumption on the peripheral spectrum)
is known to give exponential decay of operational correlations for smooth observables (Faure~\cite{Faure}*{Theorem 5}).

In this paper we show that the presence of the spectral gap observed in the deterministic case 
(as described above) is preserved under quenched random perturbations
at small noise levels, see Theorems \ref{thm:discretespectrum}
and \ref{thm:spectralgap}. 
For random transfer operators,
the notion of spectrum needs clarification; in particular, 
the notion of discrete spectrum should be understood in terms of Lyapunov exponents
and invariant subspaces instead of eigenvalues and eigenfunctions, see Section \ref{section:DR}.
We also show existence and strong stability of random measures, see Theorem \ref{thm:inv}. Using
the spectral results we then establish our main theorem: 
if $f_0$ is partially captive then the
quenched random correlations for $\Ci^\infty$ observables decay exponentially fast, see Theorem \ref{thm:expdecay}.

The rest of the paper is organized as follows: 
Section \ref{section:spectralgap}
is devoted to the proof of Theorem
\ref{thm:spectralgap}. A significant complication
compared to the deterministic case is that it is no longer sufficient
to only study the leading term in the pseudodifferential symbolic calculus;
this would for example result in a decrease of the spectral 
radius in the semiclassical limit $\lvert\nu\rvert\to\infty$
which holds only {\it pointwise} with respect to the noise parameter. The drawbacks
of this would in turn be quite severe,
see the remark at the end of Section \ref{section:spectralgap}.
The more detailed symbolic calculus that we shall require has been
collected in Appendix \ref{app:PsiDO}. A key ingredient in the proof of Theorem
\ref{thm:spectralgap} is a careful analysis of the partial captivity condition in a suitably adapted random setting (resulting in Proposition \ref{wPC}),
which is postponed until Section \ref{section:analysisofpc}. The cornerstone of
this analysis is the crucial perturbation result
Proposition \ref{prop:differentnoiselevels}.
The tools developed also allow for a comparison between partial captivity and
transversality in $\S$3.2, where we prove that
the partial captivity condition is generic when $E_0$ is linear,
see Theorem \ref{thm:generic}.
In Section \ref{section:discrete} we prove Theorem \ref{thm:discretespectrum}.
The proof includes a description of the peripheral spectrum of the reduced transfer operators
when $f_0$ is partially captive. This has implications also for the unperturbed dynamics
studied by Faure~\cite{Faure}, see Theorem \ref{thm:nonrandomperspec}.

\section{A randomly perturbed partially expanding map}\label{section:DR}

\subsection{Random dynamical systems}

Let $(\Omega,\mathcal F,\mathbb{P})$ be a probability space with probability measure $\mathbb{P}$.
Let $\theta :\Omega \rightarrow \Omega$ be a measure-preserving 
transformation. Let $X$ be a measurable space. We say that a measurable 
mapping $\Phi :\N \times \Omega \times X \to X$ is a {\it random dynamical system} 
(abbreviated RDS henceforth) on $X$ over $\theta$ when $\Phi$ 
satisfies the cocycle property $\Phi (0,\omega) 
=\mathrm{Id} _X$ for all $\omega \in\Omega$ and
\begin{equation*}
\Phi (n+m,\omega )=\Phi(n,\theta ^m\omega)\circ  \Phi (m,\omega ), \quad n,m\in \N, \quad \omega \in \Omega.
\end{equation*}
Here $\theta \omega$ denotes the value $\theta (\omega)$, and $\Phi(n, \omega)
=\Phi (n,\omega ,\cdot)$. 
For general properties of an RDS, we refer to Arnold~\cite{Arnold}.

For normed vector spaces $X$ and $Y$ we let $\mathscr L(X,Y)$ denote the space of bounded linear operators
from $X$ into $Y$ endowed with the operator norm $\lVert\phantom{i}\rVert_{\mathscr L(X,Y)}$. When $X=Y$ we simply write
$\mathscr L(X)=\mathscr L(X,X)$. We say that an RDS $\Phi $ on a separable Banach space $X$ 
over $\theta :\Omega \to \Omega$ 
is a linear RDS when $\Phi (n,\omega)=\Phi(n,\omega ,\cdot) \in \mathscr L(X)$.
Recall also that an operator $A:\Omega \rightarrow\mathscr L(X)$
is said to be strongly measurable (continuous) when $\Omega \ni \omega \mapsto A(\omega) \varphi $
is measurable (continuous) for all $\varphi \in X$.

We borrow some terminology from Gonz{\'a}lez-Tokman and Quas~\cite{ETS:8859243}.
Let $X$ be a separable Banach space. Recall that
the {\it index of compactness} 
for $T\in \mathscr L(X)$ is defined by
\[
\lVert T\rVert _{\mathrm{ic}(X)} = \inf{\{r>0 : T(B_1) \text{ can be covered by finitely many balls of radius }r\}},
\]
where $B_1$ denotes the unit ball of $X$.
Note that $\lVert T \rVert _{\mathrm{ic}(X)} 
\leq \lVert T\rVert_{\mathscr L(X)}$ for each $T\in \mathscr L(X)$ since $ T(B_1)$ can be 
covered by a ball of radius $ \lVert T \rVert_{\mathscr L(X)}$.
Moreover, the index of compactness is subadditive,
\[
\lVert T_1+T_2\rVert _{\mathrm{ic}(X)}\le\lVert T_1\rVert _{\mathrm{ic}(X)}
+\lVert T_2\rVert _{\mathrm{ic}(X)},\quad T_1,T_2\in\mathscr L(X),
\]
and if $T\in \mathscr L(X)$ is a compact operator then $\lVert T\rVert _{\mathrm{ic}(X)} = 0$.
This is a consequence of the properties of the underlying (Hausdorff) measure of noncompactness,
see Chapter 2, Proposition  2.3 in Ayerbe Toledano, Dom\'inguez Benavides and L{\'o}pez Acedo~\cite{TBA97}.

\begin{dfn}\label{def:lyapunovexponents}
Let $\Phi$ be a linear RDS on a separable Banach space $X$ over
a probability-preserving measurable map $\theta :\Omega \to \Omega$ such that the 
time-one map
$\Phi (1,\cdot):\Omega \to \mathscr L (X)$ is strongly measurable, that is, 
$\Omega \ni \omega \mapsto \Phi (1,\omega)\varphi$ is measurable for all 
$\varphi \in X$. Assume also that $\omega\mapsto\log ^+\lVert \Phi (1,\omega) \rVert \in L^1(\Omega,\mathbb{P})$, 
where $\log ^+x =\max{(\log x,0)}$ for $x>0$. 
For each $\omega\in\Omega$, the {\it maximal Lyapunov exponent} for $\omega$ is defined as
\begin{equation}\label{eq:r}
r(\omega)=\lim _{n\to \infty}\frac{1}{n} \log{\lVert \Phi (n,\omega) \rVert_{\mathscr L(X)}} 
\end{equation}
whenever the limit exists. For each $\omega\in\Omega$,
the {\it index of compactness} for $\omega$ is defined as
\begin{equation}\label{eq:rbeta}
 r_{\mathrm{ic}}(\omega)=\lim _{n\to \infty}\frac{1}{n} \log{\lVert \Phi (n,\omega ) \rVert _{\mathrm{ic}(X)}}
\end{equation}
whenever the limit exists.
\end{dfn}

If $\Phi$ satisfies the conditions of Definition \ref{def:lyapunovexponents},
and $\theta:\Omega\to\Omega$ in addition is assumed to be ergodic,
then $r(\omega)$ and $r_{\mathrm{ic}}(\omega)$ exist and are $\mathbb{P}$-almost surely constants 
as functions of $\omega$, see Lemma 2.4, Lemma 2.5 and Remark 2.6 in Gonz{\'a}lez-Tokman and Quas~\cite{ETS:8859243}.
Denoting these constants by $r^\ast$ and $r_{\mathrm{ic}}^\ast$, 
respectively, we then say that $r^\ast$ is the maximal Lyapunov exponent of
$\Phi$ and $r_\mathrm{ic}^\ast$ the index of compactness of $\Phi$.
In view of the definition of the 
index of compactness for bounded operators
we have $r_\mathrm{ic}^\ast\le r^\ast$. When $r_\mathrm{ic}^\ast< r^\ast$,
the linear RDS $\Phi$ is said to be {\it quasi-compact}. 
We mention that $r^\ast$ is the analog of (the logarithm of) the spectral radius of a deterministic system,
while $r_\mathrm{ic}^\ast$ is the analog of (the logarithm of) the essential spectral radius.
Note also that the assumption
$\omega\mapsto\log ^+\lVert \Phi (1,\omega) \rVert \in L^1(\Omega,\mathbb{P})$
implies that $r^\ast<\infty$.

It will be convenient to have available
a multiplicative ergodic theorem extended by Gonz{\'a}lez-Tokman and Quas~\cite{ETS:8859243} 
to cocycles of so-called `semi-invertible' linear operators such as transfer operators,
see Theorem A and Theorem 2.10 in the mentioned article.
Due to conflicting notation, we restate (a slight reformulation of)
their result here for easy reference. Let $\Phi$ be a strongly measurable
linear RDS on a separable Banach space $X$ over $\theta :\Omega \to \Omega$ 
and $\Phi(1,\cdot):\Omega \to \mathscr L(X)$ its time-one map $\Phi(1,\omega)=\Phi (1,\omega,\cdot)$.
We assume that $\theta$ is ergodic, so the numbers $r(\omega)$ and $r_\mathrm{ic}(\omega)$ 
given by \eqref{eq:r} and \eqref{eq:rbeta} exist and are $\mathbb{P}$-almost surely equal
to constants $r^\ast$ and $r_\mathrm{ic}^\ast$, respectively.

\begin{thm}[Gonz\'alez-Tokman \& Quas]\label{thm:cocycle}
Let $\Phi(1,\cdot):\Omega \to \mathscr L(X)$ be strongly measurable such that
$
\omega\mapsto \log ^+ \lVert \Phi(1,\omega)\rVert \in L^1(\Omega,\mathbb{P}).
$
Assume that 
$\Phi$ is quasi-compact in the sense that $r _\mathrm{ic}^\ast<r^\ast$, 
where the maximal Lyapunov exponent $r^\ast$ and the
index of compactness $r_\mathrm{ic}^\ast$ $\mathbb{P}$-almost surely coincides with
the expressions \eqref{eq:r} and \eqref{eq:rbeta}, respectively.
Then there is a number $1\leq \ell\leq \infty$,  sequences $r^\ast=\alpha _1>\alpha _2 >\cdots >\alpha _\ell >r_\mathrm{ic}^\ast$
$($or in the case $\ell=\infty$, $r^\ast=\alpha _1>\alpha _2 >\cdots $ with $\lim _{n\to \infty}\alpha _n =r_\mathrm{ic}^\ast)$,
$m_1,\ldots ,m_\ell\in \N$ and a $\theta$-invariant subset $\tilde{\Omega }\subset \Omega$ 
of full measure, and  a unique measurable splitting of $X$ into closed subspaces
$
X=\bigoplus ^\ell_{j=1}\Sigma_j(\omega)\oplus  \Sigma_- (\omega)
$
such that   for each
$\omega \in \tilde{\Omega}$, the following holds:
\begin{enumerate}
\item[$(1)$] For each $1\leq j\leq \ell$, $
\dim{\Sigma_j(\omega)}=m_j<\infty$ and
\[
\Phi(1,\omega)\Sigma _j(\omega) =\Sigma _j  (\theta \omega),\quad \Phi(1,\omega)\Sigma _- (\omega) \subset \Sigma _- (\theta \omega).
\]
\item[$(2)$] If $v \in \Sigma_j(\omega)\backslash \{ 0\}$ for some $1\leq j\leq \ell$, then
\[
\lim _{n\rightarrow \infty} \frac{1}{n} \log{\lVert \Phi (n,\omega) v\rVert} =\alpha _j.
\]
\item[$(3)$] If $v\in \Sigma _-(\omega)$,
\[
\lim _{n\rightarrow \infty} \frac{1}{n}\log{\lVert \Phi (n,\omega) v \rVert} \leq r_\mathrm{ic}^\ast.
\]
\end{enumerate}
\end{thm}
We say that the splitting in Theorem \ref{thm:cocycle} is the  Oseledets splitting of $\Phi$. 
The numbers $\alpha_1,\ldots,\alpha_\ell$
are said to be the exceptional Lyapunov exponents of $\Phi$, while 
$m_j$ is the multiplicity and $\Sigma_j(\omega)$ the associated Lyapunov subspace 
of $\alpha_j$, $1\le j\le\ell$.

\subsection{The perturbation model}

Let $(\Omega, \mathcal{F})$ be a Lebesgue space with a probability measure $\mathbb{P}$.
Let $\theta :\Omega \rightarrow \Omega$ be an ergodic $\mathbb{P}$-preserving 
bi-measurable bijection.
Let ${\Ci}^{\infty}(\mathbb{T} ^2,\mathbb{T} ^2)$ be the space of smooth  endomorphisms on $\mathbb{T} ^2$ endowed with a ${\Ci}^{\infty}$ metric,
\begin{equation*}
d_{{\Ci}^{\infty}}(f,g)=\sum _{j=0}^{\infty} 2^{-j}\frac{d_{{\Ci}^j}( f,g)}{1+d_{{\Ci}^j}( f,g)},
\end{equation*}
where $d_{{\Ci}^j}(f,g )$ is the usual ${\Ci}^j$ distance between $f$ and $g$. We endow ${\Ci}^{\infty}(\T ^2,\T ^2)$ with the Borel $\sigma$-algebra.

Let $\{ f_\epsilon \} _{\ve >0}$ be a family of measurable mappings
$f_\epsilon :\Omega \rightarrow {\Ci}^{\infty} (\mathbb{T} ^2,\mathbb{T} ^2)$ such that 
for each $\epsilon >0$, $f_{\epsilon}(\omega)$ is for $\mathbb{P}$-almost every $\omega\in\Omega$ of the form  
\begin{equation}\label{eq:perturbedsystem}
f _{\epsilon}(\omega) :\binom{x}{s}\mapsto
\left( \! \! \begin{array}{cc} kg_{\epsilon}(\omega,x) & \!\! \mod 1 \\ s+\frac{1}{2\pi}\tau _{\epsilon}(\omega,x)  
& \!\! \mod 1
\end{array} \! \! \right),
\end{equation}
where $\omega \mapsto f_{\epsilon}(\omega)(z)$ is a measurable mapping from $\Omega$ to $\T ^2$ for each $z=(x,s)\in \T ^2$.
Here $g_{\epsilon}(\omega)=g_{\epsilon}(\omega ,\cdot ):\mathbb{S} ^1 \to \mathbb{S} ^1$ 
is a ${\Ci}^{\infty}$ diffeomorphism 
and  $\tau _{\epsilon} (\omega)=\tau _{\epsilon} (\omega ,\cdot ) :\mathbb{S} ^1 \to \mathbb{R}$ is a ${\Ci}^{\infty}$ function,
$\mathbb{P}$-almost surely. 
We also assume that  
\begin{equation}\label{convinC1}
\esssup_\omega d_{{\Ci}^{\infty}}(f_{\ve }(\omega) ,f_0 ) \rightarrow 0\quad
\text{as }\ve \rightarrow 0,
\end{equation}
where $f_0$ is the partially expanding map given by \eqref{eq:unperturbedsystem}. 
The value $f_{\epsilon}(\omega)(z)$ is denoted simply by $f_{\epsilon}(\omega, z)$. 
For each $\epsilon >0$, it follows that
$(\omega ,z) \mapsto f_{\epsilon}(\omega ,z)$ 
is a measurable mapping from $\Omega \times \T ^2$ to $\T ^2$,
see Castaing and Valadier~\cite{CV77}*{Lemma 3.14}.
When convenient, we will identify $f_0:\T^2\to\T^2$ with the constant map $\Omega\ni\omega\mapsto f_0$.

For each $\epsilon \ge0$ and $\omega \in \Omega$
we let $E_{\epsilon }(\omega)$ denote the map $E_{\epsilon}(\omega):x\mapsto kg_{\epsilon} (\omega ,x)\mod 1$,
interpreted for $\epsilon=0$ to mean $E_{\epsilon=0}(\omega)\equiv E_0$ for all $\omega$.
The value $E_{\epsilon }(\omega)(x)$ is denoted simply by $E_{\epsilon }(\omega,x)$. 
In view of \eqref{convinC1} it then follows that  
$E_\epsilon (\omega)$ is an expanding map
$\mathbb{P}$-almost surely if $\ve$ is sufficiently small. In fact, if $\lambda_0=\min_x E_0' (x)$
and we set $\lambda=(\lambda_0+1)/2$, then $\lambda>1$ and we can find an
$\epsilon_0>0$ such that 
\begin{equation}\label{emin}
 \essinf_{\omega} \min_x \frac{d E_\epsilon (\omega,x)}{dx}\ge\lambda,\quad 0\le\epsilon<\epsilon_0.
\end{equation}
In the sequel, the quantity $d E_\epsilon (\omega,x)/dx$ will sometimes be denoted simply by $E_\epsilon'(\omega,x)$.

\begin{rmk}\label{rmk:tauneighborhood}
When there is no ambiguity, the noise
level $\epsilon$ will sometimes be omitted from the notation, in particular
when the dependence on the noise parameter $\omega\in\Omega$ is already
displayed. In fact, with the exception of the underlying probability space
$(\Omega,\mathcal F,\mathbb{P})$ and the map $\theta:\Omega\to\Omega$, dependence
on the noise parameter $\omega\in\Omega$ will always be taken to imply dependence on
the noise level $\epsilon$.
Throughout the rest of the paper we will also permit us
to use $\epsilon_0$ as a way to denote the upper bound of a range $0\le\epsilon<\epsilon_0$ for which
\eqref{emin} holds, even if $\epsilon_0$ may change between occurrences. This will mostly be showcased
only in the statements of our results; we shall in fact always assume that $\epsilon$
belongs to such a range.
\end{rmk}

Given $\epsilon >0$ and $n\geq 1$, let $f_{\epsilon}^{(n)}(\omega,z)$ be the 
fiber component of the $n$th iteration of the (double) skew product mapping
\[
\Theta _{\epsilon} (\omega , z) =(\theta \omega ,f_{\epsilon}(\omega ,z)),\quad \omega\in\Omega,\quad z\in \T^2,
\]
and let $f_{\epsilon}^{(0)}(\omega)=\mathrm{Id} _{\T^2}$ for all $\omega \in \Omega$.
With the notation $f_{\epsilon}^{(n)}(\omega)=f_{\epsilon}^{(n)}(\omega,\cdot)$ we explicitly have
\begin{equation}\label{fibercomponent2}
f_{\epsilon}^{(n)}(\omega )=f_{\epsilon}(\theta^{n-1}\omega)\circ f_{\epsilon}(\theta^{n-2}\omega)\circ\cdots\circ f_{\epsilon}(\omega).
\end{equation}
The mapping given by $(n,\omega ,z)\mapsto f_{\epsilon}^{(n)}(\omega ,z)$ is an 
RDS on $\T^2$ over $\theta :\Omega \to \Omega$, which we call the RDS induced 
by $f_{\epsilon}$. (Naturally, this RDS depends also on $\theta$, but 
since $\theta$ will be fixed throughout, mention of this map will be omitted.)
For convenience we introduce the notation 
\begin{align}\label{notation:E^{(n)}}
E_\epsilon^{(n)}(\omega,x)&=E_\epsilon(\theta^{n-1}\omega)\circ\ldots\circ E_\epsilon(\omega)(x),\quad n\ge 1,\\
\tau_\epsilon^{(n)}(\omega,x)&=\sum_{j=0}^{n-1} 
\tau_\epsilon(\theta^j\omega, E_\epsilon^{(j)}(\omega,x)),\quad n\ge 1.
\label{notation:tau^{(n)}}
\end{align}
In the last sum, $\theta^0$ and $E_\epsilon^{(0)}(\omega,\cdot)$ are to be interpreted as the identity maps
on $\Omega$ and $\Sone$, respectively, so that $E_\epsilon^{(1)}(\omega ,\cdot)=E_\epsilon(\omega)$ and
$\tau_\epsilon^{(1)}(\omega ,\cdot)=\tau_\epsilon(\omega)$. Then
\[
f_{\epsilon}^{(n)}(\omega):\binom{x}{s}\mapsto
\left( \! \! \begin{array}{cc} E_\epsilon^{(n)}(\omega,x)\\ s+\frac{1}{2\pi}\tau_\epsilon^{(n)}(\omega,x) \mod 1
\end{array} \! \! \right),
\quad n\ge1.
\]

The Perron-Frobenius transfer operator
$M_{f_\epsilon ^{(n)}(\omega)}^\ast:{\Ci}^\infty(\mathbb{T}^2)\rightarrow {\Ci}^\infty(\mathbb{T}^2)$
corresponding to $f_\epsilon^{(n)}(\omega)$ is defined
as the random operator cocycle
\begin{equation}\label{PFtransfer}
M^\ast_{f_\epsilon ^{(n)}(\omega)}\psi (z)=\sum_{f_\epsilon ^{(n)}(\omega,z')=z} 
\frac{\psi (z')}{\lvert \det \partial f_\epsilon^{(n)}(\omega,z')/\partial z
 \rvert },\quad \psi\in {\Ci}^\infty(\T^2),
\end{equation}
where $\partial f_\epsilon^{(n)}(\omega,z')/\partial z$ is the Jacobian matrix of $z\mapsto f_\epsilon^{(n)}(\omega,z)$ at $z'\in\T^2$, and
${\Ci}^{\infty}(\mathbb{T} ^2)$ is the space of complex valued functions on $\mathbb{T} ^2$ of class ${\Ci}^{\infty}$.
Note that by \eqref{eq:perturbedsystem} and \eqref{emin} we $\mathbb{P}$-almost surely have that
\begin{equation}\label{eq:det}
\det{(\partial f_\epsilon^{(n)}(\omega,z)/\partial z)}
=d E_\epsilon^{(n)}(\omega,x)/dx\ge\lambda^n,\quad z=(x,s).
\end{equation}
Thus, 
the operator $M_{f_\epsilon ^{(n)}(\omega)}^\ast$ extends to a bounded operator on $L^2(\T^2)$,
$\mathbb{P}$-almost surely.
The extension will also be denoted by $M_{f_\epsilon ^{(n)}(\omega)}^\ast$.
Its adjoint $M_{f_\epsilon ^{(n)}(\omega)}$ with respect to the usual scalar product on 
$L^2(\T^2)$ is the Ruelle transfer operator 
given by $M_{f_\epsilon^{(n)}(\omega)} \psi (z) = \psi (f_\epsilon ^{(n)}(\omega,z))$.

\subsection{The reduced transfer operator}\label{subsection:reduction}
For the time being, we shall fix $\epsilon$ and suppress it from the notation.
As in the treatment of the unperturbed case \eqref{eq:unperturbedsystem} undertaken by Faure~\cite{Faure}, 
we now employ the following decomposition in Fourier modes,
\begin{equation}\label{orthogonaldecomp}
L^2(\T^2)=\bigoplus _{\nu \in \mathbb{Z}} \mathcal{H} _{\nu} ,
\quad \mathcal{H} _{\nu}=\{ (x,s)\mapsto\varphi (x)e^{2i\pi \nu s} : \varphi \in L^2(\Sone) \}.
\end{equation}
The spaces $(\mathcal H_\nu,\lVert\phantom{i}\rVert_{L^2(\T^2)})$ and $L^2(\Sone)$ are isometrically isomorphic. 
For given $\nu\in\Z$ and fixed $\omega\in\Omega$, let $M_{\nu,n}(\omega)$ denote
the operator $M_{f^{(n)}(\omega)}$ restricted to $\mathcal H_\nu$. It is straightforward to check that for 
$\psi(x,s)=\varphi(x)e^{2i\pi\nu s}$ with $\varphi\in L^2(\Sone)$, we have $\mathbb{P}$-almost surely that
\[
M_{f^{(n)}(\omega)}\psi(x,s)=\varphi(E^{(n)}(\omega,x))e^{i\nu\tau^{(n)}(\omega,x)}e^{2i\pi \nu s},
\quad n\ge 1,
\]
where $E^{(n)}(\omega,x)$ and $\tau^{(n)}(\omega,x)$ are given by \eqref{notation:E^{(n)}}--\eqref{notation:tau^{(n)}}.
Thus, by identifying $\mathcal H_\nu$
with $L^2(\Sone)$ we can view $M_{\nu ,n}(\omega)$ for $\mathbb{P}$-almost every $\omega$ as an operator on $L^2(\Sone)$ given by
\begin{equation}\label{restrictionoperator}
M_{\nu ,n}(\omega)\varphi(x)=\varphi(E^{(n)}(\omega,x))e^{i\nu\tau^{(n)}(\omega,x)}, 
\quad\varphi\in L^2(\Sone).
\end{equation}

Similarly, for given $\nu\in\Z$ and $\omega\in\Omega$, let $M_{\nu,n}^\ast(\omega)$ denote
$M^\ast_{f^{(n)}(\omega)}$ restricted to $\mathcal H_\nu$. 
For $\psi(x,s)=\varphi(x)e^{2i\pi\nu s}$ with $\varphi\in L^2(\Sone)$, we have $\mathbb{P}$-almost surely that
\[
M^\ast_{f^{(n)}(\omega)}\psi(x,s)=\sum_{f^{(n)}(\omega,(x',s'))
=(x,s)} \frac{\varphi (x')e^{2i\pi\nu s'}}{\lvert \det D_{x,s}f^{(n)}(\omega,(x',s')) \rvert }
\]
in view of \eqref{PFtransfer}.
If $f^{(n)}(\omega,(x',s'))=(x,s)$ then $s'=s-\frac{1}{2\pi}\tau ^{(n)}(\omega,x')\! \mod 1$,
so the decomposition \eqref{orthogonaldecomp} is preserved by virtue of \eqref{eq:det}.
Furthermore, for fixed $s$ we have that $s'$ is uniquely determined by $x'$,
so by identifying $\mathcal H_\nu$
with $L^2(\Sone)$ we can view $M_{\nu,n}^\ast(\omega)$ for $\mathbb{P}$-almost every $\omega$ as an operator on $L^2(\Sone)$ given by
\begin{equation}\label{adjointrestrictionoperator}
M_{\nu ,n}^\ast(\omega)\varphi(x)=\sum_{E^{(n)}(\omega,y)=x}\frac{e^{-i\nu\tau ^{(n)}(\omega,y)}}{dE^{(n)}(\omega,y)/dy}\varphi(y),
\quad\varphi\in L^2(\Sone).
\end{equation}
By duality, $M_{\nu ,n}^\ast(\omega)$ coincides with the
$L^2$ adjoint of the operator $M_{\nu,n}(\omega)$ defined by \eqref{restrictionoperator}.
For convenience, we introduce $M_{\nu ,0}(\omega)=M_{\nu ,0}^\ast (\omega) 
=\mathrm{Id} _{L^2(\Sone)}$ for each $\omega \in \Omega$ and $\nu \in \Z$.
It is straightforward to check that $(n,\omega ,\varphi) \mapsto M^\ast_{\nu ,n}(\omega)\varphi$ 
satisfies the cocycle property 
$M^\ast _{\nu ,n+m}(\omega) \varphi=M^\ast _{\nu ,n}(\theta ^m\omega) 
M_{\nu ,m}(\omega) \varphi $ for each $\varphi \in L^2(\Sone)$, 
$\omega \in \Omega$ and $n,m\in \N$, so
$(n,\omega ,\varphi) \mapsto M^\ast_{\nu ,n}(\omega)\varphi$ is an RDS.

\begin{rmk}
For fixed $\omega\in\Omega$ and $n\in\N$, we shall have reason to view $M_{\nu,n}^\ast(\omega)$ both as
a collection of operators indexed by $\nu\in\Z$, as well as {\it one}
operator depending on a parameter $\nu\to\pm\infty$. More generally,
we shall let $h$ be a small semiclassical parameter and study
$M_{\pm 1/h,n}^\ast(\omega)$ in the {\it semiclassical limit} as $h\to 0$. The connection to $M_{\nu,n}^\ast(\omega)$
is obtained by setting $h=1/\lvert\nu\rvert$. When referring to this construction
we will say that $\nu\in\Z$ is viewed as a semiclassical parameter.
\end{rmk}

Next, we recall some basic facts concerning distribution theory and Sobolev spaces on $\Sone$. 
We note that the operator $M_{\nu,n}(\omega):L^2(\Sone)\to L^2(\Sone)$
defined by \eqref{restrictionoperator} has a continuous
extension to $\De'(\Sone)$, expressed through duality
by
\[
\langle M_{\nu,n}(\omega) u,\bar\varphi\rangle=\langle u, \overline{M_{\nu ,n}^\ast(\omega)\varphi}\rangle,
\quad u\in\De'(\Sone),\quad\varphi\in {\Ci}^\infty(\Sone),
\]
where $\langle\phantom{i},\phantom{i}\rangle$ is the distributional pairing and
$M_{\nu ,n}^\ast(\omega)$ is the $L^2$ adjoint of $M_{\nu ,n}(\omega)$, which explains the appearances of the complex conjugations.
For $m\in\R$ we let $H^m(\Sone)\subset\De'(\Sone)$ denote
the Sobolev spaces of index $m$, defined as the set of distributions $u$ satisfying
\[
\lVert u\rVert_{H^m(\Sone)}^2=\lVert u\rVert_{(m)}^2=\sum_{\xi\in2\pi\Z}(1+\lvert\xi\rvert^2)^m\lvert\hat u(\xi)\rvert^2<\infty.
\]
The Fourier coefficients of a distribution $u\in\De'(\Sone)$ are defined by 
\[
\hat{u}(\xi)=\langle u,\phi_{-\xi}\rangle,\quad \xi\in2\pi\Z,
\]
where the functions $\phi_\xi\in {\Ci}^\infty(\Sone)$ are given by $\phi_\xi(x)= e^{ix\xi}$
for $\xi\in2\pi\Z$ and $x\in \Sone$.
It is well-known that the Fourier series 
$\sum\hat{u}(\xi)\phi_\xi$ converges to $u$ in $\De'(\Sone)$ in the usual  
sense, thus
\[
\langle u,\varphi\rangle=\sum_{\xi\in2\pi\Z}\hat{u}(\xi)\hat{\varphi}(-\xi),
\quad \varphi\in {\Ci}^\infty(\Sone). 
\]

Let $\langle\xi\rangle=(1+\lvert\xi\rvert^2)^{\frac{1}{2}}$, $\xi\in\R$, and
for $m\in\R$, let $\langle D\rangle^m:\De'(\Sone)\to\De'(\Sone)$ denote the operator which acts through multiplication by 
$\langle\xi\rangle^m$ on the Fourier side,
\begin{equation}\label{Fouriermultiplier}
\langle \langle D\rangle^m u,\varphi\rangle=\sum_{\xi\in2\pi\Z}\langle\xi\rangle^m\hat{u}(\xi)\hat{\varphi}(-\xi),
\quad \varphi\in {\Ci}^\infty(\Sone). 
\end{equation}
Here $D$ is understood to be the differential operator $D_x=-i\partial_x$.
Then $\langle D\rangle^m$ is a pseudodifferential operator with symbol $(x,\xi)\mapsto\langle\xi\rangle^m$ (see Appendix \ref{app:PsiDO}
for a discussion on pseudodifferential operators on $\Sone$). Moreover,
\begin{align*}
\lVert \langle D\rangle^m u\rVert_{(s-m)}^2&=\sum_{\xi\in2\pi\Z}(1+\lvert\xi\rvert^2)^{s-m}\lvert(\langle D\rangle^m u)\hat{\ }(\xi)\rvert^2\\
&=\sum_{\xi\in2\pi\Z}(1+\lvert\xi\rvert^2)^{s}\lvert\hat u(\xi)\rvert^2=\lVert u\rVert_{(s)}^2,\quad u\in H^s(\Sone),
\end{align*}
so $\langle D\rangle^m:H^s\to H^{s-m}$ defines an isomorphism for all $s\in\R$.
Since $H^0(\Sone)=L^2(\Sone)$ by Parceval's formula, and
$\langle D\rangle^m$ is invertible with inverse $(\langle D\rangle^m)^{-1}=\langle D\rangle^{-m}$, this gives the standard
identification of the Sobolev space $H^m(\Sone)$ with $\langle D\rangle^{-m}(L^2(\Sone))$ for each $m\in\R$.
As usual, we identify the dual space of $H^m(\Sone)$ with $H^{-m}(\Sone)$,
expressing the action of $u\in H^{-m}(\Sone)$ on $H^m(\Sone)$ as $u(v)=(u,v)_{L^2}$ for $v\in H^m(\Sone)$,
where $(\phantom{i},\phantom{i})_{L^2}$ is the usual scalar product on $L^2(\Sone)$.

The symbol $(x,\xi)\mapsto \langle\xi\rangle^m$ also defines a semiclassical operator $\langle hD\rangle^m$ acting
through multiplication by $\langle h\xi\rangle^m$ on the Fourier side,
\begin{equation}\label{eq:toroidalsymbols}
\langle \langle hD\rangle^m u,\varphi\rangle=\sum_{\xi\in2\pi\Z}\langle h\xi\rangle^m\hat{u}(\xi)\hat{\varphi}(-\xi),
\quad \varphi\in {\Ci}^\infty(\Sone). 
\end{equation}
When $\nu\in\Z$ is viewed as a semiclassical parameter 
this gives rise to an alternative $\lvert\nu\rvert$-dependent norm $\lVert\phantom{i}\rVert_{H_\nu^m}$
on $H^m(\Sone)$ defined
by
\begin{equation}\label{eq:nu_dependentnorm}
\lVert u\rVert_{H_\nu^m}^2=\sum_{\xi\in2\pi\Z}\lvert\langle \xi/\nu\rangle^m\hat u(\xi)\rvert^2,\quad u\in H^m(\Sone),\quad 0\ne\nu\in\Z.
\end{equation}
Let $H_\nu^m(\Sone)$ denote the space $H^m(\Sone)$ equipped with the norm $\lVert\phantom{i}\rVert_{H_\nu^m}$.
If $h=1/\nu$, then $\langle hD\rangle^m:H_\nu^s(\Sone)\to H_\nu^{s-m}(\Sone)$ is
an isomorphism for all $s\in\R$. In particular, we can identify the space $H_\nu^m(\Sone)$
with $\langle h D\rangle^{-m}(L^2(\Sone))$ for each $m\in\R$, and $H_\nu^m(\Sone)$ is the
dual of $H_\nu^{-m}(\Sone)$ with respect to the usual scalar product on $L^2(\Sone)$.

Before turning to our results on (what in the deterministic case corresponds
to) the spectrum of the operator $M_{\nu ,n}^\ast(\omega):H^m(\Sone)\to H^m(\Sone)$,
we first establish the existence of an absolutely continuous $f_\epsilon$-invariant probability measure.

\subsection{Invariant measures}\label{subsection:invmeas}

Let $f:\Omega \to {\Ci}^{\infty}(\T^2,\T ^2)$ be a measurable mapping. 
Let $\mathcal{B} (\T ^2)$ be the Borel $\sigma$-field of $\T^2$. 
Recall that a measure $\mu$ on $\Omega \times \T^2$ is called $f$-invariant  when $\mu$ is invariant with respect 
to the skew product mapping $\Theta (\omega,z) =(\theta \omega,  f(\omega, z))$ and the marginal 
$\pi _{\Omega} \mu$ of $\mu$ coincides with $\mathbb{P}$, where $\pi _{\Omega} :\Omega \times \T ^2\rightarrow \Omega$
is given by $\pi _{\Omega}(\omega ,z)=\omega$. It is known that when $\mu$ is an $f$-invariant probability measure, 
there is a unique function $\mu _{\cdot}(\cdot):\Omega \times \mathcal{B} (\T^2) \to [0,1]$, 
$(\omega,B)\mapsto \mu_{\omega}(B)$, 
such that
\begin{itemize}
\item[$(\mathrm{i})$] $\omega \mapsto \mu _{\omega}(B)$ is measurable for each $B\in \mathcal{B} (\T^2)$,
\item[$(\mathrm{ii})$] $\mu _{\omega}$ is $\mathbb{P}$-almost surely a probability measure on $\T^2$,
\item[$(\mathrm{iii})$] $\int u d\mu =\int u d\mu _{\omega} d\mathbb{P}$ 
for each $u \in L^1(\mu)$.
\end{itemize} 
Moreover, since we assume that $\theta$ is measurably invertible, the pushforward 
$f(\omega)_\ast\mu_{\omega}$ of $\mu_{\omega}$ by $f(\omega)$ $\mathbb{P}$-almost surely coincides with 
$\mu _{\theta \omega}$, see Arnold~\cite{Arnold}*{Chapter 1}. We call the function $\mu _{\cdot} (\cdot)$ 
the disintegration of $\mu$ (with respect to $\mathbb{P}$).

For the perturbation scheme $f_{\epsilon}$ given by \eqref{eq:perturbedsystem}, the existence of an
absolutely continuous invariant probability measure on $\Omega\times\T^2$
is an immediate consequence
of established results concerning the existence
of such measures for uniformly expanding maps.

\begin{thm}\label{thm:inv}
For each $0\le\epsilon <\epsilon _0$, 
there exists an $f_{\epsilon}$-invariant probability measure $\mu ^{\epsilon}$ on $\Omega\times\T ^2$
such that if $\mu _{\cdot}^\epsilon(\cdot)$ is the disintegration of $\mu ^\epsilon$
then $\mu_\omega^\epsilon$ is $\mathbb{P}$-almost surely equivalent to normalized Lebesgue measure on $\T^2$,
and $d\mu_\omega ^\epsilon=h_\epsilon(\omega,x)dxds$.
Each density $h_\epsilon(\omega)$ is the uniquely defined positive function in ${\Ci}^\infty(\Sone)$
such that $\int_{\Sone} h_\epsilon(\omega,x)dx=1$ and
$M_{0,n}^\ast(\epsilon;\omega)h_\epsilon(\omega)=h_\epsilon(\theta^n\omega)$ for $n\ge1$.
Moreover,
\begin{equation}\label{eq:stability}
\esssup_\omega{\lVert h_\epsilon(\omega)-h_0\rVert_{(m)}}\to0\quad\text{as }\epsilon\to0
\end{equation}
for all $m\in\N$, where $h_0\equiv h_{\epsilon=0}(\omega)$ is independent of $\omega$.
\end{thm}

\begin{proof}
Consider $M_{0,n}^\ast(\omega)=M_{0,n}^\ast(\epsilon;\omega)$. For $0\le\epsilon<\epsilon_0$
it follows that $M_{0,1}^\ast(\omega)$ is the transfer operator of a
uniformly expanding map $E_\epsilon(\omega)$, $\mathbb{P}$-almost surely, for which the existence
of invariant measures is well-known. In fact, by Kifer~\cite{Kifer08}*{Theorem 2.2}
there is a unique $E_\epsilon$-invariant ergodic probability measure $\nu^\epsilon$
whose disintegration $\nu_\omega^\epsilon(dx)=h_\epsilon(\omega,x)dx$
is $\mathbb{P}$-almost surely equivalent to normalized Lebesgue measure $dx$ on $\Sone$.
Here $h_\epsilon=h_\epsilon(\omega)=h_\epsilon(\omega,x)$ is a positive function which is measurable
in $(\omega,x)$ and H{\"o}lder continuous in $x$ such that
$M_{0,1}^\ast(\omega)h_\epsilon(\omega)(x)=h_\epsilon(\theta\omega,x)$ and
$\int_{\Sone} h_\epsilon(\omega,x)dx=1$. It follows that
$M_{0,n}^\ast(\epsilon;\omega)h_\epsilon(\omega)=h_\epsilon(\theta^n\omega)$ for $n\ge1$.
Since $E_\epsilon(\omega)\in{\Ci}^\infty$ we find
in view of Baladi, Kondah and Schmitt~\cite{BKS96}*{Theorem A} that also $h_\epsilon(\omega)\in{\Ci}^\infty$
and that $h_\epsilon(\omega)$ is uniquely defined by the properties above.
The same result also immediately gives \eqref{eq:stability}
since the Sobolev norm on the left can be controlled by the 
$\Ci^m$ norm of $h_\epsilon(\omega)-h_0$ when $m\in\N$.

For $\mathbb{P}$-almost every $\omega\in\Omega$,
we now define a probability measure $\mu_\omega^\epsilon(dxds)=h_\epsilon(\omega,x)dxds$ on $\T^2$ equivalent to Lebesgue measure,
where each density $h_\epsilon(\omega)$ has the desired properties.
It is straightforward to check that the pushforward
of $\mu_\omega^\epsilon$ by $f_\epsilon(\omega)$ satisfies
$f_\epsilon(\omega)_\ast \mu_\omega^\epsilon=\mu_{\theta\omega}^\epsilon$.
In fact, if $h_\epsilon(\omega)\otimes 1_{\Sone}$ denotes the tensor product
$(x,s)\mapsto h_\epsilon(\omega,x)$ for $(x,s)\in\T^2$, then the
Perron-Frobenius transfer operator
$M_{f_\epsilon(\omega)}^\ast$ given by \eqref{PFtransfer}
satisfies
\begin{equation}\label{eq:eigenvalue1}
M_{f_\epsilon(\omega)}^\ast h_\epsilon(\omega)\otimes 1_{\Sone}=
(M_{0,1}^\ast(\omega)h_\epsilon(\omega))\otimes 1_{\Sone}=h_\epsilon(\theta\omega)\otimes 1_{\Sone}
\end{equation}
since the decomposition \eqref{orthogonaldecomp} is preserved.
If we let $\mu_\cdot^\epsilon(\cdot)$ be the disintegration of a measure $\mu^\epsilon$
with respect to $\mathbb{P}$,
then $\mu^\epsilon$ is an $f_\epsilon$-invariant probability measure.
This completes the proof.
\end{proof}

\subsection{Spectral properties}
We now return to the operator $M_{\nu ,n}^\ast(\epsilon;\omega):H^m(\Sone)\to H^m(\Sone)$.
To shorten the notation we shall (as in Section \ref{subsection:reduction}) usually write $M_{\nu ,n}^\ast(\omega)$
instead of $M_{\nu ,n}^\ast(\epsilon;\omega)$ when $\epsilon$ is fixed.
The first result shows that for fixed $\nu\in\Z$, the system
\begin{equation}\label{eq:linearRDS}
\N\times\Omega\times H^m(\Sone)\ni(n,\omega ,\varphi) \mapsto M_{\nu ,n}^\ast(\omega)\varphi
\end{equation}
is a linear RDS on $H^m(\Sone)$ satisfying the conditions in Definition \ref{def:lyapunovexponents}.
To indicate the dependence on the parameter $\nu$,
the corresponding maximal Lyapunov exponent and index of compactness will (for fixed $\epsilon$) be denoted by
$r^\ast\{\nu\}$ and $r_\mathrm{ic}^\ast\{\nu\}$, respectively.
The dependence on the Sobolev index $m$ will usually be clear from
context, so this should cause no confusion. In what follows, we will denote the time-one map
$M_{\nu ,1}^\ast :\Omega\to\mathscr L(H^m(\Sone))$ by 
$M_{\nu}^\ast$, that is, $M_{\nu}^\ast(\omega)=M_{\nu ,1}^\ast(\omega)$.
For convenience, we will sometimes say that the linear RDS
given by \eqref{eq:linearRDS}
is the linear RDS on $H^m(\Sone)$ induced by the time-one map $M_{\nu}^\ast:\Omega\to\mathscr L(H^m(\Sone))$.
Similarly, we will write $M_{\nu}(\omega)=M_{\nu ,1}(\omega)$
for the adjoint of $M_{\nu}^\ast(\omega)$.

Let $m>0$ and set
\begin{equation}\label{eq:rm}
r_m=\log {(\lambda ^{-m-\frac{1}{2}}k^{\frac{1}{2}} )}.
\end{equation}
Then $r_m$ does not depend on $n$, $\nu$ or $\omega$, and $r_m \rightarrow -\infty$ as $m \rightarrow \infty$.

\begin{thm}\label{thm:discretespectrum}
Let $\{f_{\epsilon}\}_{\epsilon >0}$ be a family of ${\Ci}^{\infty}(\T^2,\T^2)$ 
valued random variables satisfying \eqref{eq:perturbedsystem} and
\eqref{convinC1}.
For $\nu\in\Z$, let $M_{\nu,n}^\ast(\omega)=M_{\nu,n}^\ast(\epsilon;\omega)$ be the operator defined
by means of the decomposition \eqref{adjointrestrictionoperator}. 
Let $m$ be a positive integer and define $r_m$ by \eqref{eq:rm}. 
Then there is an $\epsilon_0(m)$, depending on $m$ but
independent of $\nu$, such that if $0\le\epsilon<\epsilon_0$ then
$(n,\omega ,\varphi) \mapsto M_{\nu ,n}^\ast(\omega)\varphi$
is a linear RDS on the Sobolev space $H^m(\Sone)$
with maximal Lyapunov exponent $r^\ast\{\nu\}\le 0$ and
index of compactness $r_\mathrm{ic}^\ast\{\nu\}\le r_m$.
For $\nu=0$ we have $r^\ast\{0\}=0$ and the linear RDS is quasi-compact
if $m$ is sufficiently large, and
\begin{itemize}
\item[$(\mathrm{i})$] the Lyapunov subspace associated to
$r^\ast\{0\}=0$ is one dimensional and spanned by the function $h_\epsilon(\omega)$
given by Theorem \ref{thm:inv}.
\end{itemize}
If, in addition, $f_0$ is partially captive (see Definition \ref{def:randompartiallycaptive}),
then
\begin{itemize}
\item[$(\mathrm{ii})$] for every $\nu\ne0$ there is an $\epsilon_\nu(m)$ such that
if $0\le\epsilon<\epsilon_\nu$ then $r^\ast\{\nu\}< 0$.
\end{itemize}
\end{thm}

If $(\Omega,\mathcal F,\mathbb{P})$ is the trivial probability space consisting of one point and $\theta$
is the identity on $\Omega$, then $M_{\nu,n}(\omega)=(M_\nu(\omega))^n$.
Theorem \ref{thm:discretespectrum} then means that $M_{\nu}^\ast(\omega)$ has 
spectral radius $\le1$ with 
discrete (possibly empty) spectrum outside the circle of radius $e^{r_m}$,
consisting of eigenvalues of finite multiplicity, 
compare with Faure~\cite{Faure}*{Theorem 2}. 
The eigenvalues are called {\it Ruelle resonances}, and are intrinsic to $M_\nu^\ast(\omega)$
and do not depend on the Sobolev index $m$. 
For $\nu=0$, the Lyapunov exponents $\alpha_j$ and associated subspaces $\Sigma_j(\epsilon;\omega)$
of the linear RDS on $H^m(\Sone)$ induced by $M_0^\ast(\omega)$
(in the notation of Theorem \ref{thm:cocycle})
are related to the eigenvalues of $M_0^\ast(\omega)$
in the following way: If $\gamma$ is an eigenvalue lying outside the circle of radius
$e^{r_m}$ then $\log{\lvert\gamma\rvert}=\alpha_j$ for some $j$.
$\Sigma_j(\epsilon;\omega)$ is the direct sum of the generalized eigenspaces for all such eigenvalues.
The reader is asked to compare with Bogensch{\"u}tz~\cite{Bogenschutz00}*{Example 1.4}.
In a similar fashion,
the next result concerns the `spectrum' of $M_{\nu ,n}^\ast(\omega):H^m(\Sone)\to H^m(\Sone)$
in the semiclassical limit $\nu\to\pm\infty$. Together with Theorem \ref{thm:discretespectrum}
it shows that the presence of a spectral gap
which can be seen for the unperturbed system (Theorems 2 and 3 in Faure~\cite{Faure}) is preserved under small 
random perturbations.

\begin{thm}\label{thm:spectralgap}
Let $\{ f_{\epsilon}\}_{\epsilon >0}$ be a family of ${\Ci}^{\infty}(\T^2,\T^2)$ 
valued random variables satisfying \eqref{eq:perturbedsystem} and
\eqref{convinC1}, and assume that $f_0$ is partially captive, see Definition \ref{def:randompartiallycaptive}. 
Then for any $\rho >\lambda ^{-\frac{1}{2}}$,
there is a positive integer $m_0$, together with numbers $\epsilon_0(m)$, $\nu_0(m)$ and $c_0(m)$ 
depending also on the integer $m\ge m_0$, 
such that for all $m\ge m_0$, $\lvert\nu\rvert\ge\nu_0$ and $0\le\epsilon<\epsilon_0$
we have 
\[
\lVert M_{\nu ,n}^\ast(\omega) \rVert_{\mathscr L(H_\nu^{m}(\Sone))}
\leq c_0 \rho ^n,
\quad n\ge1,
\]
$\mathbb{P}$-almost surely.
\end{thm}

Note that Theorem \ref{thm:spectralgap} implies that if
$0\le\epsilon<\epsilon_0$ then $r^\ast\{\nu\}< 0$ for all
$\lvert\nu\rvert\ge\nu_0$.

\begin{rmk}
The terminology {\it partially captive} is due to Faure~\cite{Faure}*{Definition 15}.
The definition was inspired by a condition of {\it transversality} introduced by Tsujii~\cite{Tsujii}, who showed that
transversality is a generic condition for suspensions of linear expanding maps, and equivalent
to the condition that the ceiling function $\tau_0$ not be cohomologous to a constant.
(See Theorems 1.2 and 1.4 in Tsujii~\cite{Tsujii}
for precise statements.) The transversality condition
has recently been introduced also for partially expanding maps of the form \eqref{eq:unperturbedsystem}
by Butterley and Eslami \cite{ButterleyEslami} (see also Eslami~\cite{Eslami}), who showed that this equivalence is still in force.
In $\S$3.2 we compare the conditions partial captivity and transversality in the special case when
the expanding map $E_0$ is linear, and establish that
partial captivity is a generic condition. The reader is asked to
compare with the remark on p.~1479 in Faure~\cite{Faure}.
\end{rmk}

In $\S$1.6 the results of Theorems \ref{thm:discretespectrum} and \ref{thm:spectralgap} will be used to obtain information
concerning the long-term behavior of the RDS induced by $f_\epsilon$.
The proofs of these theorems occupy sections \ref{section:discrete}
and \ref{section:spectralgap}, respectively.
Since we will now almost exclusively deal with function spaces defined
on $\Sone$, we will henceforth sometimes omit $\Sone$ from the notation,
and for example write $H^m$ instead of $H^m(\Sone)$ when there is
no ambiguity.

\subsection{Decay of random correlation functions}\label{subsection:decay}

Let $\mu$ be an $f$-invariant measure, and $\mu _{\cdot}(\cdot)$ the disintegration of $\mu$. 
For each $\phi,\psi \in {\Ci}^{\infty}(\T^2)$ we define the {\it quenched operational correlation function} 
$\Cor_{\phi,\psi}^\mathrm{op}(\omega,n)$ of $(f,\mu)$ by
\begin{equation}
\Cor_{\phi,\psi}^\mathrm{op} (\omega ,n) =\int \phi \circ f^{(n)}(\omega) \cdot \bar\psi  dxds 
-\int \phi d\mu _{\theta ^n\omega}\int \bar\psi  dxds.
\end{equation}
We say that the operational correlation functions of $(f,\mu)$ decay exponentially fast when there 
exists a number $0<\rho <1$ (independent of $\omega$) and a set $\tilde\Omega$ of full measure such that 
for any $\omega\in\tilde\Omega$ and $\phi ,\psi \in {\Ci}^{\infty}(\T ^2)$ there is a constant $c(\omega)$
(depending on $\phi$ and $\psi$) such that
\begin{equation}\label{eq:constantsinexpdecay}
\lvert \Cor_{\phi ,\psi }^\mathrm{op} (\omega ,n) \rvert \leq c(\omega) \rho ^n. 
\end{equation} 
Similarly, we define the {\it quenched classical correlation function} 
$\Cor_{\phi,\psi}^\mathrm{cl}(\omega,n)$ of $(f,\mu)$ by
\begin{equation}
\Cor_{\phi,\psi}^\mathrm{cl} (\omega ,n) =\int \phi \circ f^{(n)}(\omega) \cdot \bar\psi  d\mu_\omega
-\int \phi d\mu _{\theta ^n\omega}\int \bar\psi d\mu_\omega,
\end{equation}
and define exponential decay in the same way.
We can now state and prove our main result.

\begin{thm}\label{thm:expdecay}
Let $\{f_{\epsilon}\}_{\epsilon >0}$ be a family of ${\Ci}^{\infty}(\T^2,\T^2)$ 
valued random variables satisfying \eqref{eq:perturbedsystem} and
\eqref{convinC1}, and let $\mu^\epsilon$ be the $f_\epsilon$-invariant
measure provided by Theorem \ref{thm:inv} with disintegration 
$\mu_\omega^\epsilon(dxds)=h_\epsilon(\omega,x)dxds$.
Assume that $f_0$ is partially captive. Then there is an $\epsilon_0$ such that
if $0\le\epsilon <\epsilon _0$ then the quenched random (operational and classical) correlation
functions of $(f_\epsilon,\mu^\epsilon)$ decay exponentially fast.
\end{thm}

\begin{proof}
Let $\lambda^{-\frac{1}{2}}<\rho <1$. By combining Theorems \ref{thm:discretespectrum} and \ref{thm:spectralgap}
we can find numbers $m\geq m_0$, $\epsilon_0=\epsilon_0(m)$ and $\nu_0=\nu_0(m)$
for which those results are in force,
where $m$ is chosen sufficiently large so that
$r_m =\log {(\lambda ^{- m -\frac{1}{2}}k^{\frac{1}{2}})}<0$.
After reducing $\epsilon_0$ if necessary, it follows that for all $0\le\epsilon<\epsilon_0$
we have $r^\ast\{\nu\}<0$ if $\nu\ne0$, while for $\nu=0$ the Lyapunov subspace 
$\Sigma_1(\epsilon;\omega)$ associated to the maximal Lyapunov exponent
$r^\ast\{0\}$ is one dimensional,
so that $\Sigma_1(\epsilon;\omega)=\C h_\epsilon(\omega)$.
We now fix $0\le\epsilon <\epsilon _0$, and let $\tilde\Omega$ be a $\theta$-invariant
set of full measure where the statements just made are guaranteed to hold for all $\omega$.
In particular, the definitions of $r^\ast\{\nu\}$ and $r^\ast_\mathrm{ic}\{\nu\}$
are realized in $\tilde\Omega$ (compare with \eqref{eq:r} and \eqref{eq:rbeta}, respectively.)
Below we always assume that $\omega$ belongs to $\tilde\Omega$ and make no further mention of $\tilde\Omega$.

We first prove that the operational correlation functions decay exponentially fast.
To this end, let $\pi_\epsilon (\omega) : H^{m}(\Sone) \to H^{m}(\Sone)$ be the projection defined by
\begin{equation*}
\pi_\epsilon (\omega) u= (u,1_{\Sone})_{L^2} h_\epsilon (\omega).
\end{equation*}
Note that $\pi_\epsilon(\omega)$ is idempotent since we have $(h_\epsilon (\omega),1_{\Sone})_{L^2}=1$
according to Theorem \ref{thm:inv}. The kernel $\Ker{\pi_\epsilon(\omega)}$ of $\pi_\epsilon(\omega)$
consists of those elements $u\in H^{m}(\Sone)$ such that $\int u dx=0$, and 
by Theorem \ref{thm:cocycle} we have
\begin{equation}\label{eq:estimateforthekernel}
\varphi\in\Ker{\pi_\epsilon(\omega)}\quad\Longrightarrow\quad
\limsup_{n\to\infty}{\frac{1}{n}\log{\lVert M_{0,n}^\ast(\epsilon;\omega)\varphi\rVert_{(m)}}}
\le\alpha_2<0
\end{equation}
since the range of $\pi_\epsilon(\omega)$
coincides with $\Sigma_1(\epsilon;\omega)$ by Theorem \ref{thm:discretespectrum}.

Let $\phi_1, \phi_2 \in {\Ci}^{\infty}(\T ^2)$. By
Theorem \ref{thm:inv}, the operational correlation function can be written
\begin{align}
\Cor_{\phi_1,\phi_2}^\mathrm{op}(\epsilon;\omega,n)&=(\phi_1,M^\ast_{f_\epsilon^{(n)}(\omega)}\phi_2)_{L^2(\T^2)}
\\ & \quad-(\phi_1,h_\epsilon(\theta^n\omega)\otimes 1_{\Sone})_{L^2(\T^2)}\cdot(1_{\T^2},\phi_2)_{L^2(\T^2)},
\end{align}
where the tensor product $h_\epsilon(\theta^n\omega)\otimes 1_{\Sone}$ is defined
by $(x,s)\mapsto h_\epsilon(\theta^n\omega,x)$.
For $j=1,2$ let $\varphi _{j,\nu}$ denote the ``Fourier coefficients'' of $\phi_j$
with respect to the Fourier decomposition \eqref{orthogonaldecomp} of $L^2(\T^2)$. 
Then
\begin{equation}
\Cor_{\phi_1,\phi_2}^\mathrm{op}(\epsilon;\omega,n)=\sum_\nu(\varphi_{1,\nu},M^\ast_{\nu,n}(\omega)\varphi_{2,\nu})_{L^2}
-(\varphi_{1,0},h_\epsilon(\theta^n\omega))_{L^2}\cdot(1_{\Sone},\varphi_{2,0})_{L^2}.
\end{equation}
In view of the definition of $\pi_\epsilon(\omega)$ we have according to Theorem \ref{thm:inv} that
\[
(\varphi _{1,0},M_{0,n} ^\ast (\omega) \pi_\epsilon (\omega) \varphi _{2,0}) _{L^2}
=(\varphi _{1,0},h_\epsilon (\theta^n\omega))_{L^2}\cdot(1_{\Sone},\varphi _{2,0}  )_{L^2},
\]
$\mathbb{P}$-almost surely.
Hence,
\begin{equation}\label{eq:exponentialdecaystep1}
\Cor_{\phi_1,\phi_2}^\mathrm{op}(\epsilon;\omega,n)=
(\varphi _{1,0},M_{0,n} ^\ast (\omega)(\mathrm{Id}- \pi_\epsilon (\omega)) \varphi _{2,0}) _{L^2}
+\sum_{\nu\ne0}(\varphi_{1,\nu},M^\ast_{\nu,n}(\omega)\varphi_{2,\nu})_{L^2}.
\end{equation}
In the sum appearing in the right-hand side above, each term with $\lvert \nu \rvert > \nu _0$ is bounded
by $C\rho ^{n}\lVert \varphi _{1,\nu}\rVert _{H_{\nu}^{-m}} \lVert \varphi _{2,\nu} \rVert _{H_{\nu}^{m}}$
by Theorem \ref{thm:spectralgap}.
It is straightforward to check that
$\lVert \varphi _{1,\nu} \rVert _{H_{\nu}^{-m}}$ and $\lVert \varphi _{2,\nu} \rVert _{H_{\nu}^{m}}$ are
$\mathcal O(\nu ^{-N})$ for any $N\in\N$ as $\lvert\nu\rvert\to\infty$.  Choosing $N=2$ we obtain
\begin{equation}\label{eq:exponentialdecaystep2}
\sum _{\lvert \nu \rvert > \nu _0} \lvert(\varphi _{1,\nu } ,M_{\nu,n} ^\ast (\omega) \varphi _{2,\nu })_{L^2}
\rvert \leq C\rho ^n
\end{equation}
for some new constant $C$.

To estimate the other finitely many terms, we first note that
$(\mathrm{Id}- \pi_\epsilon (\omega)) \varphi _{2,0}\in\Ker{\pi_\epsilon(\omega)}$
so \eqref{eq:estimateforthekernel} is in force. Moreover,
the maximal Lyapunov exponents $r^\ast\{\nu\}$ are negative
for $\nu\ne0$. Since also $r_m<0$, we can therefore find a positive number $\hat\rho<1$ such that
\[
\log{\hat\rho}>\max{\{\alpha_2,r_m,r^\ast\{\nu\} : \nu=\pm 1,\ldots,\pm \nu_0 \}}.
\]
In view of \eqref{eq:estimateforthekernel} 
there is an $n_0(\epsilon;\omega)$ such that
\begin{equation}\label{asymLyp2for0}
\frac{1}{n}\log{\lVert M_{0,n}^\ast(\omega)
(\mathrm{Id}- \pi_\epsilon (\omega)) \varphi _{2,0}\rVert_{(m)}}
<\log{\hat\rho}, \quad n\ge n_0(\epsilon;\omega).
\end{equation}
Similarly, for $0<\lvert\nu\rvert\le\nu_0$ we can find numbers $n_\nu(\epsilon;\omega)$ such that
\begin{equation}\label{asymLyp2}
\frac{1}{n}\log{\lVert M_{\nu,n}^\ast(\omega)
 \varphi _{2,\nu}\rVert_{(m)}}
<\log{\hat\rho}, \quad n\ge n_\nu(\epsilon;\omega).
\end{equation}
For such $\nu$, we may by Theorem \ref{thm:discretespectrum} (or rather its proof,
see Lemma \ref{lem:LYineq} below) find a constant $\tilde{C}_{\nu,m}$ independent of $\epsilon$ and $\omega$ such that
\begin{equation}
\max_{1\le n< n_\nu(\epsilon;\omega)}
\{\lVert M_{\nu,n}^\ast(\omega)
 \varphi _{2,\nu}\rVert_{(m)}\}\le \tilde{C}_{\nu,m} \lVert \varphi _{2,\nu}\rVert_{(m)}.
\end{equation}
If $1\le n<n_\nu(\epsilon;\omega)$ then
\[
\lVert M_{\nu,n}^\ast(\omega)
 \varphi _{2,\nu}\rVert_{(m)}\le \frac{\tilde{C}_{\nu,m} \lVert \varphi _{2,\nu}\rVert_{(m)}}{\hat \rho^n}\hat \rho^n
\leq \tilde{C}_{\nu,m} \lVert \varphi _{2,\nu}\rVert_{(m)} e^{-r_m n _{\nu}(\epsilon ;\omega)} \hat\rho ^n
\]
since $e^{r_m}\leq \hat \rho$. 
Together with \eqref{asymLyp2}, this is easily seen to imply that
for $0<\lvert\nu\rvert\le\nu_0$
there $\mathbb{P}$-almost surely exists a constant $C_\nu(\epsilon;\omega)>0$ such that
\begin{equation}\label{asymLyp}
\lVert M_{\nu,n}^\ast(\omega)
 \varphi _{2,\nu}\rVert_{(m)} \le C_\nu(\epsilon;\omega) \hat \rho ^n,
\quad n\ge1.
\end{equation}
The case $\nu=0$ can be treated similarly by using \eqref{asymLyp2for0} in place of \eqref{asymLyp2}.
Combined with \eqref{eq:exponentialdecaystep1} and \eqref{eq:exponentialdecaystep2},
this gives the estimate
\begin{equation*}
\lvert \Cor_{\phi _1,\phi _2}^\mathrm{op}(\epsilon;\omega ,n)\rvert\le
C\rho^n+(2\nu_0+1)\max_{\lvert \nu \rvert \le \nu _0}{\{ C_\nu(\epsilon;\omega)\lVert \varphi _{1,\nu}\rVert_{(-m)}\}}\hat\rho^n,
\end{equation*}
from which the conclusion follows.

We now turn to the classical
correlation function. However, $d\mu_\omega^\epsilon=h_\epsilon(\omega,x)dxds$ and $h_\epsilon(\omega)\in\Ci^\infty$ is real valued,
so by definition we have
\begin{equation}
\Cor_{\phi_1 ,\phi_2 }^\mathrm{cl} (\epsilon;\omega ,n)=\Cor_{\phi_1 ,h_\epsilon(\omega)\phi_2 }^\mathrm{op} (\epsilon;\omega ,n).
\end{equation}
Moreover, the ``Fourier coefficients'' of $(x,s)\mapsto h_\epsilon(\omega,x)\phi_2(x,s)$
with respect to the Fourier decomposition \eqref{orthogonaldecomp} of $L^2(\T^2)$
are given by $h_\epsilon(\omega)\varphi_{2,\nu}$ with $\varphi_{2,\nu}$ defined above. 
Hence, a repetition of the proof with $\varphi_{2,\nu}$ replaced by $h_\epsilon(\omega)\varphi_{2,\nu}$
yields the result. Indeed, the only thing we need to check is that 
$\lVert h_\epsilon(\omega)\varphi_{2,\nu}\rVert _{H_{\nu}^{m}}$ is still
$\mathcal O(\nu ^{-N})$ for any $N\in\N$ as $\lvert\nu\rvert\to\infty$, and that
$\lVert h_\epsilon(\omega)\varphi_{2,\nu}\rVert _{(m)}$ can still be estimated by a constant $C_{\nu,m}$.
But this can be done in the same way as before if we also use \eqref{eq:stability}
(or Baladi et al.~\cite{BKS96}*{Theorem A}). The proof is complete.
\end{proof}

\begin{rmk}
It is usually desirable to see if decay of quenched correlation
can lead to decay of integrated correlation. This question boils down to
$\mathbb P$-integrability of the constants $c(\omega)$ appearing in \eqref{eq:constantsinexpdecay}.
Inspecting the proof above, we see that this seems to be intractable with our methods
since we have no information about the numbers $n_\nu(\epsilon;\omega)$ responsible
for the dependence on the noise parameters $\epsilon$ and $\omega$.
\end{rmk}

\section{The proof of Theorem \ref{thm:spectralgap}}\label{section:spectralgap}

In this section we give the proof of Theorem \ref{thm:spectralgap}, by essentially
adapting the proof of Faure~\cite{Faure}*{Theorem 3}, which is made possible by a careful
microlocal analysis of pseudodifferential operators depending on the parameter
$\omega$, found in Appendix \ref{app:PsiDO}. To make comparison easier,
we have tried to follow the structure of Faure's proof.

We first require some preparation.
Let $g_\ve(\omega)$, $E_\ve(\omega)$ 
and $\tau_\ve(\omega)$ be defined through the perturbation scheme $\{f_\ve\}_{\ve>0}$
as in \eqref{eq:perturbedsystem}.
By assumption we can find a number $\epsilon_0>0$ 
such that \eqref{emin} holds for any $0\le\epsilon<\epsilon_0$.
For such $\epsilon$, let $M_{\nu,n}(\omega)=M_{\nu,n}(\epsilon;\omega)$ be
the operator given by \eqref{restrictionoperator}.
As before, we simply write $M_{\nu}(\omega)$ instead of $M_{\nu,1}(\omega)$ when $n=1$.
We remark that when $\nu\in\Z$ is fixed, the operator $M_{\nu,n}(\omega)$
can be regarded as a Fourier integral operator
acting on ${\Ci}^\infty(\Sone)$. 
When $\nu=1/h$ is considered as a semiclassical parameter,
the operator $M_{\nu,n}(\omega)$
can be regarded as a semiclassical Fourier integral operator
acting on ${\Ci}^\infty(\Sone)$. 
Since the general theory of Fourier integral operators will not
be essential to our exposition, we refer the reader to H{\"o}rmander~\cite{Hormander4}*{Chapter 25}
for details. For a discussion relevant to our context,
see also Faure~\cite{Faure}*{Section 3.2} where a description of semiclassical Fourier integral operators
is given using wave packets. However, we do wish to record the following observation.

If $X$ and $Y$ are ${\Ci}^\infty$ manifolds, and $A:{\Ci}^\infty(Y)\to\De'(X)$
is a Fourier integral operator, then to $A$ is associated a {\it canonical relation}
$\Lambda\subset T^\ast(X\times Y)\cong T^\ast(X)\times T^\ast(Y)$.
Recall that a canonical relation is a manifold such that the {\it twisted manifold}
\[
\Lambda'=\{(x,\xi;y,\eta):(x,\xi;y,-\eta)\in\Lambda\}
\]
is Lagrangian with respect to the standard symplectic form on $T^\ast(X\times Y)$.
We may also view $\Lambda$ as the graph of a relation from $T^\ast (Y)$ to $T^\ast (X)$
which we also denote by $\Lambda$. If $V\subset T^\ast (Y)$, the action of this relation is defined by
\begin{equation}\label{Faurecanonicalmap}
\Lambda (V)=\{(x,\xi)\in T^\ast(X):\exists \, (y,\eta)\in V\text{ such that }(x,\xi;y,\eta)\in\Lambda\}.
\end{equation}
In the case when $A:{\Ci}^\infty(\Sone)\to {\Ci}^\infty(\Sone)$ is the operator
$M_\nu(\omega)$ given by \eqref{restrictionoperator},
the map defined by \eqref{Faurecanonicalmap} is referred to by Faure~\cite{Faure} as the {\it canonical map of $A$},
and we shall adhere to this terminology.
When $\nu$ is viewed as a semiclassical parameter, we shall let $F(\epsilon;\omega)$ denote the canonical map of
$M_\nu(\epsilon;\omega)$. Explicitly, we have
\begin{equation}\label{canonicalmap1}
F(\epsilon;\omega,y,\eta)=\{(x,E_\epsilon'(\omega,x)\eta+\tau_\epsilon'(x,\omega)): y=E_\epsilon(\omega,x)\},
\end{equation}
see Faure~\cite{Faure}*{Proposition 11}. Note that this differs from the canonical map of $M_\nu(\epsilon;\omega)$
when $\nu$ is fixed by the additional term $\tau_\epsilon'(\omega)$, see Faure~\cite{Faure}*{Proposition 6}.
Note also that since $E_\epsilon(\omega):\Sone\to \Sone$ is a $k:1$ map,
there are precisely $k$ distinct elements belonging to the set in the right-hand side of \eqref{canonicalmap1}.
Hence the map $F(\epsilon;\omega):T^\ast (\Sone)\to T^\ast (\Sone)$ is $k$ valued.
We mention that matters are generally simplified if $\Lambda$ is the graph of
a canonical transformation $\chi:T^\ast (Y)\to T^\ast (X)$,
and then $A$ is said to be associated with $\chi$. In particular, if $X=Y$ and $A$ is
a pseudodifferential operator then $\Lambda$ is the graph of the identity $\chi(x,\xi)=(x,\xi)$.

Finally, we remark that it suffices to prove Theorem \ref{thm:spectralgap} in the case when $\nu\to\infty$. Indeed,
it is straightforward to check that $M_{-\nu,n}(\omega)u(x)$
is the complex conjugate of $M_{\nu,n}(\omega)\bar u(x)$, and 
that $M_{\nu,n}(\omega)$ and $M_{-\nu,n}(\omega)$ therefore have
the same operator norms. This proves the claim.

\subsection{Dynamics on the contangent bundle}
With the previous discussion in mind, we now define (injective) maps
$F_j(\epsilon;\omega)$ on the cotangent bundle $T^\ast(\Sone)$ for $0\le j\le k-1$ by
\begin{equation}\label{canonicalmap2}
F_j(\epsilon;\omega):\binom{y}{\eta}\mapsto 
\binom{x_j(\epsilon;\omega)}{\xi_j(\epsilon;\omega)}=
\left( \! \! \begin{array}{c} g_\epsilon(\omega)^{-1}((y+j)/k)\\ 
E_\epsilon'(\omega,x_j(\epsilon;\omega))\eta+\tau_\epsilon'(\omega,x_j(\epsilon;\omega))
\end{array} \! \! \right).
\end{equation}
Then the canonical map of $M_\nu(\epsilon;\omega)$
given by \eqref{canonicalmap1} satisfies
\[
F(\epsilon;\omega):(y,\eta)\mapsto\{F_0(\epsilon;\omega,y,\eta),\ldots,F_{k-1}(\epsilon;\omega,y,\eta)\},\quad(y,\eta)\in T^\ast (\Sone)
\]
by virtue of \eqref{eq:perturbedsystem} and \eqref{canonicalmap1}.
For future purposes, we introduce the following notation, where $\N^\ast$ refers to the set of positive integers.

\begin{dfn}\label{def:randomtrajectory}
Let $(y,\eta)\in T^\ast (\Sone)$. For an arbitrary sequence $\alpha=(\ldots,\alpha_3,\alpha_2,\alpha_1)$ in $\{0,\ldots,k-1\}^{\N^\ast}$
and time $n\in\N^\ast$ let 
\[
F_\alpha^{(n)}(\epsilon;\omega,y,\eta)=F_{\alpha_n}(\epsilon;\omega)\circ
F_{\alpha_{n-1}}(\epsilon;\theta\omega)\circ\ldots\circ
F_{\alpha_{1}}(\epsilon;\theta^{n-1}\omega)(y,\eta).
\]
For a given sequence $\alpha\in\{0,\ldots,k-1\}^{\N^\ast}$,
we say that $\{F_\alpha^{(n)}(\epsilon;\omega,y,\eta):n\in\N\}$ is a {\it random trajectory}
of $F(\epsilon;\omega)$ issued from the point $(y,\eta)$. At time $n\in\N^\ast$, there are $k^n$ points
issued from a given point $(y,\eta)$. This set of points is the image of $(y,\eta)$
under the relation $F(\epsilon;\omega)\circ\ldots\circ F(\epsilon;\theta^{n-1}\omega)$ and will be denoted by
$F^{(n)}(\epsilon;\omega,y,\eta)$, that is,
\begin{equation}\label{treeofrandomtrajectories}
F^{(n)}(\epsilon;\omega,y,\eta)=\{F_\alpha^{(n)}(\epsilon;\omega,y,\eta),\ \alpha\in\{0,\ldots,k-1\}^n\}.
\end{equation}
\end{dfn}

It will be convenient to record the following version of Faure~\cite{Faure}*{Lemma 13} in the presence of
parameters $\epsilon$ and $\omega$, proved in the same way.

\begin{lem}\label{lemma13}
Let $\epsilon_0>0$ and $\lambda$ be given by \eqref{emin},
and assume that $\epsilon_0$ is small enough so that
$\esssup_\omega{\max_x{\lvert\tau_\epsilon'(\omega,x)\rvert}}\le C_\tau$ for $0\le\epsilon<\epsilon_0$
and some constant $C_\tau>0$ independent of $\omega$.
Then, for any $1<\kappa<\lambda$, there exists an $R_\kappa\ge 0$
such that if $0\le\epsilon<\epsilon_0$ and $\lvert\eta\rvert>R_\kappa$ then
\begin{equation}\label{eq:escape}
\lvert\xi_j(\epsilon;\omega)\rvert>\kappa\lvert\eta\rvert,\quad 0\le j\le k-1,
\end{equation}
$\mathbb{P}$-almost surely, where $\xi_j(\epsilon;\omega)$ is given by \eqref{canonicalmap2}.
\end{lem}

\begin{proof}
Note that the existence of an $\epsilon_0$ such that
$\esssup_\omega{\max_x{\lvert\tau_\epsilon'(\omega,x)\rvert}}\le C_\tau$ for $0\le\epsilon<\epsilon_0$
is guaranteed by \eqref{convinC1}.
We can for example take $C_\tau=\max_x{\lvert\tau_0'(x)\rvert}+1$ so that $C_\tau$ ultimately depends on $f_0$.
In view of \eqref{canonicalmap2} we have
\[
\lvert\xi_j(\epsilon;\omega)\rvert 
\ge \lvert (E_\epsilon'(\omega,x_j(\epsilon;\omega))-\kappa)\lvert\eta\rvert+\kappa\lvert\eta\rvert-\lvert\tau_\epsilon'(\omega,x_j(\epsilon;\omega))\rvert\rvert.
\]
If $\lvert\eta\rvert>C_\tau/(\lambda-\kappa)$ then
\[
(E_\epsilon'(\omega,x_j(\epsilon;\omega))-\kappa)\lvert\eta\rvert-\lvert\tau_\epsilon'(\omega,x_j(\epsilon;\omega))\rvert> 0,
\]
$\mathbb{P}$-almost surely, which yields the result.
\end{proof}

We now fix $1<\kappa<\lambda$, say $\kappa=(\lambda+1)/2$.
In view of the proof of Lemma \ref{lemma13} and the definition of $\lambda$, 
it then follows that the number $R_\kappa$ given by Lemma \ref{lemma13} ultimately depends on $f_0$.
Define a set $\mathcal Z\subset T^\ast (\Sone)$ by
\begin{equation}\label{def:Z}
\mathcal Z=\Sone\times [-R_\kappa,R_\kappa].
\end{equation}
If $(y,\eta)\in\complement\mathcal Z$, then Lemma \ref{lemma13} implies that a random trajectory issued from $(y,\eta)$ will
$\mathbb{P}$-almost surely escape in a controlled manner in the sense of \eqref{eq:escape}.
Hence, for sufficiently small $\epsilon>0$ the set defined by \eqref{treeofrandomtrajectories} 
$\mathbb{P}$-almost surely satisfies
\begin{equation}\label{eq:controledescape}
(y,\eta)\in\complement\mathcal Z\quad\Longrightarrow\quad F^{(n)}(\epsilon;\omega,y,\eta)\cap\mathcal Z=\emptyset,
\quad n\ge 1.
\end{equation}
Contrary to the deterministic case, knowledge about a random trajectory at time $n$
does not allow us to say much about the same trajectory at a later time, that is,
\eqref{eq:controledescape} cannot be used to deduce that if
$F_{\alpha}^{(n)}(\epsilon;\omega,y,\eta)$ belongs to $\complement\mathcal Z$ then so does
$F_{\alpha}^{(n')}(\epsilon;\omega,y,\eta)$ for all $n'>n$.
Indeed, if $\alpha\in\{0,\ldots,k-1\}^{\N^\ast}$
is a given sequence, let $n'>n$. Comparing $F_\alpha^{(n)}(\epsilon;\omega,y,\eta)$ and
$F_\alpha^{(n')}(\epsilon;\omega,y,\eta)$, we find that
\[
F_\alpha^{(n')}(\epsilon;\omega,y,\eta)=F_{\alpha_{n'}}(\epsilon;\omega)\circ\ldots\circ F_{\alpha_{n+1}}(\epsilon;\theta^{n'-n-1}\omega)
(F_{\alpha}^{(n)}(\epsilon;\theta^{n'-n}\omega,y,\eta)),
\]
which is not of the form $F_{\alpha_{n'}}(\epsilon;\omega)\circ\ldots\circ F_{\alpha_{n+1}}(\epsilon;\theta^{n'-n-1}\omega)
(F_{\alpha}^{(n)}(\epsilon;\omega,y,\eta))$, so \eqref{eq:controledescape} cannot be used.
However, the same arguments show that if at some time $n\in\N^\ast$ we $\mathbb{P}$-almost
surely have $F^{(n)}(\epsilon;\omega,y,\eta)\cap\mathcal Z=\emptyset$,
then we $\mathbb{P}$-almost surely also have $F^{(n')}(\epsilon;\omega,y,\eta)\cap\mathcal Z=\emptyset$ for all $n'>n$.
Note that even if the entire tree $F^{(n)}(\epsilon;\omega,y,\eta)$ does not in this way escape toward infinity,
some random trajectories issued from the point $(y,\eta)\in\mathcal Z$ still might.

\subsection{The trapped set}
To quantify
the number of random trajectories that do not escape toward infinity, we define the {\it trapped set} $K(\epsilon;\omega)$
as the complement of the set
\[
\{(y,\eta):\exists \, n\in\N\text{ such that }F^{(n)}(\epsilon;\omega,y,\eta)\cap\mathcal Z=\emptyset\}. 
\]

\begin{dfn}\label{trappedset}
The trapped set is defined as
\begin{align*}
K(\epsilon;\omega)&=\bigcap_{n\in\N}(F(\epsilon;\omega)\circ\ldots\circ F(\epsilon;\theta^{n-1}\omega))^{-1}(\mathcal Z)\\
&=\{(y,\eta):F^{(n)}(\epsilon;\omega,y,\eta)\cap\mathcal Z\neq\emptyset\text{ for all }n\ge 1\}. 
\end{align*}
\end{dfn}

For $n\in\N^\ast$, let $\mathcal N (\epsilon;\omega,n)$ be the number of random trajectories which do not escape
outside the vicinity $\mathcal Z$ of the trapped set $K(\epsilon;\omega)$ before time $n$,
\begin{equation}\label{numbertrajectories}
\mathcal N(\epsilon;\omega,n)=\max_{(y,\eta)\in 
\mathcal Z}\#\{F_\alpha^{(n)}(\epsilon;\omega,y,\eta)\in\mathcal Z:\alpha\in\{0,\ldots,k-1\}^n\}.
\end{equation}
Since the number of points issued from $(y,\eta)$ at time $n$ are $k^n$, we have $\mathcal N(\epsilon;\omega,n)\le k^n$.

Next, we recall the notion of partial captivity.
The definition will only be applied to the deterministic map $f_0$ given by \eqref{eq:unperturbedsystem},
and requires the notions corresponding to $F(\epsilon;\omega)$ and
$\mathcal N(\epsilon;\omega,n)$ introduced above for the perturbation scheme $f_\epsilon$.
Instead of introducing additional notation, we prefer the following
(somewhat convoluted) statement. The reader is asked to compare with
the equivalent Faure~\cite{Faure}*{Definition 15}.

\begin{dfn}\label{def:randompartiallycaptive}
The $k$ valued map $F(\epsilon;\omega):T^\ast (\Sone)\to T^\ast (\Sone)$ is said to be partially captive for $\omega$ if
\begin{equation}\label{eq:randompartiallycaptive}
\lim_{n\to\infty}\frac{1}{n}\log{\mathcal N(\epsilon;\omega,n)}=0.
\end{equation}
When $F(\epsilon;\omega)$ is  partially captive for $\omega$, we shall permit us to say
that also $f_{\epsilon}(\omega)$ is partially captive for $\omega$, where $f_{\epsilon}(\omega)$
is the map given by \eqref{eq:perturbedsystem}, or by \eqref{eq:unperturbedsystem}
when $\epsilon=0$.
\end{dfn}

\begin{rmk}
Consider the deterministic case when
$(\Omega,\mathcal F,\mathbb{P})$ is the trivial probability space consisting of one point and $\theta$
is the identity on $\Omega$, and assume that $\tau_\epsilon(\omega)$ is {\it cohomologous to a constant}, 
i.e.~$\tau_\epsilon(\omega,x)=\varphi(E_\epsilon(\omega,x))-\varphi(x)+c$ for some $\varphi\in\Ci^\infty(\Sone)$ and $c\in\R$.
Then, as mentioned in the introduction, $f_\epsilon(\omega)$ cannot be partially captive for $\omega$,
see Faure~\cite{Faure}*{Appendix A} together with
the remark on p.~1490 in the mentioned paper. Actually, in Faure~\cite{Faure} only the case
$c=0$ is considered but the arguments apply also to $c\ne0$. 
Indeed, if $\tau_\epsilon(\omega)$ is constant then it does not contribute to the
canonical map $F(\epsilon;\omega)$ given by \eqref{canonicalmap1},
and trajectories issued from the zero section $\{(x,0):x\in\Sone\}\subset T^\ast\Sone$ do not
leave this section. Hence, $\mathcal N(\epsilon;\omega,n)=k^n$ so $f_\epsilon(\omega)$
cannot be partially captive. As mentioned by Faure, one might instead say that $f_\epsilon(\omega)$ is then
{\it totally captive}. The case when $\tau_\epsilon(\omega)$ is cohomologous to a constant can be treated
using the arguments in Faure~\cite{Faure}*{Appendix A}.
\end{rmk}

The following proposition will be crucial for the proof of Theorem \ref{thm:spectralgap}.
In order not to interrupt the exposition we postpone its proof, which is the contents of Section \ref{app:wPC}.

\begin{prop}\label{wPC}
Suppose that $f_0$ is partially captive. Then for any $\delta >0$, there exists $n_{\delta} \in \mathbb{N}$ such that for any $n\geq n_{\delta}$ one can find $\epsilon (n)>0$ such that 
\begin{align*}
\esssup _{\omega} \left\lvert \frac{\log{\mathcal{N} (\epsilon;\omega ,n)}}{n} \right\rvert <\delta,
\quad 0\le\epsilon <\epsilon (n).
\end{align*}
\end{prop}

\begin{rmk1}
Proposition \ref{wPC} does not imply that partial captivity is stable. 
What we mean is that the proposition does not guarantee the
existence of an $\epsilon _0$ such that $f_{\epsilon}(\omega)$ is $\mathbb{P}$-almost surely  
partially captive for $\omega$ provided $0\le \epsilon <\epsilon _0$. In particular,
our methods only guarantee the existence of a number $\epsilon (n)$ which must be allowed to depend on 
$n$, see Proposition \ref{prop:differentnoiselevels} and the remark at the end of Section \ref{app:wPC}.
This is reminiscent of the perturbation lemmas of Baladi and Young~\cite{BY93}*{Section 2}.
\end{rmk1}

\begin{rmk2}
In \eqref{numbertrajectories}, the number $\mathcal N(\omega,n)$ naturally depends on
the definition of the set $\mathcal Z$.
However, property \eqref{eq:randompartiallycaptive} does not, see the remark
on page 1490 in Faure~\cite{Faure}.
\end{rmk2}

As a final preparation before turning to the proof of Theorem \ref{thm:spectralgap},
we will indicate how Proposition \ref{wPC} will be used. 
Let therefore $\rho>\lambda^{-\frac{1}{2}}$ be
an arbitrary number as in the statement of Theorem \ref{thm:spectralgap}.
For reasons which will become apparent, 
apply Proposition \ref{wPC} with $\delta=\log{(\rho^2\lambda)}/2$ to find
numbers $n_0\in\N$ and $\epsilon(n_0)>0$ such that for any $\epsilon$ in the range
$0\le\epsilon<\epsilon(n_0)$ we have
\begin{equation}\label{eq:choosingnzero1}
\frac{\log{(C_{\mathrm{dim}}\cdot k)}}{n_0-1}
+\esssup_\omega\bigg\lvert\frac{\log{\mathcal N(\epsilon;\omega,n_0-1)}}{n_0-1}\bigg\rvert<\log{(\rho^2\lambda)},
\end{equation}
where $C_{\mathrm{dim}}$ is a constant depending only on the dimension $\dim{\Sone}=1$,
which will appear in Theorem \ref{thm:Pn} below.

\subsection{The escape function}
We can now prove Theorem \ref{thm:spectralgap}. We will keep the notation already
introduced in this section, and
let $\epsilon>0$ be an arbitrary number
for which \eqref{emin}, \eqref{eq:choosingnzero1} and Lemma \ref{lemma13} are in force,
and suppress it from the notation.
Note that the latter condition implies that, in
particular, we have
\begin{equation}\label{eq:choosingnzero2}
\frac{C_{\mathrm{dim}}\cdot k\mathcal N(\theta\omega,n_0-1)}{\lambda^{n_0}}<\rho^{2n_0},\quad
\text{$\mathbb{P}$-almost surely}.
\end{equation}
One additional restriction on the range of $\epsilon$ will be imposed in Theorem \ref{thm:Pn}
below, but since the interluding discussion is unaffected, this should cause no confusion.
Throughout the rest of the proof, $h=1/\nu$ and we will permit us to switch
between $h$ and $\nu$ as befitting the situation.

Recall that we fixed $\kappa=(\lambda+1)/2$ above, so that the number
$R=R_\kappa>0$ given by Lemma \ref{lemma13} only depends on $f_0$.
Now let $0<\delta_0<(\kappa-1)R$ be small. Let $m>0$ and define an {\it escape function} $a_m\in {\Ci}^\infty(T^\ast (\Sone))$ by
\begin{equation}\label{eq:escapefunction}
a_m(y,\eta)\equiv  a_{R,\delta_0,m}(y,\eta)=\begin{cases} (1+\eta^2)^{m/2} & \text{if }\lvert\eta\rvert\ge R+\delta_0, \\
1 & \text{if }\lvert\eta\rvert\le R, \end{cases}
\end{equation}
in such a way that for $R<\lvert\eta\rvert<R+\delta_0$, $a_m(y,\eta)$ is an even function of $\eta$, independent of $y$, which
satisfies
\[
1< a_m(y,\eta)< (1+\eta^2)^{m/2},\quad R<\lvert\eta\rvert<R+\delta_0.
\]
Extend the definition to negative numbers $-m<0$ by requiring that the identity $a_m^{-1}=a_{-m}$ holds for all $m>0$.
By virtue of \eqref{eq:escape}, the construction ensures that if $m>0$ then
$a_m$ is $\mathbb{P}$-almost surely 
strictly increasing along the random trajectories of $F(\omega)$
outside the vicinity $\mathcal Z$ of the trapped set,
and that for all $\lvert\eta\rvert>R$ and $0\le j\le k-1$ we have
\begin{equation}\label{eq:outsidetrappedbound}
\frac{a_m(y,\eta)}{a_m(F_j(\omega,y,\eta))}\le \left(\frac{1+R^2}{1+\kappa^2R^2}\right)^{m/2},
\quad\text{$\mathbb{P}$-almost surely if $\lvert\eta\rvert>R$}. 
\end{equation}
Indeed, if $\lvert\eta\rvert>R$ then $\lvert\xi_j(\omega)\rvert>\kappa R=R+(\kappa-1)R>R+\delta_0$, so
\[
a_m(F_j(\omega,y,\eta))=(1+\lvert\xi_j(\omega)\rvert^2)^{m/2}\ge (1+\kappa^2\eta^2)^{m/2},
\]
where the last inequality follows by Lemma \ref{lemma13}.
Since $a_m(y,\eta)\le (1+\eta^2)^{m/2}$, standard differential calculus yields the estimate
\eqref{eq:outsidetrappedbound}.

Let $C_\kappa$ denote the square root appearing
in \eqref{eq:outsidetrappedbound}, so that when $\lvert\eta\rvert>R$ we $\mathbb{P}$-almost surely have
$a_m(y,\eta)\le C_\kappa^m a_m(F_j(\omega,y,\eta))$. Note that $C_\kappa<1$.
If $\lvert\eta\rvert\le R$ then $a_m(y,\eta)=1$ so we have the general bound
\begin{equation}\label{eq:generalbound}
\frac{a_m(y,\eta)}{a_m(F_j(\omega,y,\eta))}\le 1,\quad\text{for all }(y,\eta)\in T^\ast (\Sone),\quad 0\le j\le k-1.
\end{equation}

Let $A_m=a_m(hD)$ be the semiclassical operator
acting through multiplication by $\xi\mapsto a_m(h\xi)$ on the Fourier side, so that
for $u\in\De'(\Sone)$, the Fourier coefficients of $A_m u$ are given by
\[
(A_m u)\hat{\ }(\xi)=a_m(h\xi)\hat u(\xi),
\quad u\in\De'(\Sone).
\]
Then $A_m$ is a semiclassical operator with symbol $a_m$ belonging to the symbol class $S^m(T^\ast(\Sone))$,
see Appendix \ref{app:PsiDO}. Note that $A_m$ is formally self-adjoint and invertible on ${\Ci}^\infty(\Sone)$.
With $h=1/\nu>0$, define a norm $\vertiii{\phantom{i}}_{H_\nu^m}$ on $H^m(\Sone)$
by
\begin{equation}\label{equivalentnorm}
\vertiii{u}_{H_\nu^m}^2=\sum_{\xi\in2\pi\Z}\lvert a_m(\xi/\nu)\hat u(\xi)\rvert^2,\quad u\in H^m(\Sone),\quad h=1/\nu>0.
\end{equation}
This norm is equivalent to the Sobolev norm $\lVert\phantom{i}\rVert_{H_\nu^m}$ on $H^m(\Sone)$
which appears in the statement of Theorem \ref{thm:spectralgap}. In fact, with $C=(1+(R+\delta_0)^2)^{\frac{1}{2}}$
we have
\[
\frac{1}{C^m}\lVert u\rVert_{H_\nu^m}\le \vertiii{u}_{H_\nu^m}\le C^m \lVert u\rVert_{H_\nu^m}. 
\]
Since the constant $c_0(m)$ appearing in the statement of Theorem \ref{thm:spectralgap}
is allowed to depend on $m$, it therefore suffices to prove the theorem with $\lVert\phantom{i}\rVert_{H_\nu^m}$
replaced by $\vertiii{\phantom{i}}_{H_\nu^m}$.
(That $c_0(m)$ depends on $m$ is a direct artifact of this approach.)
Since we will not be switching between these norms for the rest of this section,
we shall for simplicity write $\lVert\phantom{i}\rVert_{H_\nu^m}$ for the norm defined by \eqref{equivalentnorm}, and
let $H_\nu^m(\Sone)$ denote the set of distributions $u\in H^m(\Sone)$
equipped with this norm. 
$A_m:H_\nu^s(\Sone)\to H_\nu^{s-m}(\Sone)$ is then
an isomorphism for all $s\in\R$, and we identify $H_\nu^m(\Sone)$
with $A_m^{-1}(L^2(\Sone))$ for each $m$.

The commutative diagram
\[
\begin{array}{ccc}
L^2(\Sone) & \stackrel{Q_\nu(\omega)}{\to} & L^2(\Sone) \\
\downarrow\text{\scriptsize{$A_m$}} & \circlearrowleft & \downarrow\text{\scriptsize{$A_m$}} \\
H_\nu^{-m}(\Sone) & \stackrel{M_\nu(\omega)}{\to} & H_\nu^{-m}(\Sone)
\end{array}
\]
shows that $M_\nu(\omega):H_\nu^{-m}(\Sone)\to H_\nu^{-m}(\Sone)$ is unitarily equivalent to
\[
Q_\nu(\omega)=A_m^{-1} M_\nu(\omega)A_m:L^2(\Sone)\to L^2(\Sone).
\]
With $M_{\nu,n}(\omega)$ given by \eqref{restrictionoperator},
a straightforward computation shows that 
$M_{\nu,n}(\omega):H_\nu^{-m}(\Sone)\to H_\nu^{-m}(\Sone)$ 
is then unitarily equivalent to the operator $Q_{\nu,n}(\omega)$, defined by
\begin{equation}\label{auxiliaryoperator}
Q_{\nu,n}(\omega)=Q_\nu(\omega)\circ\ldots\circ Q_\nu(\theta^{n-1}\omega):L^2(\Sone)\to L^2(\Sone).
\end{equation}
For each $n\in\N^\ast$, define $P_n(\omega)=(Q_{\nu,n}(\omega))^\ast Q_{\nu,n}(\omega)$. Then
\[
P_n(\omega)
=A_m M_{\nu}^\ast(\theta^{n-1}\omega)\circ\ldots\circ
M_{\nu}^\ast(\omega)A_m^{-2} M_{\nu}(\omega)
\circ\ldots\circ M_{\nu}(\theta^{n-1}\omega)A_m.
\]
Although not signified in the notation, $P_n(\omega)$ implicitly depends on the Sobolev index $m$.
For a sequence $\alpha\in\{0,\ldots,k-1\}^{\N^\ast}$ 
let $G_{\alpha}^{(n)}(\omega):T^\ast(\Sone)\to\R$ 
be the image of the projection onto the first coordinate of $F_\alpha^{(n)}(\omega,y,\eta)$. 
Then $G_{\alpha}^{(n)}(\omega)$ is independent of $\eta$, and
\begin{equation}\label{eq:introducingG}
G_\alpha^{(n)}(\omega,y)=G_{\alpha_n}(\omega)\circ
G_{\alpha_{n-1}}(\theta\omega)\circ\ldots\circ
G_{\alpha_{1}}(\theta^{n-1}\omega)(y),
\end{equation}
where $G_j(\omega,y)=g(\omega)^{-1}((y+j)/k)=x_j(\omega)$
in view of Definition \ref{def:randomtrajectory} and \eqref{canonicalmap2}.
It is straightforward to check that
\begin{equation}\label{eq:derivativeofG}
\frac{d G_\alpha^{(n)}(\omega,y)}{dy}
=\prod_{j=1}^{n}\frac{1}{E'(\theta^{n-j}\omega)\circ G_\alpha^{(j)}(\theta^{n-j}\omega,y)}.
\end{equation}
Hence, by \eqref{emin} it follows that $dG_{\alpha}^{(n)}(\omega,y)/dy\le\lambda^{-n}$ $\mathbb{P}$-almost surely.

\begin{thm}\label{thm:Pn}
The operator $P_n(\omega)$ is $\mathbb{P}$-almost surely a semiclassical operator with
principal symbol $p_n(\omega)$ given by
\begin{equation}\label{principalsymbolP}
p_n(\omega,y,\eta)=\sum_{\alpha\in\{0,\ldots,k-1\}^{n}}
\frac{a_m^2(y,\eta)}{a_m^2(F_\alpha^{(n)}(\omega,y,\eta))}\cdot
\frac{d G_\alpha^{(n)}(\omega,y)}{dy}.
\end{equation}
Moreover, there is an $\epsilon_0$, depending on $m$ but independent of time $n\in\N^\ast$, such that
for all $0\le\epsilon<\epsilon_0$ and all $n\in\N^\ast$, the operator
$P_n(\omega):L^2(\Sone)\to L^2(\Sone)$ $\mathbb{P}$-almost surely satisfies
\begin{equation}\label{normbound1}
\lVert P_n(\omega)\rVert_{\mathscr L(L^2(\Sone))}\le C_{\mathrm{dim}}\lVert p_n(\omega)\rVert_{L^\infty(T^\ast(\Sone))}
+\mathcal O_{n,m}(h^{\frac{1}{2}})
\quad\text{as }h\to0,
\end{equation}
where the error term as indicated depends on $n$ and $m$ but not on $\omega$,
and $C_{\mathrm{dim}}$ only depends on the dimension $\dim{\Sone}=1$.
\end{thm}

The proof of this result relies on the contents of the extensive Appendix \ref{app:PsiDO},
and can be found at the end of that appendix, see page \pageref{appendixproofoftheorem}.
The complicating factor is that the error term of the operator norm bound
needs to be independent of $\omega$ for the proof of Theorem \ref{thm:spectralgap} to work,
and this means that it does not suffice to calculate only the principal symbol $p_n(\omega)$.
This notwithstanding, we wish to mention that Theorem \ref{thm:Pn} is essentially a special
case of Egorov's theorem, together with an $L^2$ continuity result for semiclassical operators
on $\Sone$.

To find a bound for the principal symbol $p_n(\omega)$ we note that
\eqref{principalsymbolP} together with the properties of $dG_{\alpha}^{(n)}(\omega,y)/dy$ gives
\begin{equation}\label{estimate:principalsymbolP}
0<p_n(\omega,y,\eta)\le\lambda^{-n}\sum_{\alpha\in\{0,\ldots,k-1\}^{n}}
\frac{a_m^2(y,\eta)}{a_m^2(F_\alpha^{(n)}(\omega,y,\eta))}.
\end{equation}
Next, we consider three cases:

1) $(y,\eta)\notin \mathcal Z$. Note that if $\alpha=(\ldots,\alpha_2,\alpha_1)\in\{0,\ldots,k-1\}^{\N^\ast}$
and $\omega\in\Omega$, then $F_\alpha^{(n)}(\omega,y,\eta)=F_{\alpha_n}(\omega,F_\alpha^{(n-1)}(\theta\omega,y,\eta))$.
Hence, we observe that we can write 
\[
\frac{a_m^2(y,\eta)}{a_m^2(F_\alpha^{(n)}(\omega,y,\eta))}=
\frac{a_m^2(F_\alpha^{(n-1)}(\theta\omega,y,\eta))}{a_m^2(F_{\alpha_n}(\omega,F_\alpha^{(n-1)}(\theta\omega,y,\eta)))}
\cdots \frac{a_m^2(y,\eta)}{a_m^2(F_{\alpha_1}(\theta^{n-1}\omega,y,\eta))},
\]
which in view of \eqref{eq:outsidetrappedbound} gives
$a_m^2(y,\eta)/a_m^2(F_\alpha^{(n)}(\omega,y,\eta))\le (C_\kappa^{2m})^n$, $\mathbb{P}$-almost surely.
Thus, $\mathbb{P}$-almost surely
\[
(y,\eta)\notin \mathcal Z\quad\Longrightarrow\quad 0< p_n(\omega,y,\eta)\le \Big(\frac{k}{\lambda}\Big)^n (C_\kappa^{2m})^n.
\]

2) $(y,\eta)\in\mathcal Z$ but $F_\alpha^{(n-1)}(\theta\omega,y,\eta)\notin\mathcal Z$. By using the observation above
we find in view of \eqref{eq:outsidetrappedbound} and \eqref{eq:generalbound} that
\[
\frac{a_m^2(y,\eta)}{a_m^2(F_\alpha^{(n)}(\omega,y,\eta))}
\le C_\kappa^{2m},\quad\text{$\mathbb{P}$-almost surely}.
\]

3) $(y,\eta)\in\mathcal Z$ and $F_\alpha^{(n-1)}(\theta\omega,y,\eta)\in\mathcal Z$. Then \eqref{eq:generalbound}
gives
\[
\frac{a_m^2(y,\eta)}{a_m^2(F_\alpha^{(n)}(\omega,y,\eta))}\le 1,\quad\text{$\mathbb{P}$-almost surely}.
\]

Recall that by \eqref{numbertrajectories} we have that
\[
\#\{F_\alpha^{(n-1)}(\theta\omega,y,\eta)\in\mathcal Z:\alpha\in\{0,\ldots,k-1\}^{n-1}\}\le\mathcal N(\theta\omega,n-1),
\]
which implies that
\begin{equation}\label{eq:countingbound}
\#\{F_\alpha^{(n-1)}(\theta\omega,y,\eta)\in\mathcal Z:\alpha\in\{0,\ldots,k-1\}^{n}\}\le k\mathcal N(\theta\omega,n-1).
\end{equation}
For $(y,\eta)\in\mathcal Z$, we split the sum in \eqref{estimate:principalsymbolP} into cases 2 and 3 considered
above. 
Using the estimate from case 2 together with the fact that $\#\{\alpha\in\{0,\ldots,k-1\}^{n}\}=k^n$,
we find in view of the estimate from case 3 together with \eqref{eq:countingbound},
that $0< p_n(\omega,y,\eta)\le\mathcal B(\omega,n)$ for $(y,\eta)\in\mathcal Z$, where
\begin{equation}\label{bofomega}
\mathcal B(\omega,n)=\Big(\frac{k}{\lambda}\Big)^n C_\kappa^{2m}+\frac{k\mathcal N(\theta\omega,n-1)}{\lambda^n}.
\end{equation}
Next, note that the bound obtained
for $(y,\eta)\notin\mathcal Z$ (case 1) is smaller than $\mathcal B(\omega,n)$
since $(C_\kappa^{2m})^n\le C_\kappa^{2m}$. Thus
$\lVert p_n(\omega)\rVert_{L^\infty(T^\ast(\Sone))}\le\mathcal B(\omega,n)$,
which by Theorem \ref{thm:Pn} implies that $\mathbb{P}$-almost surely,
\begin{equation}\label{normclaim}
\lVert P_n(\omega)\rVert_{\mathscr L(L^2(\Sone))}\le C_{\mathrm{dim}}
\mathcal B(\omega,n)+\mathcal O_{n,m}(h^{\frac{1}{2}}),\quad h\to0,
\end{equation}
where the error term as indicated is independent of $\omega$.

We now proceed to analyze $Q_{\nu,n}(\omega)$. Note that if
$(\phantom{i},\phantom{i})_{L^2}$
denotes the scalar product on $L^2(\Sone)$, then
\begin{align*}
\lVert Q_{\nu,n}(\omega)u\rVert_{L^2}^2&
=(Q_{\nu,n}(\omega)u,Q_{\nu,n}(\omega)u)_{L^2}\\
&=(Q_{\nu,n}^\ast(\omega)Q_{\nu,n}(\omega)u,u)_{L^2}
\le \lVert Q_{\nu,n}^\ast(\omega)Q_{\nu,n}(\omega)u\rVert_{L^2}\lVert u\rVert_{L^2}
\end{align*}
and that $\lVert Q_{\nu,n}^\ast(\omega)Q_{\nu,n}(\omega)\rVert_{\mathscr L(L^2(\Sone))}
\le \lVert Q_{\nu,n}(\omega)\rVert_{\mathscr L(L^2(\Sone))}^2$. Hence we have equality. 
By virtue of \eqref{normclaim} together with the fact that $P_n(\omega)=Q_{\nu,n}^\ast(\omega)Q_{\nu,n}(\omega)$, this gives
\begin{equation}\label{normbound3}
\lVert Q_{\nu,n}(\omega)\rVert_{\mathscr L(L^2(\Sone))}\le 
(C_{\mathrm{dim}}\mathcal B(\omega,n)+\mathcal O_{n,m}(\nu^{-\frac{1}{2}}))^{\frac{1}{2}}\quad\text{as }\nu\to\infty.
\end{equation}
We now use the assumption that $f_0$ is partially captive, that is, as explained above we
apply Proposition \ref{wPC} to find
numbers $n_0\in\N$ and $\epsilon(n_0)>0$ such that \eqref{eq:choosingnzero1}
holds for any $\epsilon$ in the range
$0\le\epsilon<\epsilon(n_0)$. Recall that this implies \eqref{eq:choosingnzero2},
and consider the right-hand side of \eqref{normbound3} for $n=n_0$,
with $\mathcal B(\omega,n_0)$ defined in accordance with \eqref{bofomega}.
Recall also that the constant $C_\kappa$ in \eqref{bofomega} denotes the square root appearing
in \eqref{eq:outsidetrappedbound}, so $C_\kappa<1$.
Since the inequality in \eqref{eq:choosingnzero2} is strict, 
we can thus find a number $m_0$ such that
\begin{equation}\label{eq:rhobound1}
\lVert Q_{\nu,n_0}(\omega)\rVert_{\mathscr L(L^2(\Sone))}
\le\rho^{n_0},\quad\text{$\mathbb{P}$-almost surely},
\end{equation}
for all $m\ge m_0$ as long as $\nu\ge\nu_m$ for some sufficiently large $\nu_m$. 
For arbitrary $n\in\N$, we write $n=\ell n_0+r$ with $0\le r<n_0$, and
note that we for each $0<r<n_0$
can bound $\mathcal B(\omega,r)$ by a constant depending only on $n_0$
since $\lambda>1$ by \eqref{emin}, $k\mathcal N(\omega,r-1)< k^{n_0}$ for $0<r<n_0$ and $C_\kappa^{2m}\le 1$ for all $m\ge 0$.
Hence, (after possibly increasing $\nu_m$ if necessary to make sure that, for $0< r<n_0$, the error terms
$\mathcal O_{r,m}(\nu^{-\frac{1}{2}})$ are valid and smaller than 1 for $\nu\ge\nu_m$) there is a constant $c$,
independent of $\nu\ge\nu_m$, such that for any $0<r<n_0$ we have
$\lVert Q_{\nu,r}(\omega)\rVert_{\mathscr L(L^2(\Sone))}<c$, $\mathbb{P}$-almost surely.
We may assume that $c\ge 1$, so that if $Q_{\nu,0}(\omega)$ is understood to be the identity on $L^2(\Sone)$, then
$\lVert Q_{\nu,r}(\omega)\rVert_{\mathscr L(L^2(\Sone))}\le c$ for all $0\le r<n_0$, $\mathbb{P}$-almost surely.
Now, by \eqref{auxiliaryoperator} we have
\begin{equation}\label{eq:Qdecomposition}
Q_{\nu,n}(\omega)=Q_{\nu,n_0}(\omega)
\circ\ldots\circ Q_{\nu,n_0}(\theta^{(\ell-1)n_0}\omega)\circ Q_{\nu,r}(\theta^{\ell n_0}\omega).
\end{equation}
Since $Q_{\nu,n}(\omega):L^2(\Sone)\to L^2(\Sone)$ is unitarily
equivalent to $M_{\nu,n}(\omega):H_\nu^{-m}(\Sone)\to H_\nu^{-m}(\Sone)$,
we have $\lVert M_{\nu,n}(\omega)\rVert_{\mathscr L(H_\nu^{-m}(\Sone))}=\lVert Q_{\nu,n}(\omega)\rVert_{\mathscr L(L^2(\Sone))}$.
Thus, by \eqref{eq:rhobound1} and \eqref{eq:Qdecomposition} we have
\[
\lVert M_{\nu,n}(\omega)\rVert_{\mathscr L(H_\nu^{-m})} \le
(\rho^{n_0})^\ell \lVert Q_{\nu,r}(\theta^{\ell n_0}\omega)\rVert_{\mathscr L(L^2)}
\le \frac{c}{\rho^r}\rho^{\ell n_0+r}\le \lambda^{n_0/2}c\rho^n,
\]
where the last inequality follows from the fact that $\rho>\lambda^{-\frac{1}{2}}$ and $\lambda>1$. Since $n_0$ only depends
on $\rho$, we conclude that $\lambda^{n_0/2}c$
satisfies the properties required of the constant $c_0$ in the statement of Theorem \ref{thm:spectralgap}.
We set $\nu_0=\nu_0(m)=\nu_m$.
Since $H_\nu^m(\Sone)$ is the dual of $H_\nu^{-m}(\Sone)$ with respect to the usual scalar product on $L^2(\Sone)$,
a standard duality argument yields the statement concerning the norm $\lVert M_{\nu,n}^\ast(\omega)\rVert_{\mathscr L(H_\nu^{m}(\Sone))}$
of the adjoint of $M_{\nu,n}(\omega):H_\nu^{-m}(\Sone)\to H_\nu^{-m}(\Sone)$
with respect to the $L^2$ pairing.
This completes the proof of Theorem \ref{thm:spectralgap}.

\begin{rmk}
If the error term in Theorem \ref{thm:Pn} is not shown to be independent of $\omega$,
then the previous proof yields a significantly weaker statement compared to Theorem \ref{thm:spectralgap}.
In fact, \eqref{eq:rhobound1}
would then hold for all $m\ge m_0$ as long as $\nu\ge\nu_m(\omega)$ for some sufficiently large 
$\nu_m(\omega)$. Combined with \eqref{eq:Qdecomposition}, this leads to
\[
\lVert M_{\nu,n}(\omega)\rVert_{\mathscr L(H_\nu^{-m})} \le\lambda^{n_0/2}c\rho^n
\]
for all $m\ge m_0$ and $\nu\ge\tilde\nu_0(m)$, where
\[
\tilde\nu_0(m)=\max{(\nu_m(\omega),\nu_m(\theta^{n_0}\omega),\ldots,\nu_m(\theta^{(\ell-1)n_0}\omega))}
\]
and $\ell$ is determined by $n=\ell n_0+r$. Hence, this $\tilde\nu_0$ would depend not only
on $\omega$ but also on time $n\in\N^\ast$.
\end{rmk}

\section{Analysis of the partial captivity condition}\label{section:analysisofpc}

The main purpose of this section is to prove Proposition \ref{wPC}. The key ingredient is a
perturbation lemma which compares the ``captivity'' at different noise levels,
see Proposition \ref{prop:differentnoiselevels}. 
We also include a brief discussion on the connection between
partial captivity and the transversality condition mentioned in the introduction.

\subsection{Proof of Proposition \ref{wPC}}\label{app:wPC}
We will assume that the hypotheses of Theorem \ref{thm:spectralgap} are in force,
and use notation introduced in Section \ref{section:spectralgap}. 
In particular, we assume that $\epsilon$ is in the range
$0\leq \epsilon <\epsilon _0$ for some $\epsilon_0$ such that \eqref{emin} holds.
In the case $\epsilon=0$, we identify $f_0$ with the constant mapping $\Omega \ni \omega \mapsto f_0$. 
We shall assume that $\epsilon_0$ is chosen sufficiently small so that
\begin{equation}\label{eq:assumptionC1close}
\esssup_\omega{\sup_x{\lvert E_\epsilon'(\omega,x)-E_0'(x)\rvert}}<1,\quad 0\le\epsilon<\epsilon_0.
\end{equation}
Moreover, we let $\kappa$ be a fixed number in the range $1<\kappa<\lambda$.
In accordance with Lemma \ref{lemma13}, we let  
$R_\kappa$ be a number such that $R_\kappa\ge C_\tau(\lambda-\kappa)^{-1}$,
where $C_\tau$ is a constant depending only on $\tau_0$ and $\epsilon_0$ such that
\begin{equation}\label{eq:tauconstant}
\esssup_\omega{\sup_x{\lvert\tau_\epsilon'(\omega,x)\rvert}}\le C_\tau,\quad 0\le\epsilon<\epsilon_0.
\end{equation}
For future purposes we remark that we by construction have
\begin{equation}\label{eq:choiceofkappagives}
R_\kappa-C_\tau\sum_{j=1}^\infty\lambda^{-j}>0,
\end{equation}
which can be checked by a straightforward calculation.
Recall also that $\mathcal Z$ denotes the set \eqref{def:Z} defined in terms of the number $R_\kappa$ chosen above.

We first adapt the ideas in Faure~\cite{Faure}*{Appendix B}
and describe partial captivity
in terms of the random dynamics \eqref{canonicalmap1}, lifted on the cover $\mathbb{R} ^2$.
To this end, let $g_\epsilon(\omega)$, $E_\epsilon(\omega)$ and $\tau_\epsilon(\omega)$ 
be defined through the perturbation scheme $\{f_\epsilon\}_{\epsilon>0}$
as in \eqref{eq:perturbedsystem}. Thus, for fixed $\epsilon>0$ and $\omega\in\Omega$ we have that
$g_\epsilon(\omega): \Sone \rightarrow  \Sone$ is a ${\Ci}^\infty$ diffeomorphism,
$E_\epsilon(\omega)$ is a $k$ valued map on $\Sone$ given by $E_\epsilon(\omega):x\mapsto E_\epsilon(\omega,x)$, while
$\tau_\epsilon(\omega):\Sone\to\R$ is given by $\tau_\epsilon(\omega):x\mapsto \tau_\epsilon(\omega,x)$.
We identify $\Sone$ with the fundamental domain
\[
\Sone\simeq\{x:0\le x<1\}\subset\R,
\]
and recall that an orientation preserving
${\Ci}^\infty$ diffeomorphism $g_\epsilon(\omega): \Sone \rightarrow  \Sone$ can be viewed as a smooth function $g_\epsilon(\omega):\R\to\R$
with the property that $g_\epsilon(\omega,x+1)=g_\epsilon(\omega,x)+1$ for all $x\in\R$, and with an inverse $g_\epsilon^{-1}(\omega)$
enjoying the same property. Similarly, $E_\epsilon(\omega)$ can be identified with the smooth function $E_\epsilon(\omega):\R\to\R$
given by $E_\epsilon(\omega,x)=kg_\epsilon(\omega,x)$ for $x\in\R$. 
Thus $E_\epsilon(\omega,x+\ell)=E_\epsilon(\omega,x)+k\ell$ for all $\ell\in\Z$ and all $x\in\R$,
and it follows that $E_\epsilon'(\omega):\R\to\R$ is a smooth periodic function with period 1. Similarly,
any smooth function $\tau_\epsilon(\omega):\Sone\to\R$ can be identified with a periodic function
$\tau_\epsilon(\omega):\R\to\R$ with period 1. By \eqref{convinC1} we have
\begin{equation}\label{eq:assumptionCrclose}
\esssup_\omega{d_{{\Ci}^\infty}( E_\epsilon(\omega),E_0)}\to0,
\quad\esssup_\omega{d_{{\Ci}^\infty}( \tau_\epsilon(\omega),\tau_0)}\to0,\quad \epsilon\to0.
\end{equation}
Here $E_0$ and $\tau_0$ are the smooth functions given by the deterministic dynamical system $f_0$ defined by \eqref{eq:unperturbedsystem},
lifted on $\R$.
Note also that the map $E_\epsilon(\omega):\R\to\R$ is invertible, with inverse 
$E_\epsilon^{-1}(\omega,x)=g_\epsilon^{-1}(\omega,x/k)$ for $x\in\R$.
$E_\epsilon^{-1}(\omega)$ then satisfies $E_\epsilon^{-1}(\omega,x+k)=E_\epsilon^{-1}(\omega,x)+1$, so $E_\epsilon^{-1}(\omega)$ cannot be
identified with a map $\Sone\to \Sone$. ($E_\epsilon(\omega):\R\to\R$ is a diffeomorphism on $\R$,
but not on the fundamental domain $[0,1)$.) However, each inverse branch of 
$E_\epsilon(\omega):\Sone\to \Sone$ can be identified 
with a map $\R\ni x\mapsto E_\epsilon^{-1}(\omega,x+j)$ for some $0\le j\le k-1$.
For convenience, we shall throughout this section reserve the notation $G_\epsilon(\omega)$ for the inverse of 
$E_\epsilon(\omega):\R\to\R$,
and permit us to switch between $G_\epsilon(\omega)$ and $E^{-1}_{\epsilon}(\omega)$
as befitting the situation.

In view of the previous discussion, we now introduce the lifted dynamics on the cover $\R^2$,
that is, the diffeomorphism on $\R^2$ which 
for each $\omega \in \Omega$ is given by
\begin{align}\label{lifted}
\tilde{F} (\epsilon ;\omega) :\binom{y}{\eta} \mapsto \binom{x}{\xi} =
\left( \! \! \begin{array}{c} E_{\epsilon}^{-1}(\omega,y) \\ 
E_\epsilon '(\omega,x)\eta +\tau_\epsilon '(\omega,x) \end{array} \! \! \right).
\end{align}
(We will distinguish $\tilde{F} (\epsilon ;\omega)$ from
the random dynamics \eqref{canonicalmap1} for reasons of comparison, see 
\eqref{eq:setequality} below.) 
Let  $\N \times \Omega \times \R ^2\ni (n,\omega ,y,\eta) \mapsto \tilde{F} ^{(n)} (\epsilon ;\omega ,y,\eta)$ be the backward 
cocycle induced by $\tilde{F} (\epsilon ;\cdot):\Omega \to {\Ci}^{\infty} (\R^2,\R^2)$, that is,
$\tilde F^{(0)}(\epsilon ;\omega) =\mathrm{Id} _{\R^2}$
for each $\omega \in \Omega$ and
\begin{equation}\label{eq:lifteddynamics}
\tilde{F} ^{(n)}(\epsilon;\omega) =\tilde{F}(\epsilon ;\omega) \circ \tilde{F} 
(\epsilon ;\theta \omega)\circ \cdots \circ \tilde{F} (\epsilon ;\theta ^{n-1}\omega),
\quad n\ge1.
\end{equation}
Let $\N \times \Omega \times \R \ni (n,\omega ,y) \mapsto G_\epsilon^{(n)}(\omega ,y)$ 
be the backward cocycle induced by $G_\epsilon(\cdot):\Omega \to {\Ci}^{\infty}(\R ,\R)$, 
defined in the same manner. 
(Note that if we omit the subscript $\epsilon$, this is in accordance with the notation used in \eqref{eq:introducingG}
where the subscript $\alpha$ refers to translation.)
Then $G_\epsilon^{(n)}(\omega ,y)$
is the first component of $\tilde F^{(n)}(\epsilon;\omega ,y,\eta)$; the second component is
\begin{align}\label{eq:Ftildetimensecondargument}
\frac{\eta}{dG_\epsilon^{(n)}(\omega,y)/dy}+\sum_{j=0}^{n-1}\frac{\tau_\epsilon'(\theta^j\omega,G_\epsilon^{(n-j)}(\theta^j\omega ,y))}
{dG_\epsilon^{(j)}(\omega,G_\epsilon^{(n-j)}(\theta^j\omega ,y))/dy},
\end{align}
where $dG_\epsilon^{(0)}(\omega,\cdot)/dy=1$, and
\begin{equation}\label{eq:Ftildetimensecondargument2}
(dG_\epsilon^{(j)}(\omega,G_\epsilon^{(n-j)}(\theta^j\omega ,y))/dy)^{-1}
=\prod_{\ell=0}^{j-1}E_\epsilon'(\theta^\ell\omega,G_\epsilon^{(n-\ell)}(\theta^\ell\omega ,y)),
\quad 1\le j\le n.
\end{equation}

For each $0\le\epsilon<\epsilon_0$ we now solve the equation
\begin{equation}\label{eq:cohomological}
S_\epsilon(\omega,G_\epsilon(\omega,x))=E_\epsilon'(\omega,G_\epsilon(\omega,x))
S_\epsilon(\theta\omega,x)+\tau_\epsilon'(\omega,G_\epsilon(\omega,x))
\end{equation}
for $S_\epsilon:\Omega\to {\Ci}^\infty(\R)$.
(When $\epsilon=0$ we obtain a constant map $S_0(\omega)\equiv S_0$.)
Note that $S_\epsilon$ is a function such that
the point $(x,S_\epsilon(\theta\omega,x))$ is mapped to $(x',S_\epsilon(\omega,x'))$ under $\tilde F(\epsilon;\omega)$.
More generally we have 
\begin{equation}\label{eq:invariantgraphoftildeF}
\tilde F^{(n)}(\epsilon;\omega)(x,S_\epsilon(\theta^n\omega,x))=(x',S_\epsilon(\omega,x')),
\quad x'=G_\epsilon^{(n)}(\omega,x).
\end{equation}
In the deterministic case, the graph of $S_\epsilon(\omega)$ 
coincides with the stable manifold of $\tilde F(\epsilon ;\omega)$; see Faure~\cite{Faure}*{Appendix B.1}.
To solve \eqref{eq:cohomological}, we observe that if $S_\epsilon$ is a solution, then replacing $\omega$ with $\theta^{-1}\omega$
gives
\begin{equation}
S_\epsilon(\omega,x)=\frac{S_\epsilon(\theta^{-1}\omega,G_\epsilon(\theta^{-1}\omega,x))}
{E_\epsilon'(\theta^{-1}\omega,G_\epsilon(\theta^{-1}\omega,x))}
-\frac{\tau_\epsilon'(\theta^{-1}\omega,G_\epsilon(\theta^{-1}\omega,x))}{E_\epsilon'(\theta^{-1}\omega,G_\epsilon(\theta^{-1}\omega,x))}.
\end{equation}
Repeated use of this identity yields the formula
\begin{equation}\label{eq:solutiontocohomological}
S_\epsilon(\omega,x)=-\sum_{j=1}^\infty
\tau_\epsilon'(\theta^{-j}\omega,G_\epsilon^{(j)}(\theta^{-j}\omega,x))\cdot\frac{dG_\epsilon^{(j)}(\theta^{-j}\omega,x)}{dx},
\end{equation}
where $G_\epsilon^{(j)}(\theta^{-j}\omega , x)=E_\epsilon^{-1}(\theta^{-j}\omega)\circ\cdots\circ E_\epsilon^{-1}(\theta^{-1}\omega)(x)$ and
\begin{equation}\label{eq:ee1}
\frac{dG_\epsilon^{(j)}(\theta^{-j}\omega , x)}{dx} =\prod _{\ell=1}^{j}
\frac{1}{E_{\epsilon}'(\theta ^{-\ell}\omega , G_\epsilon^{(\ell)}(\theta^{-\ell}\omega,x))}.
\end{equation}
Since $\theta$ is $\mathbb{P}$-preserving, the series \eqref{eq:solutiontocohomological}
$\mathbb{P}$-almost surely converges uniformly in $x$ by virtue of \eqref{eq:ee1} and \eqref{emin}.
Now $G_\epsilon^{(\ell)}(\theta^{-\ell}\omega)\circ G_\epsilon(\omega)(x)=G_\epsilon^{(\ell+1)}(\theta^{-\ell}\omega,x)$,
and using this observation it is straightforward to check that \eqref{eq:solutiontocohomological}
solves \eqref{eq:cohomological}.

For the proof of Proposition \ref{wPC}, we need a modified version of 
a result by Faure~\cite{Faure}*{Proposition 17}, proved using similar techniques,
see Proposition \ref{prop:PCequiv1} below. The strategy includes a comparison between
the set \eqref{treeofrandomtrajectories}
of random trajectories of the dynamics \eqref{canonicalmap1} on $T^\ast(\Sone)$
with a related set defined in terms of the backward cocycle
induced by $\tilde{F} (\epsilon ;\cdot)$ on $\R^2$.
To this end, let $\mathcal P:\R^2\to T^\ast(\Sone)=\Sone\times\R$ be the canonical projection.
Given a sequence $\alpha=(\alpha_n,\ldots,\alpha_1)\in\{0,\ldots,k-1\}^n$,
let $\bar\alpha$ be the number $\bar\alpha=\sum_{j=1}^n\alpha_j k^{j-1}$.
We can think of $\alpha$ as being the number $\bar\alpha$ written in base $k$,
and this gives a bijection between the sets $\{0,\ldots,k-1\}^n$ and $\{0,\ldots,k^n-1\}$.
For fixed $(y,\eta)\in\R^2$, it is straightforward to check that as $\alpha$
varies in $\{0,\ldots,k-1\}^n$ we obtain the set equality
\begin{equation}\label{eq:setequality}
F^{(n)}(\epsilon;\omega,\mathcal P(y,\eta))
=\{\mathcal P(\tilde F^{(n)}(\epsilon;\omega,y+\bar\alpha,\eta)):0\le\bar\alpha\le k^n-1\},
\end{equation}
where the left-hand side is the set defined by \eqref{treeofrandomtrajectories}.
Hence, if $\mathcal N(\epsilon;\omega,n)$ is the number given by \eqref{numbertrajectories}
of random trajectories that do not escape outside $\mathcal Z=\Sone\times[-R_\kappa,R_\kappa]$ before time $n$,
then $\mathcal N(\epsilon;\omega,n)$ coincides with
the maximum number of $\bar\alpha\in\{0,\ldots,k^n-1\}$ such that
the second component of $\tilde F^{(n)}(\epsilon;\omega,y+\bar\alpha,\eta)$
has length bounded by $R_\kappa$.
The following notation will prove useful.

\begin{dfn}\label{def:tildeAandbarA}
Let $n\in \N^\ast$ and $R >0$. For each $(y,\eta)\in\R^2$, let $\widetilde{\mathcal{A}}_{R}(\epsilon;\omega,n)(y,\eta)$
be the set of $0\leq \bar\alpha \leq k^n-1$ such that
\begin{equation}\label{eq:defofalphatilde}
\lvert\eta-S_\epsilon(\theta^n\omega,y+\bar\alpha)\rvert
\leq R\cdot \frac{dG_\epsilon^{(n)}(\omega,y+\bar\alpha)}{dy}.
\end{equation}
The maximum cardinality of $\widetilde{\mathcal{A}}_{R}(\epsilon;\omega,n)(y,\eta)$
as $(y,\eta)$ varies in $\R^2$ will be denoted by $\widetilde{\mathcal{N}} _{R} (\epsilon ;\omega ,n)$,
that is,
\[
\widetilde{\mathcal{N}}_{R} (\epsilon ;\omega ,n) 
=\sup _{y,\eta} \# \widetilde{\mathcal{A}} _{R}(\epsilon ;\omega ,n)(y,\eta).
\]
\end{dfn}

When $\epsilon=0$ we shall write $\widetilde{\mathcal{A}}_{R}(0;n)(y,\eta)$
and $\widetilde{\mathcal N}_R(0;n)$ for the corresponding set and number, respectively.
By Faure~\cite{Faure}*{Proposition 17} it then follows that the unperturbed dynamical system
$f_0$ is partially captive if and only if
\begin{equation}\label{eq:newdefofpartiallycaptive}
\lim_{n\to\infty}\frac{1}{n}\log{\widetilde{\mathcal N}_R(0;n)}=0
\end{equation}
for all $R>0$. We shall prove Proposition \ref{wPC} using \eqref{eq:newdefofpartiallycaptive}
together with the following comparison results.

\begin{prop}\label{prop:PCequiv1}
Let $0\le\epsilon<\epsilon_0$. Let $C_\tau$ be given by \eqref{eq:tauconstant}, and set
$C_1=C_\tau\sum_{j=1}^\infty\lambda^{-j}$.
For each $n\in \N^\ast$ we $\mathbb{P}$-almost surely have 
\begin{equation*}
\widetilde{\mathcal{N} }_{R_\kappa-C_1}(\epsilon ;\omega, n)
\leq \mathcal{N}(\epsilon ;\omega ,n)\leq \widetilde{\mathcal{N}} _{R_\kappa +C_1}(\epsilon ;\omega,n),
\end{equation*}
where $\mathcal N (\epsilon ;\omega ,n)=\mathcal N_{R_\kappa}(\epsilon ;\omega ,n)$ is defined in \eqref{numbertrajectories}.
\end{prop}

\begin{proof}
We adapt the ideas used in the proof of Faure~\cite{Faure}*{Proposition 17}.
First note that
\eqref{eq:solutiontocohomological} and \eqref{eq:ee1} implies that
\begin{equation}\label{eq:propertiesofC1}
\esssup_\omega{\sup_x{\lvert S_\epsilon(\omega,x)\rvert}}\le C_1,
\end{equation}
and this holds independently of $0\le\epsilon<\epsilon_0$ in view of \eqref{emin}.
We now fix $\epsilon$. 
Given $(y,\eta ) \in \mathbb{R} ^2$, $n\in \N^\ast$ and $0\leq \bar\alpha  \leq k^n-1$, let 
$(x_{\bar\alpha}(\omega) ,\xi _{\bar\alpha}(\omega))=\tilde{F} ^{(n)}(\epsilon;\omega)(y+\bar\alpha,\eta )$.
Then \eqref{eq:setequality} implies that
\begin{equation}\label{eq:samecardinality}
\#\{\bar\alpha:\lvert \xi _{\bar\alpha} (\omega ) \rvert \leq R_\kappa\}
=\#\{F_\alpha^{(n)}(\epsilon;\omega,y,\eta)\in\mathcal Z:\alpha\in\{0,\ldots,k-1\}^n\}.
\end{equation}
Now note that the second component of $\tilde{F} ^{(n)} (\epsilon;\omega) (y+\bar\alpha,S_\epsilon (\theta ^n\omega,y+\bar\alpha))$
is equal to $S_\epsilon (\omega,x_{\bar\alpha}(\omega))$ according to \eqref{eq:invariantgraphoftildeF}.
In view of \eqref{lifted}, a straightforward calculation then shows that
\begin{equation*}
\xi _{\bar\alpha}(\omega)-S_\epsilon (\omega,x_{\bar\alpha}(\omega))
=\prod _{j=0}^{n-1}E_\epsilon' (\theta ^j\omega ,G_\epsilon^{(n-j)}(\omega ,y+\bar\alpha))\cdot 
(\eta -S_\epsilon (\theta ^{n}\omega,y+\bar\alpha)).
\end{equation*}
Therefore, it follows from \eqref{eq:ee1} that
$\lvert \xi _{\bar\alpha} (\omega ) \rvert \leq R_\kappa$ if and only if 
\begin{equation*}
\bigg\lvert S_\epsilon (\omega,x_{\bar\alpha}(\omega))\cdot
 \frac{dG_\epsilon^{(n)}(\omega ,y+\bar\alpha)}{dy}
 +\eta -S_\epsilon (\theta ^n\omega,y+\bar\alpha) \bigg\rvert
 \leq R_\kappa\cdot \frac{dG_\epsilon^{(n)}(\omega ,y+\bar\alpha)}{dy}.
\end{equation*}
Hence, if $\bar\alpha\in\widetilde{\mathcal A}_{R_\kappa-C_1}(\epsilon;\omega,n)(y,\eta)$
then $\lvert\xi_{\bar\alpha}(\omega)\rvert\le R_\kappa$ by the triangle inequality and \eqref{eq:propertiesofC1}.
In view of \eqref{eq:samecardinality} together with the definitions of
$\widetilde{\mathcal{N}}_{R_\kappa- C_1}(\epsilon;\omega,n)$ and $\mathcal N(\epsilon;\omega,n)$,
we conclude that $\widetilde{\mathcal{N} }_{R_\kappa-C_1}(\epsilon ;\omega, n)
\leq \mathcal{N}(\epsilon ;\omega ,n)$. The other inequality is proved in a similar fashion,
which completes the proof.
\end{proof}

\begin{prop}\label{prop:differentnoiselevels}
Let $\varrho>0$. Then for each $R>\varrho$ and $n\in \N^\ast$ there is an
$\epsilon(n)>0$ depending also on $R$ and $\varrho$ such that for all $0\le\epsilon<\epsilon(n)$
and $\mathbb{P}$-almost every $\omega$ we have 
\begin{equation}
\widetilde{\mathcal{N}}_{R-\varrho}(0;n)
\leq\widetilde{\mathcal{N}}_{R}(\epsilon;\omega,n)
\leq\widetilde{\mathcal{N}}_{R+\varrho}(0;n).
\end{equation}
\end{prop}

For the proof we need a technical lemma.

\begin{lem}\label{lem:technical1}
Let $0\le\epsilon<\epsilon_0$.
For every $n\ge1$
we have
\[
\lvert G_0^{(n)}(x)-G_\epsilon^{(n)}(\omega,x)\rvert
\le\sum_{j=1}^n\lambda_0^{-j}\sup_x{\lvert E_\epsilon(\theta^{j-1}\omega,x)-E_0(x)\rvert}.
\] 
Moreover,
\[
\esssup_\omega{\sup_x{\lvert G_0^{(n)}(x)-G_\epsilon^{(n)}(\omega,x)\rvert}}
\le(\lambda_0-1)^{-1}\esssup_\omega{\sup_x{\lvert E_\epsilon(\omega,x)-E_0(x)\rvert}},
\]
where the right-hand side tends to zero as $\epsilon\to0$.
\end{lem}

\begin{proof}
We begin by proving the first claim.
When $n=1$ we need to show that
\begin{equation}\label{eq:inductionstep1}
\lvert E_0^{-1}(x)-E_\epsilon^{-1}(\omega,x)\rvert
\le\lambda_0^{-1}\sup_x{\lvert E_\epsilon(\omega,x)-E_0(x)\rvert}.
\end{equation}
Write $y=E_\epsilon^{-1}(\omega,x)$
and apply Taylor's formula to $E_0^{-1}(E_\epsilon(\omega,y))$ about the point $E_0(y)$.
Since $\min_x{E_0'(x)}>\lambda_0$, the estimate follows. We proceed by induction,
so assume that the result holds for some $n\ge1$. The triangle inequality gives
\begin{align}
\lvert G_0^{(n+1)}(x)-G_\epsilon^{(n+1)}(\omega,x)\rvert
& \le \lvert E_0^{-1}\circ G_0^{(n)}(x)-E_0^{-1}\circ G_\epsilon^{(n)}(\theta\omega)(x)\rvert
\\ & \quad+\lvert E_0^{-1}\circ G_\epsilon^{(n)}(\theta\omega)(x)
 -E_\epsilon^{-1}(\omega)\circ G_\epsilon^{(n)}(\theta\omega)(x)\rvert.
\end{align}
By \eqref{eq:inductionstep1}, the second term on the right is bounded by
$\lambda_0^{-1}\lVert E_\epsilon(\omega)-E_0\rVert_{L^\infty}$.
Using Taylor's formula to estimate the first term on the right we get
\begin{align}
\lvert E_0^{-1}\circ G_0^{(n)}(x)-E_0^{-1}\circ G_\epsilon^{(n)}(\theta\omega)(x)\rvert
&\le\lambda_0^{-1}\lvert G_0^{(n)}(x)-G_\epsilon^{(n)}(\theta\omega,x)\rvert
\\ &\le\sum_{j=2}^{n+1}\lambda_0^{-j}\sup_x{\lvert E_\epsilon(\theta^{j-1}\omega,x)-E_0(x)\rvert},
\end{align}
where we invoked the induction hypothesis to obtain the last inequality.
Combining these estimates proves the first claim.
Since $\theta$ is $\mathbb{P}$-preserving, 
the second claim now follows by taking the essential supremum in $\omega$
and noting that $\sum_{j=1}^n\lambda_0^{-j}<(\lambda_0-1)^{-1}$
for all $n$ since $1<\lambda_0$. The proof is completed by invoking \eqref{eq:assumptionCrclose}.
\end{proof}

\begin{proof}[Proof of Proposition \ref{prop:differentnoiselevels}]
We begin with the second inequality.
Suppose that $\bar\alpha\in\widetilde{\mathcal{A}}_R(\epsilon;\omega,n)(y,\eta)$. By
Definition \ref{def:tildeAandbarA} and the triangle inequality, we then have
\[
\lvert\eta-S_0(y+\bar\alpha)\rvert\le R\cdot\frac{dG_0^{(n)}(y+\bar\alpha)}{dy}
+RI_n(\epsilon)+J_n(\epsilon),
\]
where $I_n(\epsilon)\equiv I_n(\epsilon;\omega,y,\bar\alpha)$
and $J_n(\epsilon)\equiv J_n(\epsilon;\omega,y,\bar\alpha)$ are given by
\begin{align}
I_n(\epsilon)&=\bigg\lvert \frac{dG_\epsilon^{(n)}(\omega,y+\bar\alpha)}{dy}
-\frac{dG_0^{(n)}(y+\bar\alpha)}{dy}\bigg\rvert,\\
J_n(\epsilon)&=\lvert S_\epsilon(\theta^n\omega,y+\bar\alpha)-S_0(y+\bar\alpha)\rvert.
\end{align}
Introduce a positive number $L(n)$ (independent of $\epsilon$ and $\omega$) defined by
\[
L(n)=\inf_{0\leq \epsilon <\epsilon_0}{
\essinf_{\omega}{\inf_{x}{ \frac{dG_\epsilon^{(n)}(\omega , x)}{dx}}}}. 
\]
It is clear that $L(n)\to0$ as $n\to\infty$.
To see that $L(n)>0$, note that since $\theta$ is $\mathbb{P}$-preserving, we have
for $\mathbb{P}$-almost every $\omega$ that
\begin{align}\label{eq:controlofG(n)}
\frac{dG_\epsilon^{(n)}(\omega , x)}{dx}&>\prod_{j=0}^{n-1}\frac{1}{E_0'(G_\epsilon^{(n-j)}(\theta^j\omega,x))
+\sup_x{\lvert E_\epsilon'(\theta^j\omega,x)-E_0'(x) \rvert}}\\
&>(\sup_x{\lvert E_0'(x)\rvert}+1)^{-n}
\end{align}
by virtue of \eqref{eq:assumptionC1close}.

Assume for the moment that $I_n(\epsilon)$ and $J_n(\epsilon)$ tend to zero as $\epsilon\to0$,
uniformly with respect to $y$ and $\bar\alpha$ and $\mathbb{P}$-almost every $\omega$.
Then there is an $\epsilon(n)$, depending on $n$ and also on $R$ and $\varrho$, such that
$RI_n(\epsilon)+J_n(\epsilon)<\varrho L(n)$ for all $0\le\epsilon<\epsilon(n)$,
$\mathbb{P}$-almost surely.
By the definition of $L(n)$, this $\mathbb{P}$-almost surely gives
\[
\lvert\eta-S_0(y+\bar\alpha)\rvert\le R\cdot\frac{dG_0^{(n)}(y+\bar\alpha)}{dy}
+\varrho L(n)<(R+\varrho)\frac{dG_0^{(n)}(y+\bar\alpha)}{dy},
\]
that is, $\bar\alpha\in\widetilde{\mathcal{A}}_{R+\varrho}(0;n)(y,\eta)$.
Hence $\#\widetilde{\mathcal{A}}_{R}(\epsilon;\omega,n)(y,\eta)
\le\#\widetilde{\mathcal{A}}_{R+\varrho}(0;n)(y,\eta)$, $\mathbb{P}$-almost surely, and taking the supremum over
$(y,\eta)\in\R^2$ shows that we $\mathbb{P}$-almost surely have 
$\widetilde{\mathcal{N}}_{R}(\epsilon;\omega,n)\leq\widetilde{\mathcal{N}}_{R+\varrho}(0;n)$.

To prove the first inequality, suppose that $\bar\alpha\in\widetilde{\mathcal{A}}_{R-\varrho}(0;n)(y,\eta)$. Then
\[
\lvert \eta-S_\epsilon(\theta^n\omega,y+\bar\alpha)\rvert
 \le (R-\varrho)\frac{dG_\epsilon^{(n)}(\omega,y+\bar\alpha)}{dy}
+RI_n(\epsilon)+J_n(\epsilon),
\]
where we also used the fact that $R-\varrho\le R$ to estimate the coefficient of $I_n$. 
Hence, a repetition of the first part of the proof
shows that $\bar\alpha\in\widetilde{\mathcal{A}}_R(\epsilon;\omega,n)(y,\eta)$,
$\mathbb{P}$-almost surely, and the desired inequality follows.

We now prove that $I_n(\epsilon)$ and $J_n(\epsilon)$ tend to zero as $\epsilon\to0$,
uniformly with respect to $y$ and $\bar\alpha$ and $\mathbb{P}$-almost every $\omega$.
Since 
\[
I_n(\epsilon)\le\sup_x{
\bigg\lvert \frac{dG_\epsilon^{(n)}(\omega,x)}{dx}
-\frac{dG_0^{(n)}(x)}{dx}\bigg\rvert}
\]
the claim for $I_n(\epsilon)$ follows by \eqref{eq:assumptionCrclose}
and Lemma \ref{lem:technical1}, if $n\in\N^\ast$ is fixed. For future purposes we also note that
if $0\le\epsilon<\epsilon_0$ then
\begin{equation}\label{eq:controlastimetoinfinity}
I_n(\epsilon)\le \lambda^{-n}+\lambda_0^{-n}\le 2\lambda^{-n}, \quad
\text{$\mathbb{P}$-almost surely}.
\end{equation}

For $J_n(\epsilon)$, note that
$J_n(\epsilon)\le \sup_x{\lvert S_0(x)-S_\epsilon(\theta^n\omega,x)\rvert}$. Recalling 
\eqref{eq:solutiontocohomological}, a straightforward estimation then gives
\begin{align}
J_n(\epsilon) &\le 
  \sup_x\sum_{j=1}^\infty \bigg\{ 
  \lvert\tau_0'(G_0^{(j)}(x))\rvert
  \bigg\lvert\frac{dG_0^{(j)}(x)}{dx}-\frac{dG_\epsilon^{(j)}(\theta^{n-j}\omega,x)}{dx} \bigg\rvert\\
  &\qquad\qquad\quad+
  \bigg\lvert\frac{dG_\epsilon^{(j)}(\theta^{n-j}\omega,x)}{dx}\bigg\rvert
  \bigg( \lvert\tau_0'(G_0^{(j)}(x))-\tau_0'(G_\epsilon^{(j)}(\theta^{n-j}\omega,x))\rvert\\
  &\qquad\qquad\quad+\lvert\tau_0'(G_\epsilon^{(j)}(\theta^{n-j}\omega,x))
  -\tau_\epsilon'(\theta^{n-j}\omega,G_\epsilon^{(j)}(\theta^{n-j}\omega,x))\rvert \bigg)
  \bigg\},
\end{align}
which implies that
\begin{align}
J_n(\epsilon)&\le
 \sum_{j=1}^\infty  \bigg(
\lVert\tau_0'\rVert_{L^\infty} \sup_x
\bigg\lvert\frac{dG_0^{(j)}(x)}{dx}-\frac{dG_\epsilon^{(j)}(\theta^{n-j}\omega,x)}{dx} \bigg\rvert\bigg)\\
&\quad+\sum_{j=1}^\infty 
\lambda^{-j} 
\bigg( \lVert\tau_0''\rVert_{L^\infty}\sup_x{\lvert G_0^{(j)}(x)-G_\epsilon^{(j)}(\theta^{n-j}\omega,x)\rvert}\\
&\qquad\qquad\qquad\qquad+\sup_x{\lvert\tau_0'(x)-\tau_\epsilon'(\theta^{n-j}\omega,x)\rvert} \bigg).
\end{align}
Since $\theta$ is $\mathbb{P}$-preserving, the first series on the right is $\mathbb{P}$-almost surely bounded by
\begin{equation}\label{eq:smallordostep1}
\lVert\tau_0'\rVert_{L^\infty}\bigg(\sum_{j=1}^{N-1}  \esssup_\omega\sup_xI_j(\epsilon)
+2\sum_{j=N}^\infty  \lambda^{-j}\bigg)
\end{equation}
according to \eqref{eq:controlastimetoinfinity},
where $\sum_{j=N}^\infty  \lambda^{-j}\to0$ as $N\to\infty$ since $\lambda>1$. For
$1\le j<N$ we can argue as above and conclude that $\esssup_\omega\sup_xI_j(\epsilon)=o_j(1)$
as $\epsilon\to0$. Hence, \eqref{eq:smallordostep1} is $o(1)$ as $\epsilon\to0$.
If we use Lemma \ref{lem:technical1} and \eqref{eq:assumptionCrclose},
respectively, 
to estimate the two types of terms making up the second series, similar arguments show that
also the second series is $\mathbb{P}$-almost surely $o(1)$ as $\epsilon\to0$.
Hence, $J_n(\epsilon)$ tends to zero as $\epsilon\to0$,
uniformly with respect to $y$ and $\bar\alpha$ and $\mathbb{P}$-almost every $\omega$ (the convergence is even uniform with
respect to time $n$). This completes the proof.
\end{proof}

\begin{proof}[Proof of Proposition \ref{wPC}]
Recall that $C_1=C_\tau\sum_{j=1}^\infty\lambda^{-j}$,
where $C_\tau$ is given by \eqref{eq:tauconstant}.
By \eqref{eq:choiceofkappagives} we then have $R_\kappa-C_1>0$. For each $n\in\N^\ast$
we can therefore use
Propositions \ref{prop:PCequiv1} and \ref{prop:differentnoiselevels}
with $\varrho=(R_\kappa-C_1)/2$ to find an $\epsilon(n)>0$ such that for all $0\le\epsilon<\epsilon(n)$
and $\mathbb{P}$-almost every $\omega$ we have
\begin{equation}\label{eq:PCstep1}
\widetilde{\mathcal N}_{R_\kappa-C_1-\varrho}(0;n)\le
\mathcal N_{R_\kappa}(\epsilon;\omega,n)\le
\widetilde{\mathcal N}_{R_\kappa+C_1+\varrho}(0;n),
\end{equation}
where $R_\kappa-C_1-\varrho=(R_\kappa-C_1)/2>0$.

Let $\delta>0$. By \eqref{eq:newdefofpartiallycaptive}
we can find an $n_0$ such that $\widetilde{\mathcal N}_{R_\kappa-C_1-\varrho}(0;n)\ge 1$
and $n^{-1}\log{\widetilde{\mathcal N}_{R_\kappa+C_1+\varrho}(0;n)}<\delta$ for all $n\ge n_0$.
In view of \eqref{eq:PCstep1} it follows that for such $n$ we have
\[
0\le\frac{1}{n}\log{\mathcal N_{R_\kappa}(\epsilon;\omega,n)}<\delta
\]
for all $0\le\epsilon<\epsilon(n)$
and $\mathbb{P}$-almost every $\omega$. This completes the proof.
\end{proof}

\begin{rmk}
Using the notation from the proof of Proposition \ref{prop:differentnoiselevels},
we make the following observation.
With the help of Lemma \ref{lem:technical1} it is not hard to show that the quantity $I_n(\epsilon)$
is in fact $o(1)$ as $\epsilon\to0$, independent of time $n\in\N^\ast$.
As shown, the same is true for $J_n(\epsilon)$. However, this does not by itself imply that
the number $\epsilon(n)$ can be chosen independently of $n$, since the number $L(n)$ tends to zero as $n\to\infty$.
\end{rmk}

\subsection{Relation to the condition of transversality}\label{subsection:transversality}

Here we compare
partial captivity to the {\it transversality condition} of Tsujii~\cite{Tsujii},
introduced to study mixing of suspension semi-flows and
shown to be generic there for linear expanding maps. The transversality condition was 
recently formulated for \eqref{eq:unperturbedsystem} by Butterley and Eslami~\cites{ButterleyEslami,Eslami}
under weak regularity assumptions.
After establishing the connection we discuss in what sense partial captivity is generic
when the unperturbed expanding map $E_0$ is linear, i.e.~$E_0(x)=kx$ mod 1.
Since the general case
will be studied
in an upcoming joint paper with Masato Tsujii we only sketch the details.

We first recall the definitions involved.
Given a ceiling function $\tau_0$, define
\begin{equation}\label{eq:varthetatau}
\vartheta_\tau=\frac{1}{2\pi}\frac{\sup_x{\lvert\tau_0'(x)\rvert}}{\lambda_0-1}.
\end{equation}
(The presence of the factor $(2\pi)^{-1}$ is a result of the normalization in \eqref{eq:unperturbedsystem}.)
Introduce also, for each $R>0$, the number
\begin{equation}\label{eq:varthetaK}
\vartheta_R=\frac{R}{\lambda_0-1},
\end{equation}
and the corresponding cone $\mathscr K_R=\{(\xi,\eta)\in\R^2:\lvert\eta\rvert\le\vartheta_R\lvert\xi\rvert\}$.
It is straightforward to check that this cone is invariant under the Jacobian
\[
D f_0(z)=\left( \begin{array}{cc} E_0 '(x) & 0\\ (2\pi)^{-1}\tau'_0(x) & 1\end{array} \right),
\quad z=(x,s)\in\T^2,
\]
as long as $\vartheta_R\ge\vartheta_\tau$, 
where by invariant we mean $D f_0(z)\mathscr K_R\subset \mathscr K_R$.
Fix $z=(x,s)\in\T^2$ and $n\ge1$, and let $\zeta$ and $w$ be two distinct points in the pre-image of $z$
under $f_{0}^{(n)}$, that is, $\zeta\ne w$ and $f_{0}^{(n)}(\zeta)=z=f_{0}^{(n)}(w)$.
In this subsection, we will write e.g.~$\zeta\in f_0^{-n}(z)$ in this case.
We also write $f_0^n$ instead of $f_0^{(n)}=f_0\circ\cdots\circ f_0$,
and $G_0^n$ instead of $G_0^{(n)}$ for iterates of the inverse of the map $E_0$ lifted on $\R$.
We say that $\zeta$ and $w$ are {\it transversal}, denoted $\zeta\pitchfork w$,
when $Df_0^n(\zeta) \mathscr K_R\cap Df_0^n(w)\mathscr K_R=\{0\}$.
As a way to measure transversality, define
\[
\varphi(n)\equiv\varphi_R(n)=\sup_z{\sup_{\zeta\in f_0^{-n}(z)}{\sum_{w\not\pitchfork \zeta}
\frac{1}{\det Df_0^n(w)}}}
\]
with the sum taken over those $w\in f_0^{-n}(z)$ such that 
$Df_0^n(\zeta) \mathscr K_R\cap Df_0^n(w)\mathscr K_R\neq\{0\}$.
When $E_0$ is linear, this reduces to
\begin{equation}\label{eq:defofvarphi(n)}
\varphi(n)=k^{-n}\sup_z{\sup_{\zeta\in f_0^{-n}(z)}{\#\{ w\in f_0^{-n}(z) : w\not\pitchfork\zeta\}}}.
\end{equation}
We say that $f_0$ satisfies the {\it transversality condition} if
\begin{equation}\label{def:transversality}
\limsup_{n\to\infty}\varphi(n)^\frac{1}{n}<1.
\end{equation}
The condition is violated precisely when $\tau_0$ is cohomologous to a constant, 
see Butterley and Eslami~\cite{ButterleyEslami}*{Proposition 1} (compare with Tsujii~\cite{Tsujii}*{Theorem 1.4}).
In view of the remark following Definition \ref{def:randompartiallycaptive}
it follows that $f_0$ satisfies the transversality condition if $f_0$ is partially captive.
Conversely, we make the following observation.

\begin{lem}\label{lem:transversal}
When $E_0$ is linear we have
\begin{equation}\label{eq:conversecomparisonprime}
\mathcal{N}_{R_\kappa}(0;n)\le k^n\varphi_R(n)
\end{equation}
if $R\ge \pi^{-1}(\sup_{x}{\lvert\tau_0'(x)\rvert}+\frac{1}{2}(k-1)R_\kappa)$. 
\end{lem}

\begin{proof}
Let $\widetilde R=R_\kappa+(k-1)^{-1}\sup_{x}{\lvert\tau_0'(x)\rvert}$.
An application of Proposition \ref{prop:PCequiv1} with $\epsilon=0$ shows that the lemma follows
if we prove that 
\begin{equation}\label{eq:conversecomparison}
\widetilde{\mathcal{N}}_{\widetilde R}(0;n)\le k^n\varphi_R(n)
\end{equation}
when
\begin{equation}\label{eq:conditiononR}
R\ge (2\pi )^{-1}(\sup_{x}{\lvert\tau_0'(x)\rvert}+(k-1)\widetilde R).
\end{equation}
Fix $z=(x,s)\in\T^2$ and $n\ge1$, and find $\eta_0\in\R$
such that
\[
\#\widetilde{\mathcal{A}}_{\widetilde R}(0;n)(x,\eta_0)\ge \#\widetilde{\mathcal{A}}_{\widetilde R}(0;n)(x,\eta),
\quad\eta\in\R.
\]
For any $0\le\bar\alpha\le k^n-1$, the choice $\eta=S(x+\bar\alpha)$ gives $\bar\alpha\in \widetilde{\mathcal{A}}_{\widetilde R}(0;n)(x,\eta)$,
so we have $\bar\alpha_1\in\widetilde{\mathcal{A}}_{\widetilde R}(0;n)(x,\eta_0)$ for some $\bar\alpha_1$.
We now show that if there is an $\bar\alpha_2\ne\bar\alpha_1$ with $\bar\alpha_2\in\widetilde{\mathcal{A}}_{\widetilde R}(0;n)(x,\eta_0)$
then there are two distinct points $\zeta_1$ 
and $\zeta_2$ in $f_0^{-n}(z)$ with $\zeta_1\not\pitchfork\zeta_2$.
To this end, let $\pi_x:\T^2\to\Sone$ denote the canonical projection
and define $\zeta_j\in f_0^{-n}(z)$ by $\pi_x(\zeta_j)=G_0^n(x+\bar\alpha_j)$ for $j=1,2$.
(A point $(y,s')\in f_{0}^{-n}(z)$
is uniquely determined by $y$ and $z$ and vice versa, see the discussion preceding \eqref{adjointrestrictionoperator}.)
Set $\xi_j=-2\pi k^{-n}$ and
\[
\eta_j=\eta_0-\frac{\xi_j}{2\pi}\frac{d\tau_0^{(n)}(\pi_x(\zeta_j))}{dx}
=\eta_0+\frac{1}{k^n}\frac{d\tau_0^{(n)}(\pi_x(\zeta_j))}{dx},
\]
where $\tau_0^{(n)}$ is defined in accordance with \eqref{notation:tau^{(n)}}.
Then $Df_0^n(\zeta_j)(\xi_j,\eta_j)^t=(-2\pi,\eta_0)^t$ and we claim that $(\xi_j,\eta_j)\in\mathscr K_R$
if $R$ satisfies \eqref{eq:conditiononR}. 
Indeed, it is easy to see that
\[
\frac{1}{k^n}\frac{d\tau_0^{(n)}(\pi_x(\zeta_j))}{dx}
=\sum_{\ell=1}^n \tau_0'(G_0^\ell(x+\alpha_j))k^{-\ell}
\]
which is the partial sum of the first $n$
terms of $-S_0(x+\bar\alpha_j)$ when $E_0$ is linear, see \eqref{eq:solutiontocohomological}.
Hence, Definition \ref{def:tildeAandbarA} and a simple calculation gives
\begin{align}
\lvert\eta_j\rvert&\le\lvert\eta_0-S(x+\bar\alpha_j)\rvert
+\lvert S(x+\bar\alpha_j)+k^{-n}d\tau_0^{(n)}(\pi_x(\zeta_j))/dx\rvert
\\ &\le k^{-n}\bigg(\widetilde R+\frac{\sup_x{\lvert\tau_0'(x)\rvert}}{k-1}\bigg).
\end{align}
Now, since $\xi_j=-2\pi k^{-n}$ we have
$(\xi_j,\eta_j)\in\mathscr K_R$ if $\lvert\eta_j\rvert\le2\pi k^{-n}\vartheta_R$,
where $\vartheta_R$ is defined by \eqref{eq:varthetaK}.
This is easily seen to hold if $R$ satisfies 
\eqref{eq:conditiononR}. 
It follows that $Df_0^n(\zeta_1)\mathscr K_R\cap Df_0^n(\zeta_2)\mathscr K_R\ne\{0\}$,
i.e.~$\zeta_1\not\pitchfork\zeta_2$. Hence, by \eqref{eq:defofvarphi(n)} we have
\[
\#\widetilde{\mathcal{A}}_{\widetilde R}(0;n)(x,\eta_0)
\le \#\{ \zeta_2\in f_0^{-n}(z) : \zeta_1\not\pitchfork\zeta_2\}\le k^n\varphi_R(n)
\]
which yields \eqref{eq:conversecomparison}. This completes the proof.
\end{proof}

\begin{thm}\label{thm:generic}
Let $3\le m\le\infty$, and let $E_0$ be fixed and linear.
Let $\mathscr V_{(m)}\subset\Ci^m(\Sone)$ be the subset consisting of the $\tau_0$'s for which \eqref{eq:unperturbedsystem}
is not partially captive. Then $\mathscr V_{(m)}$ is a meager set with empty interior.
\end{thm}

\begin{proof}[Sketch of proof]
Let $m$ be finite.
By adapting the proof of Tsujii~\cite{Tsujii}*{Theorem 1.2} it can be checked that for each $\varrho>1$ (and $R>1$) there
is an open dense subset $U_\varrho\subset\Ci^m(\Sone)$ such that if $\tau_0\in U_\varrho$
then
\begin{equation}\label{eq:transversalityimprovedagain}
\limsup_{n\to\infty}\varphi_R(n)^\frac{1}{n}\le\varrho k^{-1}.
\end{equation}
In particular, we can replace the space $\Ci_+^m(\Sone)$ of positive ceiling functions under consideration in Tsujii~\cite{Tsujii}*{Theorem 1.2}
by $\Ci^m(\Sone)$ since addition of $2\pi n$ ($n$ integer) does not change the situation, see \eqref{eq:unperturbedsystem}.
Let $j\in\N^\ast$ be arbitrary and choose $\varrho$ such that $\log\varrho< 1/j$. By Lemma \ref{lem:transversal}
and \eqref{eq:transversalityimprovedagain} we get
\begin{equation}\label{eq:infimuminsteadoflimit}
\inf_{n}{ n^{-1}\log{\mathcal{N}_{R_\kappa}(0;n)}}< 1/j
\end{equation}
for $\tau_0\in U_j\equiv U_{\varrho(j)}$. (Here $R_\kappa$ depends on $\tau_0$.)
In fact, $n\mapsto\log{\mathcal{N}_{R_\kappa}(0;n)}$
is subadditive (Faure~\cite{Faure}*{p.~1489}) so the infimum coincides with the limit as $n\to\infty$.
If $\tau_0$ belongs to the intersection $\mathscr U_{(m)}=\cap_{j=1}^\infty U_j$ then $f_0$ is partially captive.
The complement of $\mathscr U_{(m)}$ is a countable union of closed nowhere dense sets,
i.e.~a meager set. Since $\mathscr V_{(m)}\subset\complement\mathscr U_{(m)}$ and subsets of meager sets are meager, the claim follows
by Baire's category theorem.

Next, consider the situation in $\Ci^\infty(\Sone)$. Let $X_j$ be the set of $\tau\in\Ci^\infty$
such that there is an $n$ for which $n^{-1}\log{\mathcal{N}_{R_\kappa}(0;n)}<1/j$.
In view of the discussion following \eqref{eq:infimuminsteadoflimit} we have
$\mathscr V_{(\infty)}\subset\cup_{j=1}^\infty\complement X_j$.
In view of Proposition \ref{wPC}
it is easy to see that each $X_j$ is open in $\Ci^\infty(\Sone)$. 
We also claim that each $X_j$
is dense in $\Ci^\infty(\Sone)$. Indeed, if not then there is a $\tau_0$ together with $m$ and $\epsilon$
such that $\tau\in\Ci^\infty(\Sone)$ and $\lVert\tau-\tau_0\rVert_{\Ci^m}<\epsilon$ implies
that $\tau\in\complement X_j$. But viewed as a subset of $\Ci^m(\Sone)$
we have $\complement X_j\subset\complement U_j$
with $U_j$ defined above, so this contradicts the fact that the closure of $\complement U_j$
has empty interior. Hence, $\complement X_j$ is closed and nowhere dense in $\Ci^\infty(\Sone)$,
and the proof is completed by repeating the arguments at the end of the previous paragraph.
\end{proof}

\section{The proof of Theorem \ref{thm:discretespectrum}}\label{section:discrete}

In this section we give the proof of Theorem \ref{thm:discretespectrum}.
We begin by showing that $(n,\omega ,\varphi) \mapsto M^\ast_{\nu ,n}(\omega)\varphi$ is a
linear RDS on $H^m(\Sone)$.
This is a consequence of known results
for transfer operators of expanding maps on Sobolev spaces. It will be
convenient to use a formulation of Faure~\cite{Faure}*{Theorem 2}, but we mention
the earlier results of Baillif and Baladi~\cite{BaillifBaladi}*{Section 4} as well as Baladi and Tsujii~\cite{BaladiTsujii07}*{Theorem 1.1}.
The corresponding results for operators on $\Ci^m$ go back to
Ruelle~\cites{Ruelle89,Ruelle90}.

\begin{prop}\label{prop:discretespectrum2}
For $\nu\in\Z$, let $M_{\nu,n}^\ast(\omega)=M_{\nu,n}^\ast(\epsilon;\omega)$ be the operator defined
by means of the decomposition \eqref{adjointrestrictionoperator}.
Let $m$ be a positive integer and define $r_m$ by \eqref{eq:rm}. 
The following holds $\mathbb{P}$-almost surely:
For each $\omega\in\Omega$ and $n\ge1$, $M_{\nu ,n}^\ast (\omega) :H^m(\Sone) \rightarrow H^m(\Sone)$
is continuous and can be written
\[
M_{\nu ,n}^\ast(\omega) =R_{\nu ,n}^\ast(\omega) +K_{\nu ,n}^\ast(\omega)
\]
where $K_{\nu ,n}^\ast(\omega):H^m(\Sone) \rightarrow H^m(\Sone)$ is a compact operator, and
\[
\lVert R_{\nu ,n}^\ast (\omega) \rVert_{\mathscr L(H^m(\Sone))} \leq (e^{r_m})^n.
\]
\end{prop}

\begin{proof}
By virtue of \eqref{emin}
we can $\mathbb{P}$-almost surely  apply Faure~\cite{Faure}*{Theorem 2} and conclude that 
$M_{\nu} (\omega) :H^{-m}(\Sone) \rightarrow H^{-m}(\Sone)$
is $\mathbb{P}$-almost surely a bounded operator which can be written
$M_{\nu}(\omega) =R_{\nu}(\omega) +K_{\nu}(\omega)$,
where $K_{\nu}(\omega):H^{-m} \rightarrow H^{-m}$ is compact
and $\lVert R_{\nu} (\omega) \rVert_{\mathscr L(H^{-m})} \leq e^{r_m}$ with the number $r_m$ defined above.
Since $\lVert R_{\nu}^\ast (\omega) \rVert_{\mathscr L(H^m)}=\lVert R_{\nu} (\omega) \rVert_{\mathscr L(H^{-m})}$
by duality, and $K_{\nu}^\ast(\omega):H^m \rightarrow H^m$ is compact by Schauder's theorem
(see Chapter X, Section 4 in Yosida~\cite{Yosida}), it follows that
$M_{\nu}^\ast(\omega) =R_{\nu}^\ast(\omega) +K_{\nu}^\ast(\omega)$
can be decomposed in a similar manner.

Since $\theta$ is $\mathbb{P}$-preserving it follows that
for fixed $j$, the conclusions made in the previous paragraph concerning
the operator $M_{\nu}^\ast (\omega)$ acting on $H^m(\Sone)$ 
remain valid for $M_{\nu}^\ast (\theta^j\omega) :H^m(\Sone) \rightarrow H^m(\Sone)$
$\mathbb{P}$-almost surely. This implies that we $\mathbb{P}$-almost surely have
\[
M_{\nu ,n}^\ast (\omega )=(R_{\nu}^\ast(\theta^{n-1}\omega) 
+K_{\nu}^\ast(\theta^{n-1}\omega))\circ\ldots\circ(R_{\nu}^\ast(\omega) +K_{\nu}^\ast(\omega)).
\]
Let $K_{\nu ,n}^\ast(\omega)$ denote the difference between the right-hand side and
$R_{\nu ,n}^\ast(\omega)=R_{\nu}^\ast(\theta^{n-1}\omega)\circ\ldots\circ R_{\nu}^\ast(\omega)$,
so that $M_{\nu ,n}^\ast (\omega )=R_{\nu ,n}^\ast(\omega) +K_{\nu ,n}^\ast(\omega)$.
Then $K_{\nu ,n}^\ast(\omega)$ is a finite sum of operators where each term is given by
the composition of finitely many bounded operators, at least one of which is compact.
Hence $K_{\nu ,n}^\ast(\omega)$ is compact. Moreover,
\begin{align*}
\lVert R_{\nu ,n}^\ast (\omega) \rVert_{\mathscr L(H^m)}
&=\lVert R_{\nu}^\ast(\theta^{n-1}\omega)\circ\ldots\circ 
R_{\nu}^\ast(\omega) \rVert_{\mathscr L(H^m)}\\
&\le \lVert R_{\nu}^\ast(\theta^{n-1}\omega)\rVert_{\mathscr L(H^m)} \cdots \lVert R_{\nu}^\ast(\omega) \rVert_{\mathscr L(H^m)}\le (e^{r_m})^n,
\end{align*}
which completes the proof.
\end{proof}

Next, we show that the conditions of Definition \ref{def:lyapunovexponents}
are satisfied. To verify that the time-one map $M_\nu^\ast:\Omega\to\mathscr L(H^m(\Sone))$
is strongly measurable, note that it suffices to check that
$M_\nu^\ast$ is measurable with respect to the $\sigma$-algebra on $\Omega$ and the Borel $\sigma$-algebra
on $\mathscr L(H^m(\Sone))$ generated by the strong operator topology, see for example
Gonz{\'a}lez-Tokman and Quas~\cite{ETS:8859243}*{Appendix A}.
We may regard $M_\nu^\ast$ as the
composition of a map from $\Omega$ into ${\Ci}^\infty(\Sone,\Sone)\times {\Ci}^\infty(\Sone,\R)$
given by $\omega\mapsto (E_\epsilon(\omega),\tau_\epsilon(\omega))$,
and a map from a subset of ${\Ci}^\infty(\Sone,\Sone)\times {\Ci}^\infty(\Sone,\R)$ into $\mathscr L(H^m(\Sone))$,
the domain being a neighborhood of $(E_0,\tau_0)$
consisting of pairs $(E,\tau)$ where the first component is a uniformly expanding map.
It is straightforward (although somewhat tedious) to check that the second map is
continuous with respect to the strong operator topology, while the first
is measurable by assumption. Hence, $M_\nu^\ast:\Omega\to\mathscr L(H^m(\Sone))$
is strongly measurable.
We remark for expanding maps, the corresponding transfer operators
$\Omega\to\mathscr L({\Ci}^m(\Sone))$ are known
to be strongly measurable, and since $M_\nu^\ast(\omega)$ can be viewed as a {\it weighted} transfer operator
of the expanding map $x\mapsto E_\epsilon(\omega,x)$ with weight $\exp{(i\nu\tau_\epsilon(\omega))}$,
said property is to be expected. 
(Due to \eqref{eq:perturbedsystem}--\eqref{convinC1},
the presence of  the additional weight is unproblematic in this regard.)
However, note that transfer operators 
are generally not continuous in the norm topology, see Baladi and Young~\cite{BY93}
as well as Bogensch{\"u}tz~\cite{Bogenschutz00}*{Section 3.6}.
Compare with Gonz{\'a}lez-Tokman and Quas~\cite{ETS:8859243}*{Section 1}.

To prove that $\omega\mapsto \log^+\lVert M_\nu^\ast(\omega)\rVert_{\mathscr L(H^m)}$
is integrable with respect to the probability measure $\mathbb{P}$, we shall have to
revisit the method of proof of Faure~\cite{Faure}*{Theorem 2}, if only briefly.
Identify $H^m(\Sone)$ with $\langle D\rangle^{-m}(L^2(\Sone))$, where
$\langle \xi\rangle^{m}=(1+\lvert\xi\rvert^2)^{m/2}$ and $\langle D\rangle^{m}$ is 
the formally self-adjoint invertible pseudodifferential operator
acting through multiplication by $\langle \xi\rangle^{m}$ on the Fourier side, see
\eqref{Fouriermultiplier}. The commutative diagram
\[
\begin{array}{ccc}
L^2(\Sone) & \stackrel{Q_\nu(\omega)}{\to} & L^2(\Sone) \\
\downarrow\text{\scriptsize{$\langle D\rangle^{m}$}} & \circlearrowleft & \downarrow\text{\scriptsize{$\langle D\rangle^{m}$}} \\
H^{-m}(\Sone) & \stackrel{M_\nu(\omega)}{\to} & H^{-m}(\Sone)
\end{array}
\]
shows that $M_\nu(\omega):H^{-m}(\Sone)\to H^{-m}(\Sone)$ is unitarily equivalent to
\begin{equation}\label{eq:definitionofQnew}
Q_\nu(\omega)=\langle D\rangle^{-m} M_\nu(\omega)\langle D\rangle^{m}:L^2(\Sone)\to L^2(\Sone).
\end{equation}
Consider the operator $(Q_{\nu}(\omega))^\ast Q_{\nu}(\omega)$ (which naturally depends on $m$
although this is not reflected in the notation). A careful analysis
of this operator, including the calculation of its principal symbol, is at the core
of the proof of Faure~\cite{Faure}*{Theorem 2}, and would provide an alternative proof of
Proposition \ref{prop:discretespectrum2}. At this point however, all we need is the following
result, the proof of which can be found at the end of Appendix \ref{app:PsiDO}.

\begin{thm}\label{thm:Q}
Let $\nu\in\Z$ be fixed and let $Q_\nu(\omega)$ be given by \eqref{eq:definitionofQnew}. Then 
there is an $\epsilon_0=\epsilon_0(m)$, independent of $\nu$, such that if $0\le\epsilon<\epsilon_0$,
the operator $(Q_{\nu}(\omega))^\ast Q_{\nu}(\omega):L^2(\Sone)\to L^2(\Sone)$ is a pseudodifferential operator
satisfying
\begin{equation*}
\lVert (Q_{\nu}(\omega))^\ast Q_{\nu}(\omega)\rVert_{\mathscr L(L^2(\Sone))}\le C_{\nu,m}, \quad\text{$\mathbb{P}$-almost surely},
\end{equation*}
where the constant as indicated may depend on $\nu$ and $m$ but not on $\omega$.
\end{thm}

Note that if $(\phantom{i},\phantom{i})_{L^2}$
denotes the scalar product on $L^2(\Sone)$, then
\begin{align*}
\lVert Q_{\nu}(\omega)u\rVert_{L^2}^2&
=(Q_{\nu}(\omega)u,Q_{\nu}(\omega)u)_{L^2}\\
&=(Q_{\nu}^\ast(\omega)Q_{\nu}(\omega)u,u)_{L^2}
\le \lVert Q_{\nu}^\ast(\omega)Q_{\nu}(\omega)u\rVert_{L^2}\lVert u\rVert_{L^2}
\end{align*}
and that 
$\lVert Q_{\nu}^\ast(\omega)Q_{\nu}(\omega)\rVert_{\mathscr L(L^2)}
\le \lVert Q_{\nu}(\omega)\rVert_{\mathscr L(L^2)}^2$. Hence we have equality.
Since $M_\nu(\omega):H^{-m}(\Sone)\to H^{-m}(\Sone)$ is unitarily equivalent to
$Q_\nu(\omega):L^2(\Sone)\to L^2(\Sone)$,
we can therefore by virtue of Theorem \ref{thm:Q} together with a duality argument find a constant
(depending on $\nu$ and $m$ but not on $\omega$) such that
\begin{equation}\label{eq:consequenceofthm:Q}
\lVert M_{\nu}^\ast(\omega)\rVert_{\mathscr L(H^{m}(\Sone))}\le C_{\nu,m}\quad\text{$\mathbb{P}$-almost surely}.
\end{equation}
Hence, $\omega\mapsto \log^+\lVert M_{\nu}^\ast(\omega)\rVert_{\mathscr L(H^m)}\in L^1(\Omega,\mathbb{P})$
for each $\nu\in\Z$.

We have now shown that the conditions in Definition \ref{def:lyapunovexponents}
are satisfied by the linear RDS on $H^m(\Sone)$ induced by
$M_{\nu}^\ast$. In view of the discussion following the definition,
the maximal Lyapunov exponent $r^\ast\{\nu\}$ and
the index of compactness $r_{\mathrm{ic}}^\ast\{\nu\}$ then exist.
By Proposition \ref{prop:discretespectrum2} we have $M_{\nu,n}^\ast(\omega)
=R_{\nu,n}^\ast(\omega)+K_{\nu,n}^\ast(\omega)$, where $K_{\nu,n}^\ast(\omega)$
is a compact operator and $\lVert R_{\nu ,n}^\ast (\omega)\rVert_{\mathscr L(H^m)}
\le (e^{r_m})^n$.
By virtue of the definition of the 
index of compactness for bounded operators,
this gives
\[
\lVert M_{\nu ,n}^\ast (\omega)\rVert_{\mathrm{ic}(H^m)}
\le\lVert R_{\nu ,n}^\ast (\omega)\rVert_{\mathrm{ic}(H^m)}
\le\lVert R_{\nu ,n}^\ast (\omega)\rVert_{\mathscr L(H^m)}
\le (e^{r_m})^n.
\]
Hence,
\begin{equation}\label{riclessthanrm}
r_{\mathrm{ic}}^\ast\{\nu\}=\lim_{n\to\infty}\frac{1}{n}\log{\lVert
M_{\nu ,n}^\ast (\omega)\rVert_{\mathrm{ic}(H^m)}}\le r_m.
\end{equation}

Next, we show that $r^\ast\{\nu\}\le 0$. This will be a direct consequence
of the following so-called 
weak Lasota-Yorke inequality,
which we prove by 
using \eqref{eq:consequenceofthm:Q} and adapting the ideas in 
Baladi et al.~\cite{BKS96}*{Lemma 4.2} to our situation.
(Note that writing
$M_{\nu,n}^\ast(\omega)=M_{\nu}^\ast(\theta^{n-1}\omega)\circ\cdots\circ
M_{\nu}^\ast(\omega)$ and using \eqref{eq:consequenceofthm:Q} at $\omega,\ldots,\theta^{n-1}\omega$
gives
\begin{equation}\label{eq:continuityestimate}
\lVert M_{\nu ,n}^\ast(\omega) \rVert_{\mathscr L(H^m(\Sone))} \leq (C_{\nu,m})^n, \quad n\ge1,
\end{equation}
$\mathbb{P}$-almost surely, which
only shows that $r^\ast\{\nu\}\le \log{C_{\nu,m}}$.)

\begin{lem}\label{lem:LYineq}
For each $m\in \N$ there is an $\epsilon_0=\epsilon_0(m)$ such that
if $0\le\epsilon <\epsilon_0$ and $\nu\in\Z$ then
\begin{equation}\label{eq:LYineq1}
\lVert  M_{\nu,n}^\ast (\epsilon ;\omega)\rVert _{\mathscr L(H^m(\Sone))}\leq C_{\nu,m}
\end{equation}
for all $n\geq 1$, $\mathbb{P}$-almost surely, where the constant $C_{\nu,m}$ is independent of 
$0\le\epsilon <\epsilon_0$ and $\omega\in\Omega$. 
\end{lem}

\begin{proof}
By the density of ${\Ci}^{\infty}(\Sone)$ in $H^m(\Sone)$ together with \eqref{eq:continuityestimate}, 
it suffices to show that for each $\nu\in\Z$ and $m\in \N$ there is a constant $C_{\nu,m}>0$ such that 
if $u \in {\Ci}^{\infty}(\Sone)$ and $n\geq 1$, then
\begin{equation}
\lVert M_{\nu,n}^\ast (\epsilon; \omega)u \rVert_{H^m} \leq C_{\nu,m}\lVert u \rVert_{H^m}.
\end{equation} 
Here $\lVert \phantom{i} \rVert_{H^m}$ is
the equivalent norm on $H^m(\Sone)$ which for $m\in\N$ is given by
$\lVert u\rVert_{H^m}^2=\sum_{j=0}^m\lVert \partial^j u\rVert_{L^2}^2$,
where $\partial^j u$ denotes the $j$th derivative of $u$ with respect to the independent variable.
We will let $0\le\epsilon<\epsilon_0$ be arbitrary, where $\epsilon_0=\epsilon_0(m)$
is the number provided by Theorem \ref{thm:Q},
and suppress $\epsilon$ from the notation.

We prove the lemma by induction, and consider first the case $m=0$. By duality we have 
$\lVert M^\ast _{\nu,n}(\omega)\rVert_{\mathscr L(L^2)}= \lVert M_{\nu,n}(\omega) \rVert_{\mathscr L(L^2)}$ 
for each $n\geq 1$ and $\mathbb{P}$-almost every $\omega \in \Omega$.
Let $E^{(n)}(\omega)$ and $\tau^{(n)}(\omega)$
be given by \eqref{notation:E^{(n)}}--\eqref{notation:tau^{(n)}}.
In view of \eqref{restrictionoperator} we
get
\begin{equation}\label{eq:firstind00}
\lVert M_{\nu,n}(\omega)u\rVert_{L^2}^2 =\int \lvert e^{i\nu \tau^{(n)}(\omega,x)}u(E^{(n)}(\omega,x))\rvert^2 dx
=( 1_{\Sone}, M_{0,n}(\omega)\lvert u\rvert^2)_{L^2}
\end{equation}
for $\nu\in\Z$,
where 
\begin{equation}\label{eq:secondind00}
( 1_{\Sone}, M_{0,n}(\omega)\lvert u\rvert^2)_{L^2}
=( M_{0,n}^\ast(\omega)1_{\Sone}, \lvert u\rvert^2)_{L^2}
\leq \sup _{x} M_{0,n}^\ast (\omega)1_{\Sone}(x)\lVert u\rVert_{L^2}^2
\end{equation}
for each $u\in L^2(\Sone)$, $n\geq 1$ and $\mathbb{P}$-almost every $\omega$.
Now recall that the time-one map of $M_{0,n}^\ast (\omega)$ is the transfer operator of the
expanding map $E(\omega)$. We can therefore use the results of Baladi et al.~\cite{BKS96}
to find a constant $C>0$ (independent of $\epsilon >0$) so that  
$\sup _{x} M_{0,n}^\ast (\omega)1_{\Sone}(x) \leq C\sup _{x} 1_{\Sone}(x)$ 
for each $n\geq 1$ and $\omega \in \Omega$, see in particular
the estimate (4.5) in Baladi et al.~\cite{BKS96}. Together with \eqref{eq:firstind00}, \eqref{eq:secondind00}
and a duality argument,
this is easily seen to imply that we $\mathbb{P}$-almost surely have
\begin{align}\label{eq:ind000}
\lVert M_{\nu,n}^\ast(\omega)u \rVert_{L^2}\leq C^{\frac{1}{2}}\lVert u \rVert_{L^2},
\quad u\in L^2(\Sone),
\end{align}
that is, \eqref{eq:LYineq1} holds for $m=0$.

Assume that \eqref{eq:LYineq1} holds with $m$ replaced by $m-1$. By the properties of the norm  
we then have
\begin{align}\label{eq:mthstep}
\lVert M_{\nu,n}^\ast (\omega)u \rVert_{H^m} & \leq
\lVert M_{\nu,n}^\ast (\omega)u \rVert_{H^{m-1}}
+ \lVert \partial^m M_{\nu,n}^\ast (\omega)u \rVert_{L^2} \\ & \leq
 C_{\nu,m-1}\lVert u \rVert_{H^m}+\lVert \partial^m M_{\nu,n}^\ast (\omega)u \rVert_{L^2},
\end{align}
so to show that \eqref{eq:LYineq1} holds also for $m$,
it suffices to estimate $\lVert \partial^m M_{\nu,n}^\ast (\omega)u \rVert_{L^2}$.
In view of \eqref{adjointrestrictionoperator},
a straightforward calculation followed by a moment's reflection shows that the $m$th derivative of $M_{\nu,n}^\ast(\omega)u(x)$ 
with respect to $x$ can be written as
\begin{equation}\label{eq:betterind00}
\partial^m (M_{\nu,n}^\ast (\omega)u)(x)
=\sum _{j=0}^m M_{\nu,n}^\ast (\omega) \big[\varphi_{\nu,n,j}(\omega) (\partial^j u)\big](x),
\end{equation}
where each $\varphi _{\nu,n,j}(\omega)$ is smooth, independent 
of $u$, and involves derivatives of the functions $E^{(n)}(\omega)$ and $\tau^{(n)}(\omega)$
and $(dE^{(n)} (\omega ,y)/dy)^{-1}$. In particular, for $j=m$ we have
\begin{equation}\label{eq:varphiforjequalm}
\varphi _{\nu,n,m}(\omega,x)=(dE^{(n)}(\omega,x)/dx)^{-m},
\end{equation}
which by \eqref{emin} satisfies
\[
\esssup_\omega{\sup_x{\lvert \varphi_{\nu,n,m}(\omega,x)\rvert}}
\le\lambda^{-m n},
\quad n\ge1,
\]
if $0\le\epsilon<\epsilon_0$.
Similarly, by \eqref{convinC1} and \eqref{emin} it follows (after decreasing $\epsilon_0$ if necessary) that 
\[
\esssup_\omega{\sup_x{\lvert \varphi_{\nu,n,j}(\omega,x)\rvert}} \le \tilde C_{\nu,n,j},
\quad 0\le j\le m-1.
\]
Combining \eqref{eq:betterind00} with \eqref{eq:ind000} and 
these estimates, we get
\begin{align}
\lVert \partial^m M_{\nu,n}^\ast (\omega)u \rVert_{L^2}& 
\le C^{\frac{1}{2}}\sum _{j=0}^{m-1} \tilde C_{\nu,n,j}\lVert \partial^j u\rVert_{L^2}
+C^{\frac{1}{2}}\lambda^{-mn}\lVert \partial^m u\rVert_{L^2}\\
&\le C_{\nu,n,m-1}\lVert u\rVert_{H^{m-1}}
+C^{\frac{1}{2}}\lambda^{-mn}\lVert \partial^m u\rVert_{L^2}.
\end{align}
Hence, if $1<\rho<\lambda$ we can find an $n_0$ such that we $\mathbb{P}$-almost surely have
\begin{equation}\label{eq:betterind000}
\lVert \partial^m M_{\nu,n_0}^\ast (\omega)u \rVert_{L^2}
\le C_{\nu,m}\lVert u\rVert_{H^{m-1}}
+\rho^{-n_0}\lVert \partial^m u\rVert_{L^2},
\end{equation}
where $C_{\nu,m-1}\equiv C_{\nu,n_0,m-1}$.

Now let $n\ge1$ be arbitrary, and write $n=\ell n_0+r$, where $\ell\in\N$
and $0\le r<n_0$. First note that for $0\le n<n_0$ we can use
\eqref{eq:continuityestimate} to find a constant independent of $n$ such that
$\lVert \partial^m M_{\nu,n}^\ast (\omega)u \rVert_{L^2}
\le C_{\nu,m}\lVert u\rVert_{H^{m}}$. Assume next that $n\ge n_0$.
Since $M_{\nu,n}^\ast (\omega)=M_{\nu,n_0}^\ast (\theta^{n-n_0}\omega)\circ
M_{\nu,n-n_0}^\ast (\omega)$ and $\theta$ is $\mathbb{P}$-preserving, an application of \eqref{eq:betterind000} gives
\begin{align}\label{eq:betterind0000}
\lVert \partial^m M_{\nu,n}^\ast (\omega)u \rVert_{L^2}
&= \lVert \partial^m M_{\nu,n_0}^\ast (\theta^{n-n_0}\omega) [
M_{\nu,n-n_0}^\ast (\omega) u]\rVert_{L^2} \\ &
\le C_{\nu,m}\lVert M_{\nu,n-n_0}^\ast (\omega) u\rVert_{H^{m-1}}
+\rho^{-n_0}\lVert \partial^m M_{\nu,n-n_0}^\ast (\omega) u\rVert_{L^2},
\end{align}
$\mathbb{P}$-almost surely.
According to the induction hypothesis, the first term on the right can be estimated using
\eqref{eq:LYineq1} with $m$ replaced by $m-1$, which gives
\begin{equation}
\lVert \partial^m M_{\nu,n}^\ast (\omega)u \rVert_{L^2}
\le C_{\nu,m}\lVert  u\rVert_{H^{m-1}}
+\rho^{-n_0}\lVert \partial^m M_{\nu,n-n_0}^\ast (\omega) u\rVert_{L^2},
\end{equation}
$\mathbb{P}$-almost surely, for some new constant $C_{\nu,m}$.
Iterating this procedure, we get
\begin{equation}
\lVert \partial^m M_{\nu,n}^\ast (\omega)u \rVert_{L^2}
\le C_{\nu,m} (1+\cdots+\rho^{-(\ell-1)n_0})\lVert  u\rVert_{H^{m-1}}
+\rho^{-\ell n_0}\lVert \partial^m M_{\nu,r}^\ast (\omega) u\rVert_{L^2}.
\end{equation}
Since $\rho>1$ and $0\le r<n_0$, the last term can be estimated by
$C_{\nu,m}\lVert u\rVert_{H^{m}}$ by using \eqref{eq:continuityestimate} as explained above. Clearly, the sum
$1+\rho^{-n_0}+\cdots+\rho^{-(\ell-1)n_0}$ is dominated by $(1-\rho^{-n_0})^{-1}$,
independent of $n$. Hence, by invoking \eqref{eq:mthstep} we conclude that
\eqref{eq:LYineq1} holds, which completes the proof.
\end{proof}

By Lemma \ref{lem:LYineq} we have
\[
r^\ast\{\nu\}=\lim _{n\to \infty} \frac{1}{n} \log{\lVert M_{\nu,n}^\ast(\epsilon ;\omega)\rVert_{\mathscr L(H^m)}} \leq 0.
\]
On the other hand, when $\nu=0$ we have according to Theorem \ref{thm:inv} that
\begin{equation}\label{eq:integralidentity1}
\lVert M_{0,n}^\ast(\epsilon;\omega)h_\epsilon(\omega)\rVert_{(m)}
=\sup_{\lVert\varphi\rVert_{(-m)}\le1}(h_\epsilon(\theta^n\omega),\varphi)_{L^2}
\ge(h_\epsilon(\theta^n\omega),1_{\Sone})_{L^2}=1,
\end{equation}
$\mathbb{P}$-almost surely. This implies that
\[
\log{\lVert M_{0,n}^\ast(\epsilon;\omega)\rVert_{\mathscr L(H^m)}}\ge
-\log{\lVert h_\epsilon(\omega)\rVert_{(m)}}.
\]
Hence, $r^\ast\{0\}\ge0$, and we conclude that $r^\ast\{0\}=0$. 
In view of \eqref{eq:rm} and \eqref{riclessthanrm}, this means that
$r_\mathrm{ic}^\ast\{0\}<r^\ast\{0\}$ if $m$ is sufficiently large,
establishing quasi-compactness in this case.
In particular, the Oseledets splitting exists for $\nu=0$.
Applying Theorem \ref{thm:cocycle} to
the case $X=H^m(\Sone)$
and $\Phi (n,\omega)=M^\ast_{0 ,n}(\epsilon;\omega)$,
we accordingly denote the Lyapunov subspace associated to $r^\ast\{0\}=\alpha_1$ 
by $\Sigma_1(\epsilon;\omega)$, and note that 
\eqref{eq:integralidentity1} implies that
$h_\epsilon(\omega)\in\Sigma_1(\epsilon;\omega)$.
The proof of Theorem \ref{thm:discretespectrum}
will be complete upon establishing statement $(\mathrm{ii})$ in Theorem \ref{thm:discretespectrum}
and showing that $\Sigma_1(\epsilon;\omega)$
is one dimensional if $\epsilon$ is sufficiently small.

\subsection{Partial captivity and the peripheral spectrum}\label{subsection:perspec}

We first digress to discuss the peripheral spectrum in the non-perturbed case.
For each $\nu\in\Z$, let $M_\nu^\ast(0)$ be the auxiliary transfer operator of $f_0$
corresponding to $M_\nu^\ast(\epsilon;\omega)$ at noise level $\epsilon=0$.
We owe the following proof to
Masato Tsujii~\cite{TsujiiPC}.

\begin{thm}\label{thm:nonrandomperspec}
For $\nu=0$, the only eigenvalue of $M_0^\ast(0)$
on the unit circle is the simple eigenvalue $1$. In particular,
the Lyapunov subspace of $M_0^\ast(0)$ associated to the eigenvalues of modulus $1$
is one dimensional.
Moreover, if $f_0$ is partially captive then the spectral radius of $M_\nu^\ast(0):H^m(\Sone)\to H^m(\Sone)$
is strictly less than $1$ for all $\nu\ne0$. 
\end{thm}

\begin{proof}
The statements concerning $M_0^\ast(0)$ are well-known since
this is the Perron-Frobenius transfer operator of
the uniformly expanding map $E_0$, 
see for example Section 1 in Ma{\~n}{\'e}~\cite{Mane}*{Chapter 3}
and in particular Theorem 1.1 there.

The other statement is proved by contradiction. Suppose therefore that $f_0$ is partially captive and
there is a $\nu\ne0$ such that the spectral radius of $M_\nu^\ast(0)$ is $1$.
This means that there is an eigenfunction $v\in H^m(\Sone)$
corresponding to an eigenvalue $e^{i\vartheta}$ of modulus 1, that is,
\begin{equation}\label{eq:perspec1}
\sum_{E_0(y)=x} \frac{v(y)}{E'_0(y)}e^{-i\nu\tau_0(y)}=
M_\nu^\ast(0)v(x)=e^{i\vartheta}v(x),
\end{equation}
see Faure~\cite{Faure}*{Theorem 2}. In fact, according to the remark on p.~1479 in Faure~\cite{Faure}
we even have $v\in\Ci^\infty(\Sone)$.
Taking the moduli of both sides of \eqref{eq:perspec1} yields
\begin{equation}\label{eq:perspec2}
\lvert v(x)\rvert=\bigg\lvert\sum_{E_0(y)=x} \frac{v(y)}{E'_0(y)}e^{-i\nu\tau_0(y)}\bigg\rvert
\le\sum_{E_0(y)=x} \frac{\lvert v(y)\rvert}{E'_0(y)}.
\end{equation}
If we integrate in $x$ we obtain
\begin{equation}\label{eq:perspec3}
\lVert v\rVert_{L^1}=\int_{\Sone}\bigg\lvert\sum_{E_0(y)=x} \frac{v(y)}{E'_0(y)}e^{-i\nu\tau_0(y)}\bigg\rvert dx
\le\int_{\Sone} \sum_{E_0(y)=x} \frac{\lvert v(y)\rvert}{E'_0(y)}dx=
\lVert v\rVert_{L^1},
\end{equation}
where the last identity follows by a change of variables. 
Hence we have equality in \eqref{eq:perspec3}, therefore also in \eqref{eq:perspec2},
which together with \eqref{eq:perspec1} in turn implies that
\begin{equation}\label{eq:perspec4}
\arg\bigg(\frac{v(y)}{E'_0(y)}e^{-i\nu\tau_0(y)}\bigg)=\arg(v(x))+\vartheta
\quad\text{for all }y\in E_0^{-1}(x).
\end{equation}
Since $E_0'$ is real valued, we thus find that
$\arg(v(y))-\nu\tau_0(y)=\arg(v(x))+\vartheta$ whenever $E_0(y)=x$.
Note also that since we have identity in \eqref{eq:perspec2}, it follows that
$M_0^\ast(0)\lvert v\rvert(x)=\lvert v(x)\rvert$, so by virtue of Theorem \ref{thm:inv}
we have $\lvert v\rvert= c h_0$
for some constant $c$, where $h_0\in\Ci^\infty(\Sone)$ is the positive probability density function
appearing in the mentioned theorem. If we set $\varphi(y)= \Arg(v(y))$
where $\Arg{}$ is the principal value,
it therefore follows that $\varphi\in\Ci^\infty(\Sone)$ and
\begin{equation}\label{eq:perspec5}
\nu\tau_0(y)=\varphi(y)-\varphi(E_0(y))-\vartheta+2\pi n(y)
\end{equation}
for some integer valued function $y\mapsto n(y)$. However, $n(y)$ must be a constant function of $y$
since it can be written as a linear combination of smooth functions. Since $\nu\ne0$ we conclude that
$\tau_0$ is cohomologous to a constant, which is a contradiction in view of the 
remark following Definition \ref{def:randompartiallycaptive}.
This completes the proof.
\end{proof}

\begin{rmk}The result shows that the additional assumptions on the peripheral spectrum
made in Faure~\cite{Faure}*{Theorem 5} may be omitted.
\end{rmk}

We will now show that Bogensch{\"u}tz~\cite{Bogenschutz00}*{Theorem 1.10} 
may be applied to the perturbation $M_\nu^\ast(\omega)$ of $M_\nu^\ast(0)$
for each $\nu\in\Z$.
Admitting this for the moment, we find in view of that result that $r^\ast\{\nu\}$
converges to the spectral radius of $M_\nu^\ast(0)$ as $\epsilon\to0$,
and that if $\epsilon$ is sufficiently small then
the dimension of $\Sigma_1(\epsilon;\omega)$ is equal to the dimension of
the direct sum of the generalized eigenspaces for the eigenvalues of $M_0^\ast(0)$ on the
unit circle. Since we have already shown that $h_\epsilon(\omega)\in\Sigma_1(\epsilon;\omega)$,
statements $(\mathrm{i})$ and $(\mathrm{ii})$ in Theorem \ref{thm:discretespectrum}
therefore follow by virtue of Theorem \ref{thm:nonrandomperspec}, thus completing the proof of Theorem \ref{thm:discretespectrum}.

\begin{prop}
For all $u\in H^m$ we have
\begin{equation}\label{eq:Bog1}
 \esssup_\omega{\lVert M_\nu^\ast(\epsilon;\omega) u-
M_\nu^\ast(0) u\rVert_{(m)}}\to0,\quad\epsilon\to0.
\end{equation}
For all $\rho>\lambda^{-1}$ there is an $N_0$ such that if $n\ge N_0$ then we can find an $\epsilon_\nu(n)$
such that 
\begin{equation}\label{eq:Bog2}
 \esssup_\omega{\lVert M_{\nu,n}^\ast(\epsilon;\omega) -
M_{\nu,n}^\ast(0) \rVert_{\mathscr L(H^m)}}<\rho^{mn},\quad 0\le\epsilon<\epsilon_\nu(n).
\end{equation}
\end{prop}

Note that the proposition shows that
Bogensch{\"u}tz~\cite{Bogenschutz00}*{Theorem 1.10} is applicable,
compare with Bogensch{\"u}tz~\cite{Bogenschutz00}*{Remark 1.11}.
In particular,
$\{M_\nu^\ast(\epsilon;\cdot):\Omega\to \mathscr L(H^m):\epsilon>0\}$ is an {\it asymptotically small random
perturbation} of $M_\nu^\ast(0)$ (see Bogensch{\"u}tz~\cite{Bogenschutz00}*{Definition 1.7})
for any so-called split index $\ge\log{\rho^{m}}$, an example
of which is $r_m$ (assuming that $\lambda<2\le k$ and $\rho$ is chosen
sufficiently close to $\lambda^{-1}$).

\begin{proof}
We only prove \eqref{eq:Bog2} since \eqref{eq:Bog1} is easier. (As in the proof of strong measurability,
the proof of \eqref{eq:Bog1}
is simplified by the fact that derivatives of $u$ may be estimated using the supremum norm.)
Note also that \eqref{eq:Bog2} is known to hold in the case $\nu=0$ when $H^m$ is replaced by
the Banach space $\Ci^m$, see Baladi et al.~\cite{BKS96}*{Lemma A.1}.
We will use similar arguments after making the necessary modifications.
In the proof we let $c_{n,\epsilon}$ denote a constant which for fixed $n$
tends to 0 as $\epsilon\to0$. Similarly, $C_\epsilon$ denotes a constant independent of $n$
which tends to 0 as $\epsilon\to0$. These are allowed to change between occurrences.
We will also use the short-hand notation $\lvert \varphi\rvert=\sup_x{\lvert \varphi(x)\rvert}$
for $\varphi\in\Ci^0$. Whenever convenient we will without mention adopt the viewpoint
from Section \ref{app:wPC} and identify maps with the appropriate lifts on $\R$.

For fixed $x\in\Sone$ and $n\ge1$,
let $y_0(j)=G_0^{(n)}(x+j)$ and $y_\omega(j)=G^{(n)}_\epsilon(\omega,x+j)$ for each $0\le j\le k^n-1$ be points in the pre-images
of $x$ under $E_0^{(n)}$ and $E_\epsilon^{(n)}(\omega)$, respectively.
Using the fact that $\theta$
is $\mathbb P$-preserving together with \eqref{emin}, \eqref{eq:assumptionCrclose} and Lemma \ref{lem:technical1}, 
it is possible to check that for $\mathbb P$-almost every $\omega$ we have
\begin{equation}\label{eq:nepsilonclose}
\bigg\lvert\frac{e^{-i\nu \tau_0^{(n)}(y_0(j))}}{dE_0^{(n)}(y_0(j))/dy}
-\frac{e^{-i\nu\tau_\epsilon^{(n)}(\omega,y_\omega(j))}}{dE_\epsilon^{(n)}(\omega,y_\omega(j))/dy}
\bigg\rvert\le c_{n,\epsilon},
\end{equation}
uniformly in $x$.
We use the equivalent norm $\lVert\varphi\rVert_{H^m}^2=\sum_{j=0}^m\lVert\partial^j\varphi\rVert_{L^2}^2$ 
on $H^m$
and first treat the case $m=1$. 
Consider $\lVert M_{\nu,n}^\ast(0)u - M_{\nu,n}^\ast(\epsilon;\omega)u \rVert_{L^2}^2$ and estimate this 
using the triangle inequality followed by Young's inequality by
\begin{multline}\label{eq:triangleandYoung}
2\int_{\Sone} \bigg\lvert\sum_{j=0}^{k^n-1}
\bigg(\frac{e^{-i\nu \tau_0^{(n)}(y_0(j))}}{dE_0^{(n)}(y_0(j))/dy}
-\frac{e^{-i\nu\tau_\epsilon^{(n)}(\omega,y_\omega(j))}}{dE_\epsilon^{(n)}(\omega,y_\omega(j))/dy}\bigg)
u(y_0(j))\bigg\rvert^2 dx\\
+2\int_{\Sone} \bigg\lvert\sum_{j=0}^{k^n-1}
\frac{e^{-i\nu\tau_\epsilon^{(n)}(\omega,y_\omega(j))}}{dE_\epsilon^{(n)}(\omega,y_\omega(j))/dy}
(u(y_0(j))-u(y_\omega(j)))\bigg\rvert^2 dx.
\end{multline}
Applying the same argument multiple times and using \eqref{eq:nepsilonclose} gives 
\begin{align}
\lVert M_{\nu,n}^\ast(0)u -M_{\nu,n}^\ast(\epsilon;\omega)u\rVert_{L^2}^2
&\le 2^{k^n} c_{n,\epsilon}\sum_{j=0}^{k^n-1}\int_{\Sone}\lvert u(y_0(j))\rvert^2 dx \\
&\quad +2^{k^n}\lambda^{-2n}\sum_{j=0}^{k^n-1}\int_{\Sone}\lvert u(y_\omega(j))-u(y_0(j))\rvert^2 dx.
\end{align}
We now insert the expressions for $y_0(j)$ and $y_\omega(j)$ and make the change of variables $x+j\mapsto x$.
By another change of variables, the first integral is easily seen to be bounded by $\lvert E_0'\rvert^n\lVert u\rVert_{L^2}^2$.
For the other integral, write
\[
u(G_0^{(n)}(x))-u(G_\epsilon^{(n)}(\omega,x))=
(G_0^{(n)}(x)-G_\epsilon^{(n)}(\omega,x))\int_0^1 u'(\mathscr G_t(x))dt,
\]
where $\mathscr G_t(x)=\mathscr G_t(\epsilon;\omega,x)=tG_0^{(n)}(x)+(1-t)G_\epsilon^{(n)}(\omega,x)$.
By Lemma \ref{lem:technical1} and the Cauchy-Schwartz inequality we get
\[
\int_{\Sone}\lvert u(G_0^{(n)}(x))-u(G_\epsilon^{(n)}(\omega,x))\rvert^2 dx
\le C_\epsilon\int_{\Sone} \int_0^1\lvert u'(\mathscr G_t(x))\rvert^2 dt dx.
\]
Use Fubini's theorem to change the order of integration, followed by the change of variables
$\mathscr G_t(x)\mapsto x$, with $t$ now fixed. The resulting Jacobian is easily seen to be
uniformly bounded for $0\le t\le1$ by a constant multiple of $\lvert E_0'\rvert^n$,
$\mathbb P$-almost surely,
if $\epsilon=\epsilon(n)$ is sufficiently small. Hence,
\begin{equation}\label{eq:stepforL2}
\lVert M_{\nu,n}^\ast(0)u -M_{\nu,n}^\ast(\epsilon;\omega)u\rVert_{L^2}^2
\le c_{n,\epsilon}\lVert u\rVert_{L^2}^2+
 c_{n,\epsilon} \lVert\partial u\rVert_{L^2}^2\le c_{n,\epsilon}\lVert u\rVert_{H^1}^2.
\end{equation}

Next, consider $\lVert \partial(M_{\nu,n}^\ast(0)u - M_{\nu,n}^\ast(\epsilon;\omega)u) \rVert_{L^2}^2$.
Separate the expression inside the norm into one part where $u$ is not differentiated, and
one part where $u$ is differentiated. For the first part, the arguments above can be repeated
after having established that the derivative of \eqref{eq:nepsilonclose} can be estimated by $c_{n,\epsilon}$,
the details of which are left to the reader. Thus, by the triangle inequality followed
by Young's inequality we get (compare with \eqref{eq:triangleandYoung})
\begin{equation}\label{eq:step2H1}
\lVert \partial(M_{\nu,n}^\ast(0)u - M_{\nu,n}^\ast(\epsilon;\omega)u) \rVert_{L^2}^2
\le c_{n,\epsilon}\lVert u\rVert_{H^1}^2+R_1,
\end{equation}
where, in view of \eqref{eq:betterind00}--\eqref{eq:varphiforjequalm},
\begin{equation}
R_1=2\,\Bigg\lVert M_{\nu,n}^\ast(0)\Bigg(\frac{u'}{\partial E_0^{(n)}}\Bigg) - M_{\nu,n}^\ast(\epsilon;\omega)
\Bigg(\frac{u'}{\partial E_\epsilon^{(n)}(\omega)}\Bigg)
\Bigg\rVert_{L^2}^2.
\end{equation}
Using \eqref{emin} and Lemma \ref{lem:LYineq} the estimate
$R_1 \le C_\nu\lambda^{-2n}\lVert u'\rVert_{L^2}^2$ easily follows.
Combined with \eqref{eq:stepforL2}--\eqref{eq:step2H1} this gives
\begin{equation}
\lVert M_{\nu,n}^\ast(0)u -M_{\nu,n}^\ast(\epsilon;\omega)u\rVert_{H^1}^2
\le c_{n,\epsilon}\lVert u\rVert_{H^1}^2
+C_\nu\lambda^{-2n}\lVert u\rVert_{H^1}^2
\end{equation}
since $\lVert u'\rVert_{L^2}^2\le \lVert u\rVert_{H^1}^2$,
which yields \eqref{eq:Bog2} in the case $m=1$.

For general $m$, an induction argument together with properties of the norm shows
that it suffices to consider $\lVert \partial^m(M_{\nu,n}^\ast(0)u - M_{\nu,n}^\ast(\epsilon;\omega)u) \rVert_{L^2}^2$, 
see \eqref{eq:mthstep}.
We split the expression into one part containing the derivatives of $u$ of order $\le m-1$,
and one part where all derivatives land on $u$. 
Arguing as in \eqref{eq:step2H1} we get
\begin{equation}\label{eq:step3Hm}
\lVert \partial^m(M_{\nu,n}^\ast(0)u - M_{\nu,n}^\ast(\epsilon;\omega)u) \rVert_{L^2}^2
\le c_{n,\epsilon}\lVert u\rVert_{H^m}^2+R_m,
\end{equation}
where
\begin{equation}
R_m=2\,\Bigg\lVert M_{\nu,n}^\ast(0)\Bigg(\frac{\partial^m u}{(\partial E_0^{(n)})^m}\Bigg) - M_{\nu,n}^\ast(\epsilon;\omega)
\Bigg(\frac{\partial^m u}{(\partial E_\epsilon^{(n)}(\omega))^m}\Bigg)
\Bigg\rVert_{L^2}^2.
\end{equation}
Using \eqref{emin} and Lemma \ref{lem:LYineq} to estimate $R_m$ we find that
\begin{equation}
\lVert \partial^m(M_{\nu,n}^\ast(0)u - M_{\nu,n}^\ast(\epsilon;\omega)u) \rVert_{L^2}^2
\le c_{n,\epsilon}\lVert u\rVert_{H^m}^2+C_{\nu,m}\lambda^{-nm}\lVert u\rVert_{H^m},
\end{equation}
from which \eqref{eq:Bog2} follows.
\end{proof}

\appendix
\section{}\label{app:PsiDO}
In this appendix we discuss pseudodifferential operators on $\Sone$,
and prove a special case of Egorov's theorem applicable to our context.
We also revisit some results concerning $L^2$ continuity of pseudodifferential operators in the presence
of a noise parameter.
This is then used to prove Theorems \ref{thm:Pn} and \ref{thm:Q}.

Let $g_\epsilon(\omega)$, $E_\epsilon(\omega)$ and $\tau_\epsilon(\omega)$ 
be defined through the perturbation scheme $\{f_\epsilon\}_{\epsilon>0}$
as in \eqref{eq:perturbedsystem}. As in Section \ref{app:wPC}, we will identify $\Sone$ with the fundamental domain
\[
\Sone\simeq\{x:0\le x<1\}\subset\R,
\]
and view $g_\epsilon(\omega): \Sone \rightarrow  \Sone$ as a smooth function $g_\epsilon(\omega):\R\to\R$
with the property that $g_\epsilon(\omega,x+1)=g_\epsilon(\omega,x)+1$ for all $x\in\R$, with an inverse $g_\epsilon^{-1}(\omega)$
enjoying the same property. Similarly, $E_\epsilon(\omega)$ will be identified with the smooth 
invertible function $E_\epsilon(\omega):\R\to\R$
given by $E_\epsilon(\omega,x)=kg_\epsilon(\omega,x)$ for $x\in\R$, so that
$E_\epsilon(\omega,x+\ell)=E_\epsilon(\omega,x)+k\ell$ for all $\ell\in\Z$ and all $x\in\R$. 
The inverse $E^{-1}(\omega,x)=g^{-1}(\omega,x/k)$ satisfies $E^{-1}(\omega,x+k)=E^{-1}(\omega,x)+1$ for $x\in\R$.
The derivative
$E_\epsilon'(\omega):\R\to\R$ is a smooth periodic function with period 1.
Similarly,
any smooth function $\tau_\epsilon(\omega):\Sone\to\R$ will be identified with a periodic function
$\tau_\epsilon(\omega):\R\to\R$ with period 1.
We shall assume that $\epsilon$ belongs to the range $0\le\epsilon<\epsilon_0$,
where $\epsilon_0$ is sufficiently small for \eqref{emin} to hold, that is, $E_\epsilon(\omega)$
is an expanding map, $\mathbb{P}$-almost surely. We will also make extensive use of \eqref{eq:assumptionCrclose}.
When there is no ambiguity, the noise level $\epsilon$ will mostly be omitted from the notation,
in particular when the dependence on the noise parameter $\omega\in\Omega$ is
already displayed. Thus, we shall for example mostly write $E(\omega)$ instead of $E_\epsilon(\omega)$
and $\tau(\omega)$ instead of $\tau_\epsilon(\omega)$.

We shall continue to use the notation $E^{(n)}(\omega)$ and $\tau^{(n)}(\omega)$
for the functions introduced in \eqref{notation:E^{(n)}}--\eqref{notation:tau^{(n)}}
lifted on $\R$. We let $G(\omega)=E^{-1}(\omega)$ denote the inverse of $E(\omega)$,
and let $G^{(n)}(\cdot)$ be the backward cocycle induced by $G$,
so that $G^{(n)}(\omega)$ is the inverse of $E^{(n)}(\omega)$.

Let $\Ci_p^\infty(\R)$ denote the class of periodic smooth
functions on $\R$ with period 1. 
Let $h>0$ be a semiclassical parameter. In order to adhere to the notation used in the main body of this paper,
introduce the real parameter $\nu=1/h$ and consider the operator $M_{\nu,n}(\omega)=M_{1/h,n}(\omega)$ defined for $n\ge1$ by
\begin{equation}\label{MonRgeneraltime}
M_{\nu,n}(\omega)u(x)=e^{i\tau^{(n)}(\omega,x)/h}u(E^{(n)}(\omega,x)),\quad u\in\Ci_p^\infty(\R),
\end{equation}
compare with the remark in Section \ref{subsection:reduction}.
We note that if $u\in\Ci_p^\infty(\R)$ and $\ell\in\Z$ then 
$M_{\nu,n}(\omega)u(x+\ell)=M_{\nu,n}(\omega)u(x)$ so $M_{\nu,n}(\omega)$ preserves periodicity.
When $\nu$ is restricted to the set of positive integers, it follows that
$M_{\nu,n}(\omega)$ can be identified with the operator given by \eqref{restrictionoperator}.

Let $\Es(\R)$ be the space of Schwartz functions on $\R$. For $\ell\in\Z$, we let $T_\ell$ be the
translation operator $T_\ell v(x)=v(x-\ell)$ for $v\in\Es(\R)$,
and we extend $T_\ell$ to the space $\Es'(\R)$ of tempered distributions on $\R$ by duality.
In view of \eqref{restrictionoperator}--\eqref{adjointrestrictionoperator}
and the discussion connected to \eqref{eq:setequality},
it follows that the $L^2(\Sone)$ adjoint of
the operator given by \eqref{MonRgeneraltime} is
\begin{equation}\label{MstaronRgeneraltime}
M_{\nu,n}^\ast(\omega) u(x)=\sum_{\bar\alpha=0}^{k^n-1}
T_{-\bar\alpha}(G^{(n)}(\omega))^\ast\bigg(\frac{e^{-i\tau^{(n)}(\omega)/h}}{\partial E^{(n)}(\omega)}
u\bigg)(x),\quad u\in\Ci_p^\infty(\R),
\end{equation}
where we by abuse of notation write $\partial E^{(n)}(\omega)$ for the derivative 
$x\mapsto dE^{(n)}(\omega,x)/dx$.

\subsection{Pseudodifferential operators on $\Sone$}

The general definition of semiclassical (pseudodifferential) operators on manifolds
has the disadvantage of not having globally defined symbols.
Instead, we will use the identification described above 
concerning functions on $\Sone$ and periodic functions on $\R$,
and identify semiclassical operators acting on ${\Ci}^\infty(\Sone)$ with
so-called periodic semiclassical operators on $\R$. The latter class consists of
semiclassical operators $a(x,hD)$ on $\R$ with symbols periodic in $x$,
\[
a(x+\ell,\xi)=a(x,\xi),\quad\ell\in\Z,
\]
acting on periodic functions, see for example Zworski~\cite{Zworski}*{Section 5.3}
where the $n$ dimensional torus $\T^n$ is considered. 
We mention here that 
Ruzhansky and Turunen~\cite{RuzhanskyTurunenPaper}
have developed an equivalent pseudodifferential calculus for operators on $\T^n$
using Fourier series and globally defined toroidal symbols. This was the viewpoint we used
in Section \ref{subsection:reduction},
where a semiclassical operator $a(x,hD)$ acting on ${\Ci}^\infty(\Sone)$ was defined as
acting through multiplication by $a(x,h\xi)$ on the side of Fourier coefficients,
compare with \eqref{eq:toroidalsymbols}. 
We briefly discuss the equivalence below.
For more on the development of this
topic we refer to the works mentioned above and the references therein.
As we shall see, in order to prove Theorem \ref{thm:Pn}
we will have to revisit most of the details of the
calculus in the presence of a noise parameter $\omega$;
our choice of approach is due mostly to being more familiar
with the standard calculus.

For $m\in\R$ we let $S^m=S^m(\R\times\R)$ be the symbol class
of functions 
$a\in {\Ci}^\infty (\R^{2})$ such that for all integers 
$\alpha,\beta\in\N$ the derivative
$a_{(\beta)}^{(\alpha)}(x,\xi)=\partial_\xi^\alpha D_{x}^\beta a(x,\xi)$ has the bound
\begin{equation}\label{symbolclass}
\lvert a_{(\beta)}^{(\alpha)}(x,\xi)\rvert
\le C_{\alpha,\beta}(1+\lvert\xi\rvert)^{m-\lvert\alpha\rvert},\quad x, \xi\in\R.
\end{equation}
Here $D_{x}=-i\partial_{x}$.
Note that the symbol $a$ is permitted to depend on $h$, but we require the estimates
to be uniform for $h$ in some interval $0<h\le h_0$. 
$S^m$ is a Fr{\'e}chet space with semi-norms given by the smallest constants which can be used in \eqref{symbolclass},
\begin{equation}\label{eq:seminormssingle}
\lvert a\rvert_\ell^{(m)}
=\max_{\lvert\alpha+\beta\rvert\le\ell}
\sup_{x,\xi}{\{\lvert a_{(\beta)}^{(\alpha)}(x,\xi)\rvert
(1+\lvert\xi\rvert)^{-(m-\lvert\alpha\rvert)}\}},
\quad a\in S^m.
\end{equation}
If $a\in S^m(\R\times\R)$ and $A$ is the semiclassical operator $A=a(x,hD)$, then we write $A\in\Psi^m(\R)$,
and we have
\begin{equation}\label{eq:PsiDO}
Au(x)=a(x,hD)u(x)=\frac{1}{2\pi h}\iint e^{i(x-y)\xi/h}a(x,\xi)u(y)dyd\xi,\quad u\in\Es(\R),
\end{equation}
where the right-hand side may be interpreted as an iterated integral, or as the limit
in $\Es'$ when we approach $a$ by a sequence of functions in $\Es$.
The pseudodifferential operator $a(x,D)$ is obtained by setting $h=1$ in \eqref{eq:PsiDO}.

Now suppose $u\in\Ci_p^\infty(\R)$ and let $a\in S^m(\R\times\R)$. 
Then $a(x,hD):{\Ci}^\infty(\R)\to {\Ci}^\infty(\R)$ is continuous, so
the action of $a(x,hD)$ on $u$ is well defined.
If in addition $a$ is periodic
in $x$, it follows that
\[
(a(x,hD)u)(x+\ell)=(a(x,hD)u(\cdot+\ell))(x),
\]
so $a(x,hD)$ preserves
periodicity. Hence we may regard $a(x,hD):\Ci_p^\infty(\R)\to\Ci_p^\infty(\R)$
as an operator on ${\Ci}^\infty(\Sone)$.
Moreover, let $b\in {\Ci}^\infty(\Sone\times\R)$ and identify $b$ with an
element in ${\Ci}^\infty(\R_x\times\R_\xi)$ which is periodic in $x$.
Suppose that the latter belongs to $S^m(\R\times\R)$. 
Define an operator $b(x,hD)$ on ${\Ci}^\infty(\Sone)$ 
acting through multiplication by $b(x,h\xi)$ on the side of Fourier coefficients,
\begin{equation}\label{eq:PsiDOonS1}
b(x,hD)v(x)=\sum_{\xi\in 2\pi\Z} b(x,h\xi)\hat v(\xi)e^{ix\xi},\quad v\in {\Ci}^\infty(\Sone).
\end{equation}
Using the identification above, define a semiclassical operator $B$
on $\R$ with symbol $b\in S^m(\R\times\R)$ in accordance with \eqref{eq:PsiDO}. Then
$B:\Ci_p^\infty(\R)\to\Ci_p^\infty(\R)$ can be identified with the operator given by \eqref{eq:PsiDOonS1}.
Indeed, identify $v\in {\Ci}^\infty(\Sone)$ with $v\in\Ci_p^\infty(\R)$, and write the latter
as a Fourier series $v(y)=\sum\hat v(\xi)e^{iy\xi}$, convergent in ${\Ci}^\infty(\R)$. Recall that
\[
e^{-ix\xi}B(e^{i(\cdot)\xi})(x)=e^{-ixh\xi/h}B(e^{i(\cdot)h\xi/h})(x)=b(x,h\xi),
\]
see Zworski~\cite{Zworski}*{Theorem 4.19}. Since $B:{\Ci}^\infty(\R)\to {\Ci}^\infty(\R)$ is continuous,
it follows that $Bv(x)=\sum B(e^{i(\cdot)\xi})(x)\hat v(\xi)=\sum b(x,h\xi) \hat v(\xi) e^{ix\xi}$
with convergence in ${\Ci}^\infty(\R)$.

\begin{rmk}\label{rmk:notation}
Based on the previous discussion, we will permit us to use the notation
$a\in S^m(T^\ast(\Sone))$ to express the fact that $a$ is an element in $S^m(\R_x\times\R_\xi)$
which is periodic in $x$.
\end{rmk}

\subsection{Egorov's theorem}
We now return to the problem at hand. We aim to show that if $\omega\in\Omega$
and $A$ is a semiclassical operator on the one dimensional torus, then
the same is true for the conjugation $M_{\nu,n}^\ast(\omega)\circ A\circ M_{\nu,n}(\omega)$, and we shall calculate the symbol.
This is a special case of general results concerning conjugations with Fourier integral operators,
collectively known as Egorov's theorem. However, due to the relatively simple nature of our problem,
it is possible to give a direct proof using only properties of pseudodifferential operators.
In doing so, we shall also be interested in tracking certain properties of uniformity with respect to the
noise parameter $\omega$, as well as time $n\in\N^\ast$.
This will be made possible by using the pseudodifferential calculus for operators with so-called double symbols
as presented by Kumano-go~\cite{Kumano-go}, and we are indebted to Yoshinori Morimoto for the suggestion.
We will briefly recall this calculus, and refer the reader to the book by 
Kumano-go~\cite{Kumano-go}*{Chapter 2, \S 2} for a detailed exposition in the case $h=1$.

For a ${\Ci}^\infty$ function in $\R_x\times\R_\xi\times\R_{x'}\times\R_{\xi'}$ and
positive integers (or more generally, multi-indices) 
$\alpha, \alpha', \beta$ and $\beta'$,
we will use the notation
\begin{equation}\label{derivativenotationfordoublesymbols}
p_{(\beta,\beta')}^{(\alpha,\alpha')}(x,\xi,x',\xi')=\partial_\xi^\alpha\partial_{\xi'}^{\alpha'} D_x^\beta D_{x'}^{\beta'} p(x,\xi,x',\xi').
\end{equation}
A ${\Ci}^\infty$ function $p$ 
belongs to the symbol class 
$S^{m,m'}=S^{m,m'}(\R\times\R\times\R\times\R)$
if for any $\alpha, \alpha', \beta ,\beta'$ there exists a constant $C_{\alpha, \alpha', \beta ,\beta'}$ such that
\begin{equation}\label{doublesymbolclass}
\lvert p_{(\beta,\beta')}^{(\alpha,\alpha')}(x,\xi,x',\xi')\rvert 
\le C_{\alpha, \alpha', \beta ,\beta'}
(1+\lvert \xi\rvert )^{m-\lvert \alpha\rvert }(1+\lvert \xi'\rvert )^{m'-\lvert \alpha'\rvert }.
\end{equation}
$S^{m,m'}$ is a Fr{\'e}chet space with semi-norms
\begin{equation}\label{eq:seminormsdouble}
\lvert p\rvert _\ell^{(m,m')}
=\max_{\lvert \alpha+\alpha'+\beta+\beta'\rvert \le\ell}\inf{\{ C_{\alpha, \alpha', \beta ,\beta'}
\text{ for which \eqref{doublesymbolclass} holds}\}}. 
\end{equation}
A symbol of class $S^{m,m'}$ is called a double symbol. Let
$\mathscr B$ be the space of smooth functions
with bounded derivatives of any order. For a double symbol $p$ in 
$S^{m,m'}$ we define the operator $P=p(x,hD_x,x',hD_{x'})$ acting on $u\in\mathscr B$ by
\begin{equation*}
Pu(x)=\frac{1}{(2\pi h)^2}\iint e^{-i(y\xi+y'\xi')/h} p(x,\xi,x+y,\xi')u(x+y+y') dy dy' d\xi d\xi',
\end{equation*}
where the integral can be interpreted
as the limit in $\Es'$ when we approach $p$ by a sequence of functions in $\Es$.
We remark that when $p$ in $S^{m,0}$ is independent of $\xi'$, then
\begin{equation}\label{operatorwithdoublesymbol}
Pu(x)=\frac{1}{2\pi h}\iint e^{i(x-x')\xi/h} p(x,\xi,x')u(x') dx' d\xi,
\quad u\in\Es,
\end{equation}
see Corollary $3^\circ$ of Lemma 2.3 in Kumano-go~\cite{Kumano-go}*{Chapter 2}.

The following special case of the expansion formula given by Theorem 3.1 in Kumano-go~\cite{Kumano-go}*{Chapter 2}
will be important. We state it here for easy reference, where we also include the results of Theorem 2.5 in Kumano-go~\cite{Kumano-go}*{Chapter 2}.

\begin{thm}\label{thm:simplifiedsymbol}
For $p\in S^{m,m'}(\R\times\R\times\R\times\R)$ and $P=p(x,hD_x,x',hD_{x'})$, set $p_0(x,\xi)=p(x,\xi,x,\xi)$
and
\[
r_\vartheta(x,\xi)=\frac{1}{2\pi h}\iint e^{-iy\eta/h} p_{(0,1)}^{(1,0)}(x,\xi+\vartheta\eta,x+y,\xi) dy d\eta,
\]
interpreted as an oscillatory integral. Then $p_0$ belongs to $S^{m+m'}(\R\times\R)$
and $\{r_\vartheta\}_{\lvert\vartheta\rvert\le 1}$ is a bounded set of $S^{m+m'-1}(\R\times\R)$,
and for any $\ell$ there exists a constant $C_\ell$ and an integer $\ell'$,
independent of $\vartheta$, such that
\begin{equation}\label{seminormcontinuity}
\lvert r_\vartheta\rvert_\ell^{(m+m'-1)}\le C_\ell \lvert p\rvert_{\ell'}^{(m,m')},
\end{equation}
where $\ell'$ depends only on $\ell$, $m$, $m'$ and the dimension $\dim{\R}=1$. Moreover,
\begin{equation*}
P=p_0(x,hD)+hp_1(x,hD),
\end{equation*}
where $p_1(x,\xi)=\int_0^1 r_\vartheta(x,\xi)d\vartheta$. The function $p_L(x,\xi)=p_0(x,\xi)+hp_1(x,\xi)$
is called the simplified symbol. The map $S^{m,m'}\ni p\mapsto p_L\in S^{m+m'}$ is continous,
and for any $\ell$ there exists a constant $C_\ell$ and an integer $\ell'$ as above, such that
\begin{equation}\label{seminormcontinuity2}
\lvert p_L\rvert_\ell^{(m+m')}\le C_\ell \lvert p\rvert_{\ell'}^{(m,m')}.
\end{equation}
\end{thm}

Note that Theorem 3.1 in Kumano-go~\cite{Kumano-go}*{Chapter 2} provides a full asymptotic expansion for the simplified symbol $p_L$,
but as mentioned above we shall only use the special case presented here.
Also, it will sometimes be convenient to identify a symbol $p_0\in S^m$ 
with a double symbol $p\in S^{m,0}$ which is independent of both $x'$ and $\xi'$.
This makes sense since by Theorem \ref{thm:simplifiedsymbol} we
then have $p(x,hD_x,x',hD_{x'})=p_0(x,hD_x)$.

Recall that $T_\ell$ denotes the translation operator $T_\ell v(x)=v(x-\ell)$. We will use
the notation $\mathscr T_\ell$ for the translation operator $\mathscr T_\ell w(x,\xi,x',\xi')=w(x-\ell,\xi,x'-\ell,\xi')$.
Most of the double symbols we will be working with will be 
invariant under the action of $\mathscr T_\ell$ for $\ell\in\Z$, that is,
\begin{equation}\label{eq:doublesymbolperiodicity}
p(x+\ell,\xi,x'+\ell,\xi')=p(x,\xi,x',\xi'),\quad\ell\in\Z.
\end{equation}
From the definition of the simplified symbol, it then immediately follows that if
$p\in S^{m,m'}$ satisfies \eqref{eq:doublesymbolperiodicity} then 
$p_L\in S^{m+m'}(T^\ast (\Sone))$,
that is, $p_L$ is periodic in $x$. Moreover, it is straightforward to check that
$P=p(x,hD_x,x',hD_{x'})$ preserves the periodicity of $u\in\Ci_p^\infty(\R)$, that is,
$Pu(x+\ell)=Pu(x)$ for $\ell\in\Z$. The same is true if $p$ in addition is independent of $\xi'$, which
follows from \eqref{operatorwithdoublesymbol} by a change of variables. Hence, we can think of these
operators as acting on ${\Ci}^\infty(\Sone)$, and in this context we say that $P$
is a semiclassical operator on $\Sone$ with double symbol $p$ in $S^{m,m'}$ 
(with $m'=0$ if $p$ is independent of $\xi'$) satisfying \eqref{eq:doublesymbolperiodicity}.

The proof of Egorov's theorem in the form that we shall need will be split into lemmas.
The reader may want to consult the proof of Theorem \ref{thm:Egorov'sgeneraltime}
for an explanation of how these lemmas will be used.

\begin{lem}\label{lem:multiplicationcomposition}
Let $m>0$. Let $a_m(\xi)$ be the escape function defined by \eqref{eq:escapefunction},
considered as a function in ${\Ci}^\infty(\R_x\times\R_\xi)$ which is constant in $x$.
Let $A_m$ be the semiclassical operator with symbol $a_m\in S^{m}(T^\ast(\Sone))$.
Then the operator $A(\omega)\equiv A(\epsilon;\omega,n,m)$ given by
\[
A(\omega)u(x)= \frac{A_m^{-2}u(x)}{dE^{(n)}(\omega,x)/dx} ,\quad u\in\Ci_p^\infty(\R),
\]
is $\mathbb{P}$-almost surely a semiclassical operator on $\Sone$ with double symbol 
$a(\omega)$ in $S^{-2m,0}$,
where
\begin{equation}\label{eq:symbolofmultiplication}
a(\omega,x,\xi)=a_m^{-2}(\xi)(dE^{(n)}(\omega,x)/dx)^{-1}
\end{equation}
is independent of $x'$ and $\xi'$ and satisfies \eqref{eq:doublesymbolperiodicity}.
For every $\ell\in\N$ there is an $\epsilon_0(\ell)$, independent of time $n\in\N^\ast$,
such that if $0\le\epsilon<\epsilon_0$
then
\begin{equation}\label{eq:preconditionA}
\lvert a(\omega)\rvert_\ell^{(-2m,0)}\le C_{\ell,n,m}
\end{equation}
$\mathbb{P}$-almost surely,
where the constant is independent of $\omega$.
\end{lem}

\begin{proof}
First note that since $a_m$ does not depend on the base variable,
it follows from Theorem \ref{thm:simplifiedsymbol} that the simplified symbol of $A_m^{-2}=A_m^{-1}\circ A_m^{-1}$
is equal to $a_m^{-2}$. 
By the definition of $a_m$ we also have that
$a_m^{-2}=a_{-m}^2$ which is an element in $S^{-2m}$. As shown by the discussion following Theorem \ref{thm:simplifiedsymbol},
we may view $a_m^{-2}$ as a double symbol in $S^{-2m,0}$ which is independent of both $x'$ and $\xi'$.

For almost every $\omega$ we may interpret multiplication by $(dE^{(n)}(\omega,x)/dx)^{-1}$ as the action of a semiclassical 
operator with symbol $(x,\xi)\mapsto (dE^{(n)}(\omega,x)/dx)^{-1}$. Indeed,
\begin{equation}\label{eq:derivativeofE(n)}
dE^{(n)}(\omega,x)/dx=\prod_{j=0}^{n-1}
E'(\theta^j\omega,E^{(j)}(\omega,x)),
\end{equation}
where $E'(\omega)$ is periodic and
$\mathbb{P}$-almost surely satisfies $E'(\omega,x)>\lambda$ for all $x\in\R$ by \eqref{emin},
where $\lambda>1$.
It is straightforward to check
that all $x$ derivatives of $1/E'(\omega,x)$ are uniformly bounded on $\R$,
while all $\xi$ derivatives vanish. 
Since $\theta$ is $\mathbb{P}$-preserving, it follows that 
$(dE^{(n)}(\omega,x)/dx)^{-1}$ belongs to $S^0(T^\ast(\Sone))$, $\mathbb{P}$-almost surely,
and the claim follows by Zworski~\cite{Zworski}*{Theorem 4.3}. 
Clearly, the symbol of $A(\omega)$ is given by \eqref{eq:symbolofmultiplication}
in view of \eqref{operatorwithdoublesymbol}. It is also clear that $a(\omega)$ satisfies \eqref{eq:doublesymbolperiodicity}.

Let $\ell\in\N$. By virtue of \eqref{eq:assumptionCrclose}, we have
\begin{equation}\label{eq:boundsforE}
\esssup_\omega{\sup_x{\lvert\partial_x^\beta(E_\epsilon(\omega,x)-E_0(x))\rvert}}<1
\end{equation}
for all $\lvert\beta\rvert\le\ell_0+1$ and $0\le\epsilon<\epsilon_{\ell_0+1}$, if
$\epsilon_{\ell_0+1}$ is chosen sufficiently small.
Since $\theta$ is $\mathbb{P}$-preserving, this is also true for $E(\theta^j\omega)$ for all $j$.
Using also \eqref{emin} and \eqref{eq:derivativeofE(n)} together with
Fa{\`a} di Bruno's formula and the Leibniz rule, it follows that
if $0\le\epsilon<\epsilon_{\ell_0+1}$ then
\eqref{eq:preconditionA} holds.
This completes the proof.
\end{proof}

For $n\in\N^\ast$, let $\phi^{(n)}(\omega)$ be the function in ${\Ci}^\infty(\R^2)$ defined by
\begin{equation}\label{eq:phigeneraltime}
\phi^{(n)}(\omega,x,x')=\int_0^1 \frac{d\tau^{(n)}(\omega,tx+(1-t)x')}{dy}dt.
\end{equation}

\begin{lem}\label{lem:conjugationexponential}
Let $m>0$.
Let $A(\omega)\equiv A(\epsilon;\omega,n,m)$ be a semiclassical operator on $\Sone$ with double symbol $a(\omega)$
in $S^{-2m,0}$, where $a(\omega,x,\xi,x')$ is independent of $\xi'$ and satisfies \eqref{eq:doublesymbolperiodicity}.
Assume that for every $\ell\in\N$ there is an $\epsilon_0(\ell)$, independent of time $n\in\N^\ast$,
such that if $0\le\epsilon<\epsilon_0$ then \eqref{eq:preconditionA} holds.
Let $\phi^{(n)}(\omega)$ be given by \eqref{eq:phigeneraltime}.
Then the operator $B(\omega)\equiv B(\epsilon;\omega,n,m)$ given by
\[
B(\omega)u(x)=e^{-i\tau^{(n)}(\omega,x)/h} A(\omega) (e^{i\tau^{(n)}(\omega)/h}u)(x)
\]
is $\mathbb{P}$-almost surely a semiclassical operator on $\Sone$ with double symbol $b(\omega)$ in $S^{-2m,0}$,
where
\begin{equation}\label{eq:symbolofconjugationTau}
b(w,x,\xi,x')=a(\omega,x,\xi+\phi^{(n)}(\omega,x,x'),x')
\end{equation}
is independent of $\xi'$ and satisfies \eqref{eq:doublesymbolperiodicity}
and \eqref{eq:preconditionA}.
\end{lem}

\begin{proof}
Note that $(x-y)\phi^{(n)}(\omega,x,y)=\tau^{(n)}(\omega,x)-\tau^{(n)}(\omega,y)$
by definition.
In view of \eqref{operatorwithdoublesymbol} we have
\begin{align}
B(\omega)u(x)&=\frac{1}{2\pi h}\iint e^{\frac{i}{h}((x-y)\xi-\tau^{(n)}(\omega,x)+\tau^{(n)}(\omega,y))}
a(\omega,x,\xi,y)u(y)dyd\xi
\\ & =\frac{1}{2\pi h}\iint e^{\frac{i}{h}((x-y)(\xi-\phi^{(n)}(\omega,x,y))}a(\omega,x,\xi,y)u(y)dyd\xi,
\end{align}
so a change of variables now shows that $B(\omega)$ has double symbol $b(\omega)$
given by \eqref{eq:symbolofconjugationTau}. 
Note that $E^{(j)}(\omega,x+\ell)=E^{(j)}(\omega,x)+k^j\ell$ for $\ell\in\Z$.
Using the periodicity of $E'(\omega)$ and $\tau'(\omega)$ together with
\eqref{notation:tau^{(n)}} and \eqref{eq:derivativeofE(n)}, it is then straightforward to check that
$\phi^{(n)}(\omega,x+\ell,y+\ell)=\phi^{(n)}(\omega,x,y)$ for $\ell\in\Z$. Since $a$ satisfies \eqref{eq:doublesymbolperiodicity},
it follows that $b(\omega)$ satisfies \eqref{eq:doublesymbolperiodicity}.
Using \eqref{notation:tau^{(n)}} and \eqref{eq:derivativeofE(n)}, 
it is also straightforward to check that $b(\omega)$ belongs to $S^{-2m,0}$,
$\mathbb{P}$-almost surely. Moreover, given $\ell\in\N$, 
we have by virtue of \eqref{eq:assumptionCrclose} that
\begin{equation}\label{eq:boundsfortau}
\esssup_\omega{\sup_x{\lvert\partial_x^\beta(\tau_\epsilon(\omega,x)-\tau_0(x))\rvert}}<1
\end{equation}
for all $\lvert\beta\rvert\le\ell_0+1$ and $0\le\epsilon<\epsilon_{\ell_0+1}$, if
$\epsilon_{\ell_0+1}$ is chosen sufficiently small. By decreasing $\epsilon_{\ell_0+1}$
if necessary we may also assume that \eqref{eq:boundsforE} is in force.
Arguing as in the proof of Lemma \ref{lem:multiplicationcomposition},
we conclude that $b(\omega)$ satisfies \eqref{eq:preconditionA}.
This completes the proof.
\end{proof}

For $n\in\N^\ast$, let $\psi^{(n)}(\omega)$ be the function in ${\Ci}^\infty(\R^2)$ defined by
\begin{equation}\label{eq:psigeneraltime}
\psi^{(n)}(\omega,x,x')=\int_0^1\frac{dG^{(n)}(\omega,tx+(1-t)x')}{dy}dt.
\end{equation}

\begin{lem}\label{lem:conjugationdiffeomorphism}
Let $A(\omega)\equiv A(\epsilon;\omega,n,m)$ be a semiclassical operator on $\Sone$
satisfying the assumptions of Lemma \ref{lem:conjugationexponential}.
Let $\psi^{(n)}(\omega)$ be given by \eqref{eq:psigeneraltime}
and set
\begin{equation}\label{eq:nottranslateddiffeomorphism}
\tilde a(\omega,x,\eta,x')=\frac{ dG^{(n)}(\omega,x')/dx'}{\psi^{(n)}(\omega,x,x')}
a(\omega,G^{(n)}(\omega,x),\eta/\psi^{(n)}(\omega,x,x'),G^{(n)}(\omega,x')).
\end{equation}
Then the operator $B(\omega)\equiv B(\epsilon;\omega,n,m)$ given by
\[
B(\omega)u(x)=\sum_{\bar\alpha=0}^{k^n-1}T_{-\bar\alpha}(G^{(n)}(\omega))^\ast \circ A(\omega) \circ (E^{(n)}(\omega))^\ast u(x)
\]
is $\mathbb{P}$-almost surely a semiclassical operator on $\Sone$ with double symbol $b(\omega)$ in $S^{-2m,0}$,
where
\begin{equation}\label{eq:translateddiffeomorphism}
b(w,x,\xi,x')=\sum_{\bar\alpha=0}^{k^n-1} \mathscr T_{-j} 
\tilde a(\omega,x,\xi,x')
\end{equation}
is independent of $\xi'$ and satisfies \eqref{eq:doublesymbolperiodicity} and \eqref{eq:preconditionA}.
\end{lem}

\begin{proof}
In view of \eqref{eq:assumptionCrclose} we may without loss of generality assume that
\[
1<\essinf_{\omega} \min_{x} E'(\omega,x)\le \esssup_\omega
\max_{x} E'(\omega,x)< \max_{x} E_0'(x)+1,
\]
where $E_0$ is some fixed smooth periodic function such that $\min E_0'(x)>1$. Using the definition of
$G(\omega)$, this translates into the estimates
\begin{equation}\label{eq:kappaestimatesAppendix}
(1+\lVert E_0'\rVert_{L^\infty(\R)})^{-n}<dG^{(n)}(\omega,x)/dx<1,\quad\text{$\mathbb{P}$-almost surely},
\end{equation}
compare with \eqref{eq:controlofG(n)}.
Now let $\tilde A(\omega)u(x)
=A (u\circ E^{(n)}(\omega))( G^{(n)}(\omega,x))$, so that $B(\omega)=\sum_{\bar\alpha=0}^{k^n-1}
T_{-\bar\alpha}\tilde A(\omega)$.
Then
\[
\tilde A(\omega)u(x)=\frac{1}{2\pi h}\iint e^{\frac{i}{h}(G^{(n)}(\omega,x)-y)\xi}
a(G^{(n)}(\omega,x),\xi,y)u(E^{(n)}(\omega,y)) dy d\xi.
\]
By the change of variables $E^{(n)}(\omega,y)\mapsto y$ we obtain
\begin{multline}
\tilde A(\omega)u(x)=\frac{1}{2\pi h}\iint e^{\frac{i}{h}(G^{(n)}(\omega,x)-G^{(n)}(\omega,y))\xi}\\
\times
a(G^{(n)}(\omega,x),\xi,G^{(n)}(\omega,y))u(y) (dG^{(n)}(\omega,y)/dy)dy d\xi.
\end{multline}

By definition, 
$(x-y)\psi^{(n)}(\omega,x,y)=G^{(n)}(\omega,x)-G^{(n)}(\omega,y)$.
Hence, for the phase we can write
$(G^{(n)}(\omega,x)-G^{(n)}(\omega,y))\xi=(x-y)\psi^{(n)}(\omega,x,y)\xi$.
Now make the change of variables $\eta=\psi^{(n)}(\omega,x,y)\xi$, and note that this is
$\mathbb{P}$-almost surely invertible in view
of \eqref{eq:kappaestimatesAppendix} and the definition of $\psi^{(n)}(\omega)$. Thus,
\[
\tilde A(\omega)u(x)=\frac{1}{2\pi h}\iint e^{\frac{i}{h}(x-y)\eta}
\tilde a(\omega,x,\eta,y) u(y) dy d\eta,
\]
with $\tilde a(\omega,x,\eta,y)$ as in the statement of the lemma.
Since $a(\omega)\in S^{-2m,0}$, we find 
that differentiation with respect to $\eta$ gives the estimate
\begin{align}\label{eq:inpreparartionforconditionA1}
\lvert\partial_\eta^\alpha \tilde a(\omega,x,\eta,x')\rvert
&\le \frac{C_\alpha(\omega)}{(\psi^{(n)}(\omega,x,x'))^{\lvert\alpha\rvert}}
\bigg(1+\frac{\lvert\eta\rvert}{\psi^{(n)}(\omega,x,x')}\bigg)^{-2m-\lvert\alpha\rvert}\\
&\le C_\alpha(\omega) (1+\lVert E_0'\rVert_{L^\infty(\R)})^{n\lvert\alpha\rvert}(1+\lvert\eta\rvert)^{-2m-\lvert\alpha\rvert}.
\end{align}
We can similarly check that
\begin{equation}\label{eq:inpreparartionforconditionA2}
\lvert\partial_\eta^\alpha D_x^\beta D_{x'}^{\beta'}\tilde a(\omega,x,\eta,x')\rvert
\le C_{\alpha,\beta,\beta'}(\omega)(1+\lvert\eta\rvert)^{-2m-\lvert\alpha\rvert},
\end{equation}
so $\tilde a(\omega)$ is a double symbol
in $S^{-2m,0}$, $\mathbb{P}$-almost surely.
In fact, $x$ and $x'$ derivatives landing on $a^{(\alpha)}(\omega)$ will produce terms
containing factors of the form $\eta^j a^{(\alpha+j)}(\omega)$, but 
in view of the symbol class definition,
these can be handled using the same methods.
It is also straightforward to check that if $\epsilon$ is sufficiently small
then the constants $C_\alpha(\omega)$ and $C_{\alpha,\beta,\beta'}(\omega)$
in \eqref{eq:inpreparartionforconditionA1} and \eqref{eq:inpreparartionforconditionA2}
can for all $n\ge1$ be chosen independently of $\omega$, $\mathbb{P}$-almost surely,
so $\tilde a(\omega)$ satisfies \eqref{eq:preconditionA}. In fact,
$a(\omega)$ is assumed to satisfy \eqref{eq:preconditionA}
so it suffices to check that we
can estimate derivatives of $G^{(n)}(\omega,x)$
and $(\psi^{(n)}(\omega))^{-1}$ independently of $\omega$ for all $n\ge1$,
if $\epsilon$ is sufficiently small.
But this can be done by using \eqref{eq:boundsforE} together with estimates of the type
\eqref{eq:kappaestimatesAppendix},
and arguing as in the proof of Lemma \ref{lem:multiplicationcomposition}.

Now $B(\omega)=\sum_{\bar\alpha=0}^{k^n-1}T_{-j}\tilde A(\omega)$ so
\[
B(\omega)u(x)= 
\sum_{\bar\alpha=0}^{k^n-1}\frac{1}{2\pi h}\iint e^{i(x+\bar\alpha-y)\xi/h}
\tilde a(\omega,x+\bar\alpha,\xi,y)u(y)dyd\xi,\quad u\in\Ci_p^\infty(\R).
\]
Hence, the change of variables
$y-\bar\alpha\mapsto y$ in each of the integrals 
together with periodicity of $u$ shows that the symbol $b(\omega)$ of $B(\omega)$ 
is given by \eqref{eq:translateddiffeomorphism}.

Finally, we check that $b(\omega)$ satisfies \eqref{eq:doublesymbolperiodicity}. It is
of course sufficient to prove that $b(\omega,x+1,\xi,x'+1)
=b(\omega,x,\xi,x')$.
In view of the sum in \eqref{eq:translateddiffeomorphism}, this holds as long as
\begin{equation}\label{eq:tildeaalmostperiodic}
\tilde a(\omega,x+k^n,\xi,x'+k^n)
=\tilde a(\omega,x,\xi,x').
\end{equation}
Recall that $G(\omega)=E^{-1}(\omega)$ so we have $G(\omega,x+\ell)=g^{-1}(\omega,(x+\ell)/k)$.
Since $g^{-1}(\omega,x+1)=g^{-1}(\omega,x)+1$, it follows that
$G^{(n)}(\omega,x+k^n)=G^{(n)}(\omega,x)+1$.
The same arguments show that $dG^{(n)}(\omega,x)/dx$ is periodic with period $k^n$, 
which in turn implies that
\[
\psi^{(n)}(\omega,x+k^n,x'+k^n)=\psi^{(n)}(\omega,x,x').
\]
Thus, \eqref{eq:tildeaalmostperiodic} now follows from the definition of $\tilde a(\omega)$
and the fact that $a(\omega)$ is assumed to satisfy
\eqref{eq:doublesymbolperiodicity}. This completes the proof.
\end{proof}

Combining Lemmas \ref{lem:multiplicationcomposition}, \ref{lem:conjugationexponential}
and \ref{lem:conjugationdiffeomorphism} we can obtain a formula for the double symbol of
the semiclassical operator $M_{\nu,n}^\ast(\omega)A_mM_{\nu,n}(\omega)$, see the proof of
Theorem \ref{thm:Egorov'sgeneraltime} below. For this double symbol, the following definition is applicable.

\begin{dfn}\label{def:conditionA}
A family $\{a(\omega)\}_{\omega\in\Omega}$ of symbols 
$a(\omega)\equiv a(\epsilon;\omega,n)$ in $S^m$ (or double symbols in $S^{m,m'}$)
is said to satisfy condition $(\mathscr A{(m)})$ (or $(\mathscr A{(m,m')})$) 
if for any $\ell\in\N$ there is a positive number $\epsilon_0$
(allowed to depend on $\ell$ and $m$ (and $m'$) but
not on time $n\in\N^\ast$), and a constant $C_{\ell,n}$
such that for $0\le\epsilon<\epsilon_0$ we have
\[
\esssup_{\omega}{\lvert a(\omega)\rvert_\ell^{(m)}}\le C_{\ell,n}
\quad (\text{or }\esssup_{\omega}{\lvert a(\omega)\rvert_\ell^{(m,m')}}\le C_{\ell,n}),
\]
uniformly for $h$ in the interval $0<h\le 1$.
Here the semi-norms $\lvert\phantom{i}\rvert_\ell^{(m)}$ and $\lvert\phantom{i}\rvert_\ell^{(m,m')}$
are given by \eqref{eq:seminormssingle} and \eqref{eq:seminormsdouble}, respectively.
If $a(\omega)$ belongs to a family satisfying condition $(\mathscr A(m))$ or $(\mathscr A(m,m'))$
we shall say that $a(\omega)$ satisfies condition $(\mathscr A(m))$ or $(\mathscr A(m,m'))$.
\end{dfn}

By the definition of the semi-norms $\lvert\phantom{i}\rvert_\ell^{(m)}$ in $S^m$ it follows that if
$a(\omega)$ satisfies condition $(\mathscr A(m-1))$ for some $m$ 
then $a(\omega)$ and $ha(\omega)$ also satisfy condition $(\mathscr A(m))$.
Before proceeding to the proof of the version of Egorov's theorem that we need,
we digress to make some more general remarks which show how Definition \ref{def:conditionA} will be used.

\begin{lem}\label{lem:conditionA}
Let the double symbol $p(\omega)$ in $S^{m,m'}$ satisfy condition $(\mathscr A(m,m'))$. 
Then the simplified symbol $p_L(\omega)=p_0(\omega)+hp_1(\omega)$ in $S^{m+m'}$ 
satisfies condition $(\mathscr A(m+m'))$, where $p_0(\omega)$ and $p_1(\omega)$
are defined as in Theorem \ref{thm:simplifiedsymbol}. Moreover, $p_0(\omega)$ satisfies condition $(\mathscr A(m+m'))$
and $p_1(\omega)$ satisfies condition $(\mathscr A(m+m'-1))$.
\end{lem}

\begin{proof}
By Theorem \ref{thm:simplifiedsymbol}, the map $S^{m,m'}\ni p\mapsto p_L\in S^{m+m'}$ is continuous
in the sense of \eqref{seminormcontinuity2}. In view of Definition \ref{def:conditionA}
it is therefore clear that the simplified symbol $p_L(\omega)$ satisfies condition $(\mathscr A(m+m'))$.
Hence, if we show that $p_1(\omega)$ satisfies condition $(\mathscr A(m+m'-1))$,
then by the comment following Definition \ref{def:conditionA} together with the triangle inequality
it follows that $p_0(\omega)$ satisfies condition $(\mathscr A(m+m'))$ for $\lvert h\rvert\le1$.
Now, by \eqref{seminormcontinuity} it follows that 
the function $r_\vartheta(\omega)$ defined as in Theorem \ref{thm:simplifiedsymbol}
satisfies condition $(\mathscr A(m+m'-1))$ uniformly with respect to $\vartheta\in[0,1]$.
This implies that $p_1(\omega)$ satisfies condition $(\mathscr A(m+m'-1))$,
for $\{r_\vartheta(\omega)\}_{\lvert\vartheta\rvert\le 1}$ is a bounded set of $S^{m+m'-1}$
for fixed $\omega$, so we may differentiate under the integral sign in the definition of $p_1(\omega)$.
This completes the proof.
\end{proof}

As a first application of Lemma \ref{lem:conditionA} we prove the following result.

\begin{prop}\label{prop:composition}
Let $A(\omega)$ and $B(\omega)$ be semiclassical operators on $\Sone$ with symbols $a(\omega)\in S^m(T^\ast(\Sone))$
and $b(\omega)\in S^{m'}(T^\ast(\Sone))$ satisfying conditions $(\mathscr A(m))$ and $(\mathscr A(m'))$, respectively.
Then the composed operator $A(\omega)B(\omega)$ is a
semiclassical operator on $\Sone$ with
symbol $c(\omega)\in S^{m+m'}(T^\ast(\Sone))$ satisfying condition $(\mathscr A(m+m'))$.
Moreover, $c(\omega)=c_0(\omega)+hc_1(\omega)$, where $c_0(\omega,x,\xi)=a(\omega,x,\xi)b(\omega,x,\xi)$
satisfies condition $(\mathscr A(m+m'))$ and $c_1(\omega)$ satisfies condition $(\mathscr A(m+m'-1))$.
\end{prop}

\begin{proof}
Define a double symbol $\tilde c(\omega)$ by $\tilde c(\omega,x,\xi,x',\xi')=a(\omega,x,\xi)b(\omega,x',\xi')$.
It is straightforward to check that the hypotheses imply that 
$\tilde c(\omega)$ satisfies \eqref{eq:doublesymbolperiodicity} and
condition $(\mathscr A(m,m'))$. By Corollary $2^\circ$ of Lemma 2.3 in Kumano-go~\cite{Kumano-go}*{Chapter 2},
we have that $A(\omega)B(\omega)=\tilde c(\omega,x,hD_x,x',hD_{x'})$. In view of Theorem \ref{thm:simplifiedsymbol},
the result now follows by an application of Lemma \ref{lem:conditionA}.
\end{proof}

We are now ready to prove the version of Egorov's theorem that we need.
This is then used to calculate the symbol of the operator $P_n(\omega)$ appearing in Theorem \ref{thm:Pn}.

\begin{thm}\label{thm:Egorov'sgeneraltime}
Let $m>0$. Let $a_m(\xi)$ be the escape function defined by \eqref{eq:escapefunction},
considered as a function in ${\Ci}^\infty(\R_x\times\R_\xi)$ which is constant in $x$.
Let $A_m$ be the semiclassical operator with symbol $a_m\in S^{m}(T^\ast(\Sone))$.
Then the operator $B(\omega)\equiv B(\epsilon;\omega,n,m)$ given by
$B(\omega)=M_{\nu,n}^\ast(\omega) A_m^{-2}M_{\nu,n}(\omega)$ is
$\mathbb{P}$-almost surely a semiclassical operator on $\Sone$
with simplified symbol $b_L(\omega)=b_{0}(\omega)+hb_{1}(\omega)$ in $S^{-2m}(T^\ast(\Sone))$
satisfying condition $(\mathscr A(-2m))$.
The principal symbol $b_{0}(\omega)$ is given by
\begin{equation}\label{eq:principalsymbolBgeneraltime}
b_{0}(\omega,y,\eta)=\sum_{\bar\alpha=0}^{k^n-1}T_{-\bar\alpha}\bigg(
a_m^{-2}(\tilde F^{(n)}(\omega,y,\eta))\cdot\frac{dG^{(n)}(\omega,y)}{dy}\bigg),
\end{equation}
where $\tilde F^{(n)}(\omega)$ is the lifted dynamics defined by \eqref{eq:lifteddynamics}.
Moreover, $b_{0}(\omega)$ satisfies condition $(\mathscr A(-2m))$ and
$b_{1}(\omega)$ satisfies condition $(\mathscr A(-2m-1))$.
\end{thm}

\begin{proof}
In view of \eqref{MonRgeneraltime}--\eqref{MstaronRgeneraltime} we have that
\begin{equation}\label{eq:Egorov1}
B(\omega)=\sum_{\bar\alpha=0}^{k^n-1} T_{-\bar\alpha}
(G^{(n)}(\omega))^\ast \left(
\frac{e^{-i\tau^{(n)}(\omega)}}{\partial E^{(n)}(\omega)}\circ A_m^{-2}\circ
e^{i\tau^{(n)}(\omega)}\right)(E^{(n)}(\omega))^\ast,
\end{equation}
where for example $e^{i\tau^{(n)}(\omega)}$ denotes the operator acting through multiplication
by $\exp{(i\tau^{(n)}(\omega))}$. Here, the operator
$(\partial E^{(n)}(\omega))^{-1}\circ A_m^{-2}$ can be treated by Lemma \ref{lem:multiplicationcomposition},
and the resulting operator can in turn be treated by Lemma \ref{lem:conjugationexponential}.
Finally, by an application of Lemma \ref{lem:conjugationdiffeomorphism}
we thus find that $B(\omega)$ is $\mathbb{P}$-almost surely a semiclassical operator
on $\Sone$ with double symbol $b(\omega)$ in $S^{-2m,0}$,
where
\begin{equation}\label{eq:doublesymboltimen}
b(\omega,x,\xi,x')=\sum_{\bar\alpha=0}^{k^n-1}\mathscr T_{-\bar\alpha}\bigg(
\frac{(dG^{(n)}(\omega,x)/dx)(dG^{(n)}(\omega,x')/dx')}{\psi^{(n)}(\omega,x,x')
a_m^{2}(\xi/\psi^{(n)}(\omega,x,x')+\sigma^{(n)}(\omega,x,x'))}
\bigg)
\end{equation}
is independent of $\xi'$ and satisfies \eqref{eq:doublesymbolperiodicity} and condition $(\mathscr A(-2m,0))$.
Here, $\sigma^{(n)}(\omega)$ is the composition
\[
\sigma^{(n)}(\omega,x,x')=\phi^{(n)}(\omega,G^{(n)}(\omega,x),G^{(n)}(\omega,x')),
\]
and, similarly, the factor $dG^{(n)}(\omega,x)/dx$ is the outcome of composing
the factor $(dE^{(n)}(\omega,x)/dx)^{-1}$ resulting from Lemma \ref{lem:multiplicationcomposition}
with $G^{(n)}(\omega,x)$.

By Lemma \ref{lem:conditionA} it now follows that
$B(\omega)=b_L(\omega,x,hD)$, where
the simplified symbol $b_L(\omega,x,\xi)=b_{0}(\omega,x,\xi)+hb_{1}(\omega,x,\xi)$ in $S^{-2m}(T^\ast(\Sone))$
satisfies condition $(\mathscr A(-2m))$,
and $b_{0}(\omega)$ and $b_{1}(\omega)$ are defined in accordance with Theorem \ref{thm:simplifiedsymbol}.
Moreover, $b_{0}(\omega)$ satisfies condition $(\mathscr A(-2m))$
and $b_{1}(\omega)$ satisfies condition $(\mathscr A(-2m-1))$.

Finally, we check that $b_{0}(\omega)$ is given by \eqref{eq:principalsymbolBgeneraltime}.
By definition we have $b_{0}(\omega,y,\eta)=b(\omega,y,\eta,y)$
where $b(\omega)$ is the double symbol introduced above. By \eqref{eq:psigeneraltime}
we have $\psi^{(n)}(\omega,y,y)=dG^{(n)}(\omega,y)/dy$. Similarly,
by \eqref{eq:phigeneraltime} we have
$\phi^{(n)}(\omega,y,y)=d\tau^{(n)}(\omega,y)/dy$, and using \eqref{notation:E^{(n)}}
and \eqref{eq:Ftildetimensecondargument2} it is straightforward to check that
\begin{equation}
\sigma^{(n)}(\omega,y,y)=\sum_{j=0}^{n-1}\frac{\tau'(\theta^j\omega,G^{(n-j)}(\theta^j\omega ,y))}
{dG^{(j)}(\omega,G^{(n-j)}(\theta^j\omega ,y))/dy},
\end{equation}
where $dG^{(0)}(\omega,\cdot)/dy=1$. In view of \eqref{eq:Ftildetimensecondargument}
it follows that $b_{0}(\omega,y,\eta)$ can be expressed as in \eqref{eq:principalsymbolBgeneraltime},
which completes the proof.
\end{proof}

In the following corollary we adjust our notation to adhere to the conventions used
in Section \ref{section:spectralgap}.

\begin{cor}\label{cor:Egorov'sgeneraltime}
Let $m>0$, and let $A_m$ be the semiclassical operator on $\Sone$
defined as in Theorem \ref{thm:Egorov'sgeneraltime}.
For $n\ge1$, let $P_n(\omega)\equiv P_n(\epsilon;\omega,m)$ be the operator
\[
P_n(\omega)=A_m M_{\nu,n}^\ast(\omega) A_m^{-2}M_{\nu,n}(\omega)A_m.
\]
Then $P_n(\omega)=P_n(\omega,x,hD)$ is $\mathbb{P}$-almost surely a semiclassical operator on $\Sone$ with
symbol $P_n(\omega,y,\eta)$ satisfying condition $(\mathscr A(0))$. Moreover, $P_n(\omega)=p_n(\omega)+hR_n(\omega)$,
where $p_n(\omega)$ satisfies condition $(\mathscr A(0))$
and $R_n(\omega)$ satisfies condition $(\mathscr A(-1))$.
We have
\begin{equation}\label{eq:principalsymbolgeneraltime}
p_{n}(\omega,y,\eta)=\sum_{\bar\alpha=0}^{k^n-1}T_{-\bar\alpha}\bigg(
\frac{a_m^2(\eta)}{a_m^{2}(\tilde F^{(n)}(\omega,y,\eta))}\cdot\frac{dG^{(n)}(\omega,y)}{dy}\bigg),
\end{equation}
where $\tilde F^{(n)}(\omega)$ is the lifted dynamics defined by \eqref{eq:lifteddynamics}.
\end{cor}

\begin{proof}
For $\mathbb{P}$-almost every $\omega$ we may argue as follows:
For $n\ge1$, let $B_n(\omega)=M_{\nu,n}^\ast(\omega) A_m^{-2}M_{\nu,n}(\omega)$
be the semiclassical operator on $\Sone$
given by Theorem \ref{thm:Egorov'sgeneraltime}.
By the same result, the simplified symbol $B_n(\omega,x,\xi)$
of $B_n(\omega)$ belongs to $S^{-2m}(T^\ast(\Sone))$
and satisfies condition $(\mathscr A(-2m))$. Moreover,
\[
B_n(\omega,x,\xi)=b_{n0}(\omega,x,\xi)+hb_{n1}(\omega,x,\xi),
\]
where the principal symbol $b_{n0}(\omega)$ 
satisfies condition $(\mathscr A(-2m))$ and is given by
\eqref{eq:principalsymbolBgeneraltime}, 
while $b_{n1}(\omega)$ satisfies condition $(\mathscr A(-2m-1))$.

Write $P_n(\omega)=A_m \tilde B_n(\omega)$, with $\tilde B_n(\omega)=(b_{n0}(\omega,x,hD)+hb_{n1}(\omega,x,hD))A_m$.
To calculate the symbol of $\tilde B_n(\omega)$ we apply Proposition \ref{prop:composition} individually to
the operators $b_{n0}(\omega,x,hD)A_m$ and $b_{n1}(\omega,x,hD)A_m$, and denote the resulting operators by
$c(\omega,x,hD)$ and $r(\omega,x,hD)$, respectively. For the latter it suffices to apply only the first part of the results of
Proposition \ref{prop:composition},
and conclude that $r(\omega,x,hD)$ is a semiclassical operator on $\Sone$ with symbol $r(\omega)$ satisfying condition $(\mathscr A(-m-1))$.
On the other hand, $c(\omega,x,hD)$ is a semiclassical operator on $\Sone$ with symbol
$c(\omega)\in S^{-m}(T^\ast(\Sone))$ satisfying condition $(\mathscr A(-m))$.
Moreover, $c(\omega)=c_{0}(\omega)+hc_{1}(\omega)$, where $c_{0}(\omega,y,\eta)=b_{n0}(\omega,y,\eta)a_m(\eta)$
satisfies condition $(\mathscr A(-m))$, while $c_{1}(\omega)$ satisfies condition $(\mathscr A(-m-1))$.
Summing up, we conclude that $\tilde B_n(\omega)=c(\omega,x,hD)+hr(\omega,x,hD)$ has symbol
\begin{equation*}
\tilde b_n(\omega)=c(\omega)+hr(\omega)=c_{0}(\omega)+h(c_{1}(\omega)+r(\omega)) \in S^{-m}(T^\ast(\Sone)).
\end{equation*}
Clearly, the sum $c_{1}(\omega)+r(\omega)$ satisfies condition $(\mathscr A(-m-1))$.
In view of the discussion following Definition \ref{def:conditionA} together with the fact that
$c_{0}(\omega)$ satisfies condition $(\mathscr A(-m))$, this implies that $\tilde b_n(\omega)$
satisfies condition $(\mathscr A(-m))$. 
A repetition of these arguments applied to $A_m \tilde B_n(\omega)$
shows that $P_n(\omega,x,hD)$ has
symbol $P_n(\omega,y,\eta)=p_n(\omega,y,\eta)+hR_n(\omega,y,\eta)$ satisfying condition $(\mathscr A(0))$, where
$p_n(\omega)$ satisfies condition $(\mathscr A(0))$ and is given by \eqref{eq:principalsymbolgeneraltime},
while $R_n(\omega)$ satisfies condition $(\mathscr A(-1))$.
This completes the proof.
\end{proof}

\begin{rmk}
It is somewhat treacherous to omit the dependence on $m$ from the notation of the operator $P_n(\omega)$,
especially concerning the fact that the symbol $P_n(\omega,x,\xi)$ satisfies condition $(\mathscr A(0))$.
In particular, inspecting Lemma \ref{lem:conditionA}
we find that the numbers $\epsilon_0(\ell)$ appearing in Definition \ref{def:conditionA}
will by our method of proof depend on $m$, since we have in essence shown that
$P_n(\omega,x,\xi)$ satisfies condition $(\mathscr A(m-2m+m))$. Naturally, the constants
$C_{\ell,n}$ appearing in Definition \ref{def:conditionA}
will also depend on $m$. This is taken into account when we finally prove
Theorem \ref{thm:Pn} below.
\end{rmk}

\subsection{$L^2$ continuity}

Here we shall study the norm of semiclassical operators of the type
treated above, acting on $L^2(\Sone)$. 
We start by briefly discussing 
operators on $L^2(\R)$. Using the identification between semiclassical operators acting on
${\Ci}^\infty(\Sone)$ and semiclassical operators
$p(\omega,x,hD)$, periodic in $x$, acting on smooth periodic functions, we shall then apply these
results to obtain norm estimates for operators $p(w,x,hD):L^2(\Sone)\to L^2(\Sone)$.
For normed vector spaces $X$ and $Y$ we let $\mathscr L(X,Y)$ as before denote the space of bounded linear operators
from $X$ into $Y$ endowed with the operator norm $\lVert\phantom{i}\rVert_{\mathscr L(X,Y)}$. When $X=Y$ we simply write
$\mathscr L(X)=\mathscr L(X,X)$.

For operators on $L^2$ it is beneficial to work with symbol classes $S(m,g)$
more general than $S^0(\R\times\R)$. Here $m(x,\xi)$ is an order function
and $g$ is a Riemannian metric. When $g$ is the Euclidean metric and $m(x,\xi)\equiv 1$
we simply write $S(1)$ for the resulting symbol class, which consists of smooth functions that,
together with all derivatives, are bounded on $\R^2$, that is, $a\in {\Ci}^\infty(\R^2)$ belongs to $S(1)$ if and only if
\[
\lvert\partial_\xi^\alpha D_x^\beta a(x,\xi)\rvert\le C_{\alpha,\beta}.
\]
We allow the symbols in $S(1)$ to depend on $h$ but require the estimates to be uniform for
$0<h\le 1$. Note that $S^0\subset S(1)$.

When $a\in S(1)$, the pseudodifferential operator $a^w(x,D)$ defined by
\[
a^w(x,D)u(x)=\frac{1}{2\pi}\int e^{i(x-y)\xi} a((x+y)/2,\xi)u(y)dyd\xi,
\quad u\in\Es,
\]
is bounded on $L^2(\R)$ with operator norm satisfying
\begin{equation}\label{eq:Weyloperatorbound}
\lVert a^w(x,D)\rVert_{\mathscr L(L^2(\R))}\le C_{\mathrm{dim}}\sum_{\lvert\alpha\rvert\le 
C_{\mathrm{uni}}} \lVert\partial^\alpha a\rVert_{L^\infty(\R^2)},
\end{equation}
where $C_{\mathrm{uni}}$ is a universal constant and $C_{\mathrm{dim}}$ only depends on the dimension $\dim{\R}=1$,
see the first part of Zworski~\cite{Zworski}*{Theorem 4.23}.
The operator $a^w(x,D)$ is referred to as the pseudodifferential Weyl quantization of $a$.
Estimates of the type \eqref{eq:Weyloperatorbound} are also satisfied by
the standard quantization $a(x,D)$; in fact, changing quantization we can
write $a(x,D)=b^w(x,D)$, and the map $a\mapsto b$ defined in this way
is an isomorphism of $S(1)$,
see H{\"o}rmander~\cite{Hormander3}*{Theorem 18.5.10}.
In particular, the semi-norms of $b$ can be estimated by semi-norms of $a$ in $S(1)$,
which proves the claim.
Using a standard scaling argument,
it follows that the semiclassical operator $a(x,hD)$ satisfies
\[
\lVert a(x,hD)\rVert_{\mathscr L(L^2(\R))}\le C_{\mathrm{dim}}\sum_{\lvert\alpha\rvert\le 
C_{\mathrm{uni}}} h^{\lvert\alpha\rvert/2}
\lVert\partial^\alpha a\rVert_{L^\infty(\R^2)}.
\]
For the corresponding statement for the semiclassical Weyl quantization $a^w(x,hD)$ of $a$,
see the second part of Zworski~\cite{Zworski}*{Theorem 4.23}.
Hence, if $a(\omega)\equiv a(\epsilon;\omega,n)$ satisfies condition $(\mathscr A(0))$, then
there is an $\epsilon_0$, depending on $C_{\mathrm{uni}}$ but independent of time $n\in\N^\ast$,
such that if $0\le\epsilon<\epsilon_0$ then
we $\mathbb{P}$-almost surely have
\begin{equation}\label{eq:CalderonVaillancourt}
\lVert a(\omega,x,hD)\rVert_{\mathscr L(L^2(\R))}\le C_{\mathrm{dim}}
\lVert a(\omega)\rVert_{L^\infty(\R^2)}+C_n h^{\frac{1}{2}}
\end{equation}
for $0<h\le 1$, where the constant $C_n$ is independent of $\omega$.
Sharper estimates can be obtained by using Wick
quantization, see Ando and Morimoto~\cite{AndoMorimoto}, or by using
Toeplitz quantization, see Zworski~\cite{Zworski}*{Chapter 13}.
In particular, it is possible to replace $C_{\mathrm{dim}}$ with a factor 1,
and $h^{\frac{1}{2}}$ by $h$. However, this will not needed here; in fact, the bounds we obtain
for the norm of operators on $L^2(\Sone)$ will
introduce an additional factor in front of $\lVert a(\omega)\rVert_{L^\infty}$,
so the presence of $C_{\mathrm{dim}}$ makes little difference,
and also does not in any
way complicate the proof of Theorem \ref{thm:spectralgap} where it appears
in \eqref{eq:choosingnzero1}--\eqref{eq:choosingnzero2}.
Since any refinements of \eqref{eq:CalderonVaillancourt} would need to be reexamined in the presence of
a noise parameter $\omega$, this has therefore been avoided for brevity.

Note that \eqref{eq:CalderonVaillancourt} of course implies the less precise estimate
\begin{equation}\label{eq:CalderonVaillancourtlessprecise}
\lVert a(\omega,x,hD)\rVert_{\mathscr L(L^2(\R))}\le C_n
\end{equation}
for $0<h\le 1$ and some new constant $C_n$. Indeed,
if $a(\omega)$ satisfies condition $(\mathscr A(0))$ then
we can $\mathbb{P}$-almost surely estimate $\lVert a(\omega)\rVert_{L^\infty(\R^2)}$
by a constant independent of $\omega$, which gives \eqref{eq:CalderonVaillancourtlessprecise}.

\begin{thm}\label{L2cont2}
Let $a(\omega)\in S^0(T^\ast(\Sone))$, where $a(\omega)\equiv a(\epsilon;\omega,n)$ in addition is allowed
to depend on a positive parameter $h$.
Suppose that $a(\omega)$ satisfies condition $(\mathscr A(0))$. Then
there is an $\epsilon_0$, independent of time $n\in\N^\ast$, such that for all $0\le\epsilon<\epsilon_0$
and $n\in\N^\ast$,
the semiclassical operator $a(\omega,x,hD)$ is bounded on $L^2(\Sone)$ with norm
\[
\lVert a(\omega,x,hD)\rVert_{\mathscr L(L^2(\Sone))}\le C_{\mathrm{dim}} 
\lVert a(\omega)\rVert_{L^\infty(T^\ast(\Sone))}+C_n h^\frac{1}{2}
\]
for $0<h\le 1$, where the constant $C_n$ is independent of $\omega$, $\mathbb{P}$-almost surely,
and $C_{\mathrm{dim}}$ only depends on the dimension $\dim{\R}=1$.
\end{thm}

\begin{proof}
The proof follows that of Zworski~\cite{Zworski}*{Theorem 5.5}. First, recall that we identify $\Sone$
with the fundamental domain $\T=\{x:0\le x<1\}$ in $\R$.
Elements in $L^2(\Sone)$
are identified with periodic functions on $\R$ belonging to $L^2([0,1])$,
and $a(\omega)$ with a function in
${\Ci}^\infty(\R_x\times\R_\xi)$ which is periodic in $x$. 
Hence
\[
\lVert a(\omega)\rVert_{L^\infty(\R^2)}
=\lVert a(\omega)\rVert_{L^\infty(T^\ast(\Sone))}.
\]
We let $A(\omega)$ denote the operator $A(\omega)=a(\omega,x,hD)$.
Next, use periodicity to write
$A(\omega)u(x)=\sum_{j\in\Z} A_j(\omega) u(x)$ for $x\in\T$, where
\begin{equation}\label{splittingofA}
A_j(\omega) u(x)=\frac{1}{2\pi h}\int_\R\int_\T e^{i(x-y+j)\xi/h}a(\omega,x,\xi)u(y)dyd\xi,\quad x\in\T.
\end{equation}
Let $1_\T$ denote the characteristic function of $\T$, and note that $a(\omega,x,\xi)=a(\omega,x+j,\xi)$
for all $j\in\Z$. By \eqref{splittingofA} it follows that $A_j(\omega) u(x)=(A(\omega)1_\T u)(x+j)$ for $x\in\T$.
With $T_j$ denoting the translation operator $T_j v(x)=v(x-j)$
we thus have
\begin{equation}\label{shapeofAn}
A_j(\omega)=1_\T T_{-j} A(\omega) 1_\T.
\end{equation}

For $x,y\in\T$ and $\lvert j\rvert\ge 2$ we now let $L_j$ be the differential operator
\[
L_j=-\frac{h^2\partial_\xi^2}{\lvert x-y+j\rvert^2},\quad \lvert j\rvert\ge 2,
\]
and note that $L_j^N e^{i(x-y+j)\xi/h}=e^{i(x-y+j)\xi/h}$ for every $N\in\N$.
Also, the coefficient of $L_j$ is smooth, and the (real) transpose ${}^t\! L_j$ of $L_j$
satisfies ${}^t\! L_j=L_j$. Let $\chi\in \Ci_0^\infty(\R)$ be a cutoff function identically equal to 1 near $\T$
and vanishing in a neighborhood of $\R\smallsetminus(-\frac{1}{2},\frac{3}{2})$, and set
\[
\tilde A_j(\omega) u(x)=\frac{1}{2\pi h}\int_\R\int_\R e^{i(x-y)\xi/h}\tilde a_j(\omega,x,y,\xi)u(y)dyd\xi,
\]
where
\begin{equation}\label{auxiliarysymbol}
\tilde a_j(\omega,x,y,\xi)=\chi(x-j)\chi(y) h^{2N}\lvert x-y\rvert^{-2N}\partial_\xi^{2N}a(\omega,x,\xi).
\end{equation}
It follows that
\begin{align*}
1_\T T_{-j} \tilde A_j(\omega) 1_\T u(x)&=
\frac{1}{2\pi h}\int_\R\int_\T e^{i(x-y+j)\xi/h}\chi(x)\chi(y)L_j^N a(\omega,x+j,\xi)u(y)dyd\xi\\
&=\frac{1}{2\pi h}\int_\R\int_\T (L_j^Ne^{i(x-y+j)\xi/h})\chi(x)\chi(y)a(\omega,x+j,\xi)u(y)dyd\xi,
\end{align*}
where the partial integration can be justified by approximation of $a(\omega)$ by functions in $\Es$.
Using the properties of $a(\omega)$, $\chi$ and $L_j$, 
a comparison with \eqref{splittingofA} shows that $A_j(\omega)=1_\T T_{-j} \tilde A_j(\omega) 1_\T$.

In the support of $(x,y)\mapsto\chi(x-j)\chi(y)$,
we have for $\lvert j\rvert\ge 3$ that $\lvert x-y\rvert\sim 1+\lvert j\rvert$, that is,
\[
\frac{1}{C}(1+\lvert j\rvert)\le \lvert x-y\rvert\le C(1+\lvert j\rvert),
\]
where the constant $C$ only depends on the support of $\chi$.
In fact, if $0<\delta<1$ and $\supp\chi\subset(-\delta,1+\delta)$ then it is straightforward to check
that $\lvert x-y\rvert\ge\frac{1-\delta}{2}(1+\lvert j\rvert)$ for all $\lvert j\rvert\ge3$. Since
$\lvert x-y\rvert\le(1+2\delta)(1+\lvert j\rvert)$ in the support of $(x,y)\mapsto\chi(x-j)\chi(y)$,
and $1+2\delta<\frac{2}{1-\delta}$ if $\delta$ is small,
we can take $C=\frac{2}{1-\delta}$. In particular, this implies that
\begin{align*}
\lvert x-y\rvert^{-2N}&\le 2^{2N}(1-\delta)^{-2N}(1+\lvert j\rvert)^{-2N}\\
&=(4(1-\delta)^{-2}(1+\lvert j\rvert)^{-3/2})^N(1+\lvert j\rvert)^{-N/2}.
\end{align*}
Choosing $\delta<1-2^{-\frac{1}{2}}$ we find that $\lvert x-y\rvert^{-2N}\le(1+\lvert j\rvert)^{-N/2}$ for all $\lvert j\rvert\ge3$.
We have similar estimates for the derivatives of $\lvert x-y\rvert^{-2N}$.
Since $a(\omega)$ is assumed to satisfy condition $(\mathscr A(0))$,
Definition \ref{def:conditionA} together with \eqref{auxiliarysymbol} and
\eqref{eq:CalderonVaillancourtlessprecise} applied to $\partial_\xi^{2N}a(\omega)$
then shows that for any $N\in\N$ there is an $\epsilon_0(N)$ such that
if $0\le\epsilon<\epsilon_0$ then
$\lVert\tilde A_j(\omega)\rVert_{\mathscr L(L^2(\R))}\le Ch^{2N}(1+\lvert j\rvert)^{-N/2}$, $\mathbb{P}$-almost surely,
for some new constant $C$ which is independent of $j\in\Z$. 
If $a(\omega)$ depends on time $n\in\N^\ast$, then so might
$C$; however $\epsilon_0$ does not,
see Definition \ref{def:conditionA}.
Since the dimension is one it suffices to choose $N=3$. Since
\begin{align*}
\lVert A_j(\omega) u\rVert_{L^2(\Sone)}&=\lVert1_\T T_{-j}\tilde A_j(\omega)1_\T u\rVert_{L^2(\R)}\\
&\le \lVert\tilde A_j(\omega)\rVert_{\mathscr L(L^2(\R))}\lVert 1_\T u\rVert_{L^2(\R)}
=\lVert\tilde A_j(\omega)\rVert_{\mathscr L(L^2(\R))}\lVert u\rVert_{L^2(\Sone)},
\end{align*}
it follows that, in particular, $\lVert A_j(\omega)\rVert_{\mathscr L(L^2(\Sone))}
=\mathcal O(h^{\frac{1}{2}}(1+\lvert j\rvert)^{-3/2})$
for $\lvert j\rvert\ge3$, $\mathbb{P}$-almost surely. When $\lvert j\rvert=2$ we have $\lvert x-y\rvert>1-2\delta$ in the support of 
$(x,y)\mapsto\chi(x-j)\chi(y)$, so similar arguments show that
if $0\le\epsilon<\epsilon_0$ then
we at least have $\lVert A_{\pm 2}(\omega)\rVert_{\mathscr L(L^2(\Sone))}
=\mathcal O(h^{\frac{1}{2}})$,
$\mathbb{P}$-almost surely.

Now let $\lvert j\rvert<2$. By \eqref{shapeofAn} we have
$\lVert A_j(\omega) u\rVert_{L^2(\Sone)}\le \lVert A(\omega)\rVert_{\mathscr L(L^2(\R))}\lVert u\rVert_{L^2(\Sone)}$.
By hypothesis, $A(\omega)=a(\omega,x,hD)$ where $a(\omega)$ satisfies condition $(\mathscr A(0))$. 
After decreasing $\epsilon_0$ is necessary, we find by virtue of \eqref{eq:CalderonVaillancourt}
that
\[
\lVert A_j(\omega)\rVert_{\mathscr L(L^2(\Sone))}\le C_{\mathrm{dim}}
\lVert a(\omega)\rVert_{L^\infty(\R^2)}+\mathcal O(h^{\frac{1}{2}})
\]
for $j=0,\pm 1$, where the error term is independent of $\omega$, $\mathbb{P}$-almost surely,
and $C_{\mathrm{dim}}$ only depends on the dimension. 
Replacing $C_{\mathrm{dim}}$ in the statement with $3C_{\mathrm{dim}}$,
the result now follows by a summation over $j$.
This completes the proof.
\end{proof}

We end this appendix with a proof of Theorem \ref{thm:Pn},
which essentially contains the results proved thus far.
We also include a proof of Theorem \ref{thm:Q} concerning
operators in the pseudodifferential regime $h=1$, when $\nu\in\Z$
is no longer viewed as a semiclassical parameter.

\begin{proof}[Proof of Theorem \ref{thm:Pn}]\label{appendixproofoftheorem}
Recall that
\[
P_n(\omega)=A_m M_{\nu,n}^\ast(\omega) A_m^{-2}M_{\nu,n}(\omega)A_m,\quad n\ge1.
\]
By Corollary \ref{cor:Egorov'sgeneraltime},
$P_n(\omega)$ is a semiclassical operator on $\Sone$ with
symbol $P_n(\omega,y,\eta)=p_n(\omega,y,\eta)+hR_n(\omega,y,\eta)$,
where $p_n(\omega)$ is given by \eqref{eq:principalsymbolgeneraltime} and satisfies condition $(\mathscr A(0))$,
while $R_n(\omega)$ satisfies condition $(\mathscr A(-1))$.
Note that $p_n(\omega)$ and $R_n(\omega)$ depend on $m$ by construction although this is not showcased in the notation,
see the remark following Corollary \ref{cor:Egorov'sgeneraltime}.
By Theorem \ref{L2cont2} we can find $\epsilon_0$ (depending on $m$ but 
independent of time $n\in\N^\ast$) such that if $0\le\epsilon<\epsilon_0$ then
\[
\lVert p_n(\omega,x,hD)\rVert_{\mathscr L(L^2(\Sone))}
\le C_{\mathrm{dim}}\lVert p_n(\omega)\rVert_{L^\infty(T^\ast(\Sone))}+\mathcal O_{n,m}(h^{\frac{1}{2}})
\]
for all $n\ge1$,
where the error term is independent of $\omega$, $\mathbb{P}$-almost surely,
and $C_{\mathrm{dim}}$ only depends on the dimension.
The same is true with $p_n(\omega)$ replaced by $R_n(\omega)$.
According to Definition \ref{def:conditionA} we also 
have, for sufficiently small $\epsilon$, that
$\lVert R_n(\omega)\rVert_{L^\infty(T^\ast(\Sone))}\le C_{n,m}$ for all $n\ge1$, where the constant $C_{n,m}$
is independent of $\omega$. This gives
\[
\lVert P_n(\omega)\rVert_{\mathscr L(L^2(\Sone))}\le
C_{\mathrm{dim}}\lVert p_n(\omega)\rVert_{L^\infty(T^\ast(\Sone))}+\mathcal O_{n,m}(h^{\frac{1}{2}}),
\quad n\ge1.
\]
The only thing that remains is to check that the expressions \eqref{eq:principalsymbolgeneraltime}
and \eqref{principalsymbolP} for the principal symbol coincide, but this follows from
the discussion surrounding \eqref{eq:setequality}.
The proof is complete.
\end{proof}

In the following proof we will treat the case $h=1$ and allow $\nu\in\Z$ to be arbitrary, fixed and unrelated to $h$.
Note also that time $n=1$ here.

\begin{proof}[Proof of Theorem \ref{thm:Q}]
By \eqref{eq:definitionofQnew} we have that
\[
(Q_\nu(\omega))^\ast Q_\nu(\omega)=\langle D\rangle^{m}M_\nu^\ast(\omega)\langle D\rangle^{-2m}M_\nu(\omega)\langle D\rangle^{m}
\]
since $\langle D\rangle^{-m}\langle D\rangle^{-m}=\langle D\rangle^{-2m}$.
Introduce $B_\nu(\omega)=M_\nu^\ast(\omega)\langle D\rangle^{-2m}M_\nu(\omega)$.
By \eqref{restrictionoperator} and \eqref{adjointrestrictionoperator} it follows that
$B_\nu(\omega)$ can (in the sense of this appendix) be identified with the operator
\[
B_\nu(\omega)u(x)=\sum_{j=0}^{k-1}T_{-j}(E^{-1}(\omega))^\ast\bigg(\frac{e^{-i\nu\tau(\omega)}}{E'(\omega)}
\langle D\rangle^{-2m}\Big(e^{i\nu\tau(\omega)}u\circ E(\omega) \Big)\bigg)(x) 
\]
for $u\in\Ci_p^\infty(\R)$, where $E(\omega)$ and $\tau(\omega)$
are the functions on $\R$ described in the beginning of this appendix
(compare with \eqref{eq:Egorov1} where we had $h=1/\nu$). Now note that when $\nu\in\Z$
is fixed, the operator acting through multiplication by $x\mapsto e^{\pm i\nu\tau(\omega,x)}$
is a pseudodifferential operator of order 0 with symbol $e^{\pm i\nu\tau(\omega)}$.
By \eqref{eq:assumptionCrclose}, this symbol satisfies condition $(\mathscr A(0))$
with the number $\epsilon_0$ in Definition \ref{def:conditionA} independent of $\nu$.
By \eqref{eq:assumptionCrclose}, the symbol $1/E'(\omega)$ also satisfies condition $(\mathscr A(0))$.
Since $\xi\mapsto \langle \xi\rangle^{-2m}$ clearly satisfies condition $(\mathscr A(-2m))$, 
we find by applying Proposition \ref{prop:composition} three times
(all in the case when $h=1$) that
\[
B_\nu(\omega)=\sum_{j=0}^{k-1}T_{-j}(E^{-1}(\omega))^\ast
A_\nu(\omega) E(\omega)^\ast,
\]
where the symbol of $A_\nu(\omega)$ satisfies condition $(\mathscr A(-2m))$.
By Lemmas \ref{lem:conjugationdiffeomorphism} and \ref{lem:conditionA} it follows that
the same is true for the symbol of $B_\nu(\omega)$,
so by finally applying Proposition \ref{prop:composition}
twice we find that the symbol $q_{\nu,m}(\omega)$ of 
$(Q_\nu(\omega))^\ast Q_\nu(\omega)=\langle D\rangle^{m}B_\nu(\omega)\langle D\rangle^{m}$
satisfies condition $(\mathscr A(0))$. In particular, we can find $\epsilon_0$ independent of $\nu$
(but depending on $m$ in view of the remark following Corollary \ref{cor:Egorov'sgeneraltime})
such that if $0\le\epsilon<\epsilon_0$ then
$\lVert q_{\nu,m}(\omega)\rVert_{L^\infty(T^\ast(\Sone))}\le C_{\nu,m}$ for all $\nu$,
$\mathbb{P}$-almost surely, where the constant depends on $\nu$ and $m$ but not on $\omega$.
Using this together with an application of Theorem \ref{L2cont2}, we find that for sufficiently small $\epsilon$ we have
\[
\lVert(Q_\nu(\omega))^\ast Q_\nu(\omega)\rVert_{\mathscr L(L^2(\Sone))}\le \tilde C_{\nu,m},\quad\text{$\mathbb{P}$-almost surely},
\]
where the constant is independent of $\omega$. This completes the proof.
\end{proof}

\begin{rmk}
With only minor changes, the previous proof would yield a formula for the double
symbol of $B_\nu(\omega)$, which could then be used to calculate the principal symbol of 
$(Q_\nu(\omega))^\ast Q_\nu(\omega)$, as in the proof of Theorem \ref{thm:Egorov'sgeneraltime}
and Corollary \ref{cor:Egorov'sgeneraltime}. 
By using this explicit formula and essentially
repeating the proof of Faure~\cite{Faure}*{Theorem 2}, one would obtain an alternative proof
of Proposition \ref{prop:discretespectrum2}.
\end{rmk}

\section*{Acknowledgements}
We would like to express our profound gratitude to
Fr{\'e}d{\'e}ric Faure, Masato Tsujii, Yoshinori Morimoto, Nils Dencker and Shigehiro Ushiki for many fruitful discussions and valuable comments.
We are also deeply grateful to an anonymous reviewer for many suggestions, all of which
substantially improved the paper.
The research of Yushi Nakano was supported in part by JSPS Kakenhi Grant No.~11J01842, Japan Society for the Promotion of Science.
The research of Jens Wittsten was supported in part by JSPS Kakenhi Grant No.~24-02782, Japan Society for the Promotion of Science.


\bibliography{NakanoWittsten}

\end{document}